\documentclass[a4paper,11pt,leqno]{amsart}
\usepackage[applemac]{inputenc}
\usepackage{epsfig,graphics}
\usepackage[all,knot]{xy}
\usepackage{amsthm,amssymb,amsbsy,bbm,a4wide,verbatim,url}
\usepackage{rotating}
\xyoption{knot}

\def\Z{\mathbbm{Z}}
\def\C{\mathbbm{C}}
\def\Q{\mathbbm{Q}}
\def\un{\mathbbm{1}}

\def\ii{\mathrm{i}}

\def\End{\mathrm{End}}

\def\Aut{\mathrm{Aut}}
\def\Ker{\mathrm{Ker}}

\def\Hom{\mathrm{Hom}}

\def\g{\mathfrak{g}}

\def\la{\lambda}

\def\sl{\mathfrak{sl}}
\def\so{\mathfrak{so}}
\def\sp{\mathfrak{sp}}

\def\eps{\varepsilon}
\def\into{\hookrightarrow}
\def\onto{\twoheadrightarrow}

\def\WBMW{\widetilde{BMW}}

\newtheorem{theor}[equation]{Theorem}
\newtheorem{lemma}[equation]{Lemma}
\newtheorem{prop}[equation]{Proposition}
\newtheorem{cor}[equation]{Corollary}
\newtheorem{conj}[equation]{Conjecture}

\newtheorem{remark}[equation]{Remark}
\newtheorem{defi}[equation]{Definition}
\numberwithin{equation}{section}

\usepackage{tikz}
\newcounter{braid}
\newcounter{strands}
\pgfkeyssetvalue{/tikz/braid height}{1cm}
\pgfkeyssetvalue{/tikz/braid width}{1cm}
\pgfkeyssetvalue{/tikz/braid start}{(0,0)}
\pgfkeyssetvalue{/tikz/braid colour}{black}
\pgfkeys{/tikz/strands/.code={\setcounter{strands}{#1}}}

\makeatletter
\def\cross{%
  \@ifnextchar^{\message{Got sup}\cross@sup}{\cross@sub}}

\def\cross@sup^#1_#2{\render@cross{#2}{#1}}

\def\cross@sub_#1{\@ifnextchar^{\cross@@sub{#1}}{\render@cross{#1}{1}}}

\def\cross@@sub#1^#2{\render@cross{#1}{#2}}

\def\render@cross#1#2{
  \def\strand{#1}
  \def\crossing{#2}
  \pgfmathsetmacro{\cross@y}{-\value{braid}*\braid@h}
  \pgfmathtruncatemacro{\nextstrand}{#1+1}
  \foreach \thread in {1,...,\value{strands}}
  {
    \pgfmathsetmacro{\strand@x}{\thread * \braid@w}
    \ifnum\thread=\strand
    \pgfmathsetmacro{\over@x}{\strand * \braid@w + .5*(1 - \crossing) * \braid@w}
    \pgfmathsetmacro{\under@x}{\strand * \braid@w + .5*(1 + \crossing) * \braid@w}
    \draw[braid] \pgfkeysvalueof{/tikz/braid start} +(\under@x pt,\cross@y pt) to[out=-90,in=90] +(\over@x pt,\cross@y pt -\braid@h);
    \draw[braid] \pgfkeysvalueof{/tikz/braid start} +(\over@x pt,\cross@y pt) to[out=-90,in=90] +(\under@x pt,\cross@y pt -\braid@h);
    \else
    \ifnum\thread=\nextstrand
    \else
     \draw[braid] \pgfkeysvalueof{/tikz/braid start} ++(\strand@x pt,\cross@y pt) -- ++(0,-\braid@h);
    \fi
   \fi
  }
  \stepcounter{braid}
}

\tikzset{braid/.style={double=\pgfkeysvalueof{/tikz/braid colour},double distance=1pt,line width=2pt,white}}

\newcommand{\braid}[2][]{%
  \begingroup
  \pgfkeys{/tikz/strands=2}
  \tikzset{#1}
  \pgfkeysgetvalue{/tikz/braid width}{\braid@w}
  \pgfkeysgetvalue{/tikz/braid height}{\braid@h}
  \setcounter{braid}{0}
  \let\dsigma=\cross
  #2
  \endgroup
}
\makeatother
\title[Markov traces on the BMW-algebras]{Markov traces on the Birman-Wenzl-Murakami algebras}

\author{{\large I\lowercase{van} M\lowercase{arin} \& E\lowercase{mmanuel} W\lowercase{agner}} }
\date{July 17, 2014}

\begin{document}

\setcounter{tocdepth}{2}
\maketitle

\begin{center}
LAMFA \\
Universit\'e de Picardie Jules-Verne \\
33 rue Saint Leu \\
80039 Amiens Cedex 1\\
France\\
\url{ivan.marin@u-picardie.fr}\\
--- \\
Institut de Math\'ematiques de Bourgogne UMR 5584 \\
Universit\'e de Bourgogne\\
7 Avenue Alain Savary BP 47870\\
21078 Dijon Cedex\\
France\\
\url{emmanuel.wagner@u-bourgogne.fr}
\end{center}
\bigskip

\begin{abstract}
We classify the Markov traces factoring through the Birman-Wenzl-Murakami (BMW) algebras. For
this purpose, we define a common `cover' for the two variations of the BMW-algebra
originating from the quantum orthogonal/symplectic duality, which are responsible for the so-called
`Dubrovnik' variation of the Kauffman polynomial.
For generic values of the defining parameters
of the BMW algebra, this cover is isomorphic to the BMW algebra itself, and this fact provides
a shorter defining relation for it, in the generic case. For a certain 1-dimensional family of special
values however, it is a non-trivial central extension of the BMW-algebra which induces a central extension of the Temperley-Lieb
algebra. Inside this 1-dimensional family, exactly two values provide possibly additional $\Q$-valued Markov traces.
We describe both of these potential traces on the (extended) Temperley-Lieb subalgebra. While we only conjecture
the existence of one of them, we prove the existence of the other by introducing a central extension of the Iwahori-Hecke algebra
at $q=-1$ for an \emph{arbitrary} Coxeter system, and by proving that this extension indeed admits an exotic Markov trace.
These constructions provide several natural non-vanishing Hochschild cohomology classes on these classical algebras.
\end{abstract}

\tableofcontents

\section{Introduction}

\subsection{Context}
When the Jones and Homfly polynomial appeared, they were first  described as Markov traces on the
group algebra of the braid group, and therefore as invariants of oriented links produced
by the virtue of Markov theorem. Indeed  this theorem says that an invariant of oriented
links is the same thing as a Markov trace on the tower of the group algebra of the braid group.
This approach has had a great descent, on the existence of quantum invariants originating from
a quantum trace, in the study of Markov traces on generalized
Hecke algebras, and the Khovanov homology in more recent years has a description in terms
of Soergel bimodules that can be seen as part of the same thread of thinking.

On the other hand, the construction of its cousin the Kauffman polynomial has a very different story. Indeed
this polynomial was first described as an invariant of regular isotopy by way
of \emph{un}oriented skein relations, and then turned into an invariant of (oriented) links by twisting the invariant of regular
isotopy of the underlying unoriented link by the writhe of the link. This approach was shown to be
also applicable to the Jones polynomial (and later on used in the definition of Khovanov homology) through the Kauffman bracket, whose description involves unoriented skein relations.

Conversely, there has been very soon an attempt to describe the Kauffman polynomial in terms
of Markov traces. This lead to the definition, independently by Murakami and Birman-Wenzl,
of the BMW-algebra. This algebra however, in spite of the quantum moto describing both the Hecke and BMW algebra
as `quantizations' of the symmetric group and of the Brauer algebra respectively, and
therefore as quantizations of the centralizer algebras of the tensor powers of the standard modules of the
most classical Lie algebras $\sl_n$ and $\so_n/\sp_n$, never reached the level of recognition of the
Hecke algebra.
As a sample test for a not so nice behavior, a completely satisfactory `categorification' is still lacking,
and the known generalizations of this algebra to other Coxeter or complex reflection groups have flatness issues (as modules
over the natural ring attached to this context).
 The most natural description of
the BMW-algebra is probably as an algebra of unoriented tangle diagrams, following the original
work of Kauffman. Then, $BMW_n$ can be described as an homomorphic image of the group
algebra of the braid group, by converting a given braid into a non-oriented tangle. 
Describing the kernel of this map as a finite set of meaningful relations presenting $BMW_n$ as a quotient
is the next natural goal.
Following the insight given by the tangle picture, relations on 2 strands are easily found to be given by a cubic
relation $(s_i-a)(s_i-b)(s_i-c)$ on the Artin generators, hence $BMW_n$ belongs to the
family of cubic quotients of the group algebra of the braid group $B_n$. In the
sequel we will denote $x = b+c$ and $y = bc$ since most definitions involving the BMW-algebras are symmetric
in $b,c$ (but not in $a,b,c$ !).  We will also assume $x \neq 0$ which is necessary in order to express the elementary tangles $e_i$ as linear
combinations of braids. As opposed to the quadratic quotients, for which the Hecke
algebra provides a universal finite-dimensional model, the `cubic Hecke algebra' does not
enjoy universal finiteness properties, because the factor group $B_n/s_i^3$ is infinite for $n \geq 6$. There are
thus additional relations on 3 strands, and it turns out that these relations on 3 strands are enough to define
the $BMW_n$ algebras for every single $n$.

\subsection{Presentation}

The origin of this work was the apparently innocent question to determine whether the `Kauffman trace', that is the
Markov trace affording the Kauffman polynomial, was indeed the only Markov
trace factoring through the BMW-algebra (in addition to the one factoring
to the smaller Hecke algebra quotient). When we tried to answer this question, we faced
another problem. First of all, it is well-known that there are two variations of
the Kauffman polynomial, one of them having been given the name of the
city of Dubrovnik. These two variations can be seen as Markov traces,
after specialization of the parameters $a,b,c$, subject to a polynomial relation
$a = bc$ or $a = -bc$. But they are not Markov trace on the same BMW-algebra ! More
precisely, although the two relevant BMW-algebras are isomorphic as algebras, they
are a priori not the same quotient of the group algebra of the braid group. Indeed, the additional
relation on 3 strands involves $a,b,c$ and depends on whether $a=bc$ or $a = -bc$.
Viewing the BMW-algebras as centralizer algebras for quantum groups actions, these two
specializations correspond to the difference between the orthogonal and symplectic
groups (see the end of \S \ref{secttwobmwagebras}). 

Because of this, we tried to dig deeper into the defining relations of the BMW-algebras. When
translated into braid words, the usual additional relation (in both cases), relating $e_i s_j e_i $ and $e_i$
for $|i-j| = 1$, involves 12 terms. We use another relation, that relates $s_j^{-1} e_i s_j^{-1}$ and $s_i e_j s_i$, involving only 6 terms when expanded into braid words, which holds inside both BMW-algebras, and which is enough to ensure
the finiteness of the dimension. Because of this, we can define
a finite-dimensional `cover' of the two BMW-algebras involved in the computation of the Kauffman polynomial,
as well as the ortho-symplectic quantum invariants, that we
call $\widetilde{BMW}_n$.

From this we get a satisfying algebraic setup to explore the possible Markov traces for every single value
of the parameters. We prove that, for generic values of the parameters, this cover is actually isomorphic to
the usual BMW-algebras, thus providing in these cases an even simpler definition (as a quotient of the group algebra of $B_n$) of the BMW-algebras 
(see propositions \ref{propisomBMWtilde} and \ref{propisomBMWtilde2}). However, for generic
values satifying $a^2 = y$, we find that $\widetilde{BMW}_3$ has one dimension more than expected.

Concerning Markov traces we first get (proposition \ref{prop:MT3bmw}) that, when $a^2 \neq y^2$ and $a^2 \neq y$ (and actually : also when $a^2 = y$ and $a^2 \neq y^2$,
see proposition \ref{propmarkovypas1}), the only Markov trace factoring through $\widetilde{BMW}_n$ is the Ocneanu trace, defined
on the Hecke algebra quotient. 

We specialize to the situation $a^2 = y^2$. Then,
we have in addition the Kauffman trace.
When $y \neq 1$, we prove that the only Markov traces factoring through our algebra lead either to the Kauffman polynomial or to the Homfly polynomial,
and that this algebra is actually isomorphic to the usual BMW-algebra. When $y = 1$, that is when $a^2 = y^2$ and $a^2 = y$,
we have first to exclude a very degenerate case, $x = -2a$, for which there is an infinite number of Markov traces,
namely the ones factoring through the group algebra of the symmetric group (see proposition \ref{propTracesXm2a}) ; it is well-known that
these ones detect only the number of components of the links. In the general case $y=1$,
we get an additional Markov trace $t_n^{\dagger \dagger}$ that basically provides the parity of the number of
components of the link. These 3 Markov traces exhaust all possible traces, and are linearly independent one
from the other, for generic $x$. The special values for which this does not hold, besides $x = -2a$, are $x=a$ and $x=2a$.

In order to understand what happens in these two special cases, we provide a description of the algebra
$\widetilde{BMW}_n$ when specialized at $a^2=y=1$. For this we define by generators and relations an algebra over $\Q[a,x,x^{-1}]/(a^2-1)$ that we denote $F_n$.
It is a free module of rank 1 more than the dimension of $BMW_n$ (corollary \ref{cordimFn}) and it can be viewed
as a central extension of $BMW_n$ by a 1-dimensional ideal spanned by some element that we call $C$. This element squares to 0, and
therefore the extension cannot split, precisely when $x=a$ and $x=2a$. We prove that it is indeed the specialization
we want of $\widetilde{BMW}_n$, as soon as $x \neq -2a$ (theorem \ref{theoisomddaggerFn}),
while the specialization of $\widetilde{BMW}_n$ for the case $x=-2a$ provides a larger algebra, of which we provide a
partly conjectural description (see section \ref{subsectm2a}). Finally, we use the structure of $F_n$ to check that the space of Markov traces factoring through $F_n$ has dimension at most 3 : there is at most one way to find
an additional Markov trace in the special cases $x=a$ and $x=2a$.

In passing, we deduce from the existence of $F_n$ a similar central extension $\widetilde{TL}_n$ of the Temperley-Lieb algebra,
which is a subalgebra of $F_n$. We find natural (diagrammatic) interpretations of the two (potential) additional Markovs trace when restricted to this
subalgebra. Finally, we manage to construct the expected additional Markov trace in the case $x=2a$ by constructing a central
extension of the classical Hecke algebra.

More generally, we prove that, for an arbitrary Coxeter system $(W,S)$, the usual Iwahori-Hecke algebra at $q=-1$
has a natural non-split central extension, of dimension $1+\# W$ if $W$ is finite, and that there existes a Markov trace on this
algebra when $W= \mathfrak{S}_n$ (see theorems \ref{theoHeckeExt} and \ref{theoexistencetracedeltanul}). This Markov trace provides what we need.
The case $x= a$ remains conjectural, although we are confident that the corresponding invariant exists. We guess that an algebraic proof
of the existence of the Kauffman trace similar to the one that Jones provided for the Ocneanu trace should be easy to
generalize to our central extension. However it appears that no one provided such a proof yet, and finding such a proof seems to us to be quite
more tricky than the Hecke algebra case.

\subsection{Organisation of the paper}

The plan of the paper is as follows. In \S 2 we compile a few results on the `cubic Hecke algebra' on 3 strands, namely
the quotient of the group algebra of $B_3$ by a generic cubic relations $(s_i -a)(s_i-b)(s_i-c) = 0$. The finite-dimensionality
as well as the symmetric algebra structure of this algebra is a crucial tool in the sequel. In \S 3 and \S 4 we explore the algebraic
structure of the BMW-algebra, and define a suitable cover of its two avatars appearing in the construction of the
Kauffman polynomial. This provides a suitable setting for studying the Markov traces, and we do this in \S 5. 
Inside \S 5, we rediscover the classical Markov traces, and describe an additional one when $y=a^2 =1$. We prove there
that these exhaust all possible traces, except when $x=a$ or $x=2a$. In \S 6 we introduce our central extensions of
the BMW-algebra, Hecke algebras and Temperley-Lieb algebras. We define two additional traces on these
extended Temperley-Lieb algebra, and one on the extended Hecke algebra.  Finally, \S 7 is devoted to the
exploration of the link invariants obtained in this way. We show that the additional trace for $y=a^2=1$ simply counts
the parity of the number of components of the link, and we tabulate the two special ones, for which an interpretation is
lacking.

Because of the large number of specializations that we are using, we provide here a table of the various rings involved in
the paper, as a common place for reference. The second table provide a list of the main algebras used in the paper, together
with the rings involved in their definition.
{}
$$
\begin{array}{|l|l|l|}
\hline
R = \Q[a,b,c,(abc)^{-1}] 
 & \overline{R} = R/(a^2-y^2) & R_{\pm} = R/(a \mp y) \\
\hline
S = R[(b+c)^{-1}] & \overline{S} = S/(a^2-y^2) & S_{\pm} = S/(a \mp y) \\
\hline
 \overline{S}' = \overline{S}[(bc-1)^{-1}] & \overline{S}'_{\pm} = \overline{S}'/(a \mp y)  & S^{\dagger} = S/(a^2 -y) \\
\hline
S^{\dagger\dagger} = S^{\dagger}/(a^2-1) & S^{\dagger\dagger}_{\pm} = S^{\dagger}/(a\mp1)&  \\
\hline
\end{array}
\ \ \ \begin{array}{|c|c|}
\hline 
\mbox{Algebra} & \mbox{Ring} \\
\hline 
H_n & R  \\ 
\WBMW_n & R \\
BMW_n^{\pm} & S_{\pm} \\
\overline{BMW}_n & \overline{R} \\
\hline
\end{array}
$$

{\bf Acknowledgements.} We thank S. Bouc, F. Digne and A. Zimmermann for discussions and references.

\section{Preliminaries on the cubic Hecke algebras on 3 strands}
\label{sectH3}
In order to insure the coherence of notations with the forthcoming sections,
we let 
$$
R = \Q[a,a^{-1},b,b^{-1},c,c^{-1}] = \Q[a,b,c,(abc)^{-1}],$$
although all the results of the present section are actually already valid with $R=\Z[a,b,c,(abc)^{- 1}]$.

 We let $H_n$ denote
the $R$-algebra defined as the quotient of the group algebra $R B_n$
of the braid group on $n$ strands by the relations $(s_i-a)(s_i-b)(s_i-c) = 0$
for $1 \leq i \leq n-1$ or, equivalently -- since each $s_i$ is conjugated to $s_1$ --
by the relation $(s_1-a)(s_1-b)(s_1-c)= 0$. It is known that $H_n$ is a
free $R$-module of finite rank for $n \leq 5$ (see \cite{CUBIC5}).
More precisely, for $n =3$, one may excerpt from \cite{CUBIC5} the
following result (see also \cite{BROUEMALLE, FUNAR,THESE} for related statements).

\begin{prop} {\ } \label{prop:H3libre}
\begin{enumerate}
\item The algebra $H_3$ is a free $H_2$-module of rank $8$, with basis the elements $1, s_2,s_2^{-1}$, $s_1^{\alpha} s_2^{\beta}$ for $\alpha,\beta \in \{ 1,-1 \}$,
$s_2 s_1^{-1} s_2$.
\item The algebra $H_3$ is a free $R$-module of rank 24, with basis the
elements 
$$
\begin{array}{lcl}
\mathcal{B}_1&=& (1,  s_1, s_1^{-1}, s_2, s_2^{-1}, s_1s_2, s_1s_2^{-1}, s_1^{-1}s_2, s_1^{-1}s_2^{-1}, s_1s_2s_1, s_1s_2s_1^{-1}, s_1^{-1}s_2s_1, 
  s_1^{-1}s_2s_1^{-1},\\ & & s_1s_2^{-1}s_1, s_1^{-1}s_2^{-1}s_1, s_2s_1, s_2^{-1}s_1, s_2s_1^{-1}, s_2^{-1}s_1^{-1}, s_1s_2^{-1}s_1^{-1}, 
  s_1^{-1}s_2^{-1}s_1^{-1}, s_2s_1^{-1}s_2,\\ & & s_1s_2s_1^{-1}s_2, s_1^{-1}s_2s_1^{-1}s_2 ).
  \end{array}
$$ {}
\end{enumerate}
\end{prop}
\begin{proof}  From \cite{CUBIC5} theorem 3.2 we know that $H_3$ is generated as a $H_2$-module by the $8$ elements on the first statement. Since $H_2$
is spanned by $1,s_1, s_1^{-1}$ it follows that $H_3$ is generated as a $H_3$-module by the $24$ elements of the
second statement. Since $\Gamma_3$ has 24 elements and by an a argument of \cite{BMR} (see also \cite{CYCLO}, proposition 2.4 (1)) it follows that these $24$ elements are 
a basis over $R$ of $H_3$. It readily follows that the $8$ original elements provide a basis of $H_3$ as a $H_2$-module.
\end{proof}

A consequence is that $H_3$ is a free deformation
of the group algebra $R \Gamma_3$,
where $\Gamma_n$ denotes the quotient of the braid group by the relations $s_i^3 = 1$,
and $H_3$ becomes isomorphic to it after extension of scalars to the algebraic
closure $\overline{K}$ of the field of fractions $K$ of $R$. Actually, one
has the stronger result $H_3 \otimes_R K \simeq K \Gamma_3$,
because the irreducible representations of $K H_3$ 
are absolutely irreducible.

We will use the following explicit matrix models for the representations, which
are basically the same which were obtained in \cite{BROUEMALLE}, \S 5B. We endow $\{a,b,c \}$
with the total order $a<b<c$. We denote
\begin{enumerate}
\item $S_{\alpha}$ for $\alpha \in \{a,b,c\}$ the 1-dimensional representation
$s_1,s_2 \mapsto \alpha$ 
\item $U_{\alpha,\beta}$ for $\alpha , \beta \in  \{a,b,c\}$ with $\alpha < \beta$ the 2-dimensional representation
$$
U_{\alpha,\beta} : 
s_1 \mapsto \left( \begin {array}{cc} \alpha&0\\ \noalign{\medskip}-\alpha&\beta\end {array} \right)
s_2 \mapsto \left( \begin {array}{cc} \beta&\beta\\ \noalign{\medskip}0&\alpha\end {array} \right)
$$
\item $V$ the $3$-dimensional irreducible representation 
$$
s_1 \mapsto \left( \begin {array}{ccc} c&0&0\\ \noalign{\medskip}ac+{b}^{2}&b&0\\ \noalign{\medskip}b&1&a\end {array} \right)
s_2 \mapsto \left( \begin {array}{ccc} a&-1&b\\ \noalign{\medskip}0&b&-ac-{b}^{2}\\ \noalign{\medskip}0&0&c\end {array} \right) 
$$
\end{enumerate}
We note the important feature that these representations are actually defined over $R$.
As a consequence, these formulas provide an explicit embedding
$$
\Phi_{H_3} : H_3 \into R^3 \oplus M_2(R)^3 \oplus M_3(R) 
$$
and the RHS is easy to identify with $R^{24}$ as a $R$-module.

For an algebra $A$, we let $[A,A]$ denote the submodule spanned by the $ab-ba$ for $a,b \in A$.

\begin{prop} \label{prop:H3AB}
$H_3/[H_3,H_3]$ is freely generated over $R$ by the $7$ elements
    $$
    1,s_1,s_1^{-1},s_1s_2,s_1s_2^{-1}, s_1^{-1} s_2^{-1}, s_1^{-1} s_2 s_1^{-1}s_2.
    $$
\end{prop}
\begin{proof}
From the basis above it is readily proved (note for instance that
    the element $s_1s_2s_1$ conjugates $s_1$ and $s_2$) that
    $H_3/[H_3,H_3]$ is generated by these $7$ elements. The freeness is
    a consequence of the fact that, if $\tilde{H}_3 = H_3 \otimes_R K$, then $\tilde{H}_3/[\tilde{H}_3,\tilde{H}_3] \simeq Z(\tilde{H}_3)
     \simeq Z(K \Gamma_3)$ has dimension $7$. Since these $7$ elements also generate $\tilde{H}_3/[\tilde{H}_3,\tilde{H}_3]$
     they are a basis of the $7$-dimensional $K$-vector space $\tilde{H}_3/[\tilde{H}_3,\tilde{H}_3]$. If a $R$-linear combination of these elements
     belonged to $[H_3,H_3]$, this would yield a contradiction since $H_3  \subset \tilde{H}_3$.
\end{proof}

\begin{prop} {\ }
\begin{enumerate} 
\item The family $\mathcal{B}_0$ below provides a basis of $H_3$ as a $R$-module.
$$
\mathcal{B}_0 = \begin{array}{l}
 1, s_1, s_1^2,s_2, s_2^2, s_1s_2, s_1s_2^2, s_1^2s_2, 
  s_1^2s_2^2, s_1s_2s_1, s_1s_2s_1^2, s_1^2s_2s_1, 
  s_1^2s_2s_1^2, s_1s_2^2s_1, s_1^2s_2^2s_1, s_2s_1, \\
  s_2^2s_1, s_2s_1^2, s_2^2s_1^2, s_1s_2^2s_1^2, 
  s_1^2s_2^2s_1^2, s_2s_1^2s_2, s_1s_2s_1^2s_2, 
  s_1^2s_2s_1^2s_2 
  \end{array}
$$
\item The linear form $t_0 : H_3 \to R$ defined by $t_0(1) = 1$ and $t_0(g) = 0$ for all $g \in \mathcal{B}_0 \setminus \{ 1 \}$
is a (nondegenerate) symmetrizing form for $H_3$.
\end{enumerate}
\end{prop}
\begin{proof}
It was checked in \cite{BROUEMALLE} \S 4B that the linear span of $\mathcal{B}_0$ (actually of the image of $\mathcal{B}_0$
under the anti-automorphism $s_i \mapsto s_i$, see \cite{CUBIC5})
was stable
under left multiplication by $s_1, s_2$, hence generates $H_3$, and thus provides
a basis of $H_3$. It appears however that the computations in \cite{BROUEMALLE} might have (incorrectly) been made inside $H_3\otimes K$
instead of $H_3$, so we need to provide another argument. Since we proved that $\mathcal{B}_1$ is a basis of $H_3$,
it is sufficient to show that every element of $\mathcal{B}_1$ can be expressed as a linear combination of
the elements of $\mathcal{B}_0$. This readily follows from the expression of $s_i^{-1}$
as a linear combination of $1,s_i$ and $s_i^2$. This proves (1). Using the explicit injective morphism $\Phi : H_3 \to
R^3 \times M_2(R)^3 \times M_3(R)$ above, calculations inside $H_3$ are easy, and 
we can explicitely check that $(x,y) \mapsto t_0(xy)$ is indeed a symmetrizing trace, more precisely that
the matrix $t_0(b_i b_j)$ for $b_i \in \mathcal{B}_0$ (resp. $\mathcal{B}_1$) is a symmetric matrix with determinant $-(abc)^{54}$ (resp. $-(abc)^2$), which belongs to $R^{\times}$.
\end{proof}

Because of the proposition, it is possible to apply the theory of Geck's `Schur elements' (see e.g. \cite{GECKPFEIFFER}) to $H_3$, that is
to determine elements $p_{\chi} \in R$ attached to each of the irreducible representations $\chi$ of $H_3 \otimes K$ such that
$t_0 = \sum_{\chi} \frac{1}{p_{\chi}} tr_{\chi}$ where $tr_{\chi}$ denotes the matrix trace attached to the irreductible
representation $\chi$ of $H_3 \otimes K$. Once convenient matrix models as well as an explicit description
of $t_0$ are known, it is a simple matter to determine them. These elements were already
determined in \cite{MALLE}.

The one attached to $S_a$
is
$$
p_{S_a} = {\frac { \left( a-c \right)  \left( {a}^{2}-ac+{c}^{2} \right)  \left( a-b \right)  \left( {a}^{2}-ab+{b}^{2} \right) 
 \left( bc+{a}^{2} \right) }{{b}^{4}{c}^{4}}}
 $$
 
the one to $U_{b,c}$ is
$$
p_{U_{b,c}} = -{\frac { \left( {b}^{2}+{c}^{2}-bc \right)  \left( a-c \right)  \left( a-b \right)  \left( bc+{a}^{2} \right) }{{a}^{4}bc
}}
$$
and the one to $V$ is
$$
p_V = {\frac { \left( bc+{a}^{2} \right)  \left( ab+{c}^{2} \right)  \left( ac+{b}^{2} \right) }{{a}^{2}{c}^{2}{b}^{2}}}
$$
By this theory of Schur elements (see \cite{GECKPFEIFFER} theorem 7.2.6) we have that, for
each morphism $\varphi : R \to k$ for $k$ a field, $H_3 \otimes_{\varphi} k$ is semisimple
if and only if $\varphi(p_{\chi}) \neq 0$ for all irreducible representation $\chi$ of $H_3$.
Another related computation that can be found in \cite{THESE} is that the discriminant of the
trace form of the regular representation of $H_3$, that is to say of the action of $H_3$ on itself by left multiplication,
is
$$
\begin{array}{c}
2^{12}3^9a^6b^6c^6(c-b)^{10}(a-c)^{10}(b-a)^{10} \\
(c^2-cb+b^2)^6(a^2-ac+c^2)^6
(b^2-ba+a^2)^6(a^2+bc)^{14}(b^2+ac)^{14}(c^2+ab)^{14}. \\
\end{array}
$$

From this computation of Schur elements we get the following lemma.

\begin{lemma} \label{lem:injH3} Let $R_1$ be a domain, and $\varphi : R \to R_1$ a
morphism of rings such that $\varphi(p_{\chi}) \neq 0$ for all the irreducible representation $\chi$ of $H_3$ . Then the induced map
$$
\Phi_{H_3} \otimes R_1 : H_3 \otimes_{\varphi} R_1 \to R_1^3 \oplus M_2(R_1)^3 \oplus M_3(R_1)
$$
is injective.
\end{lemma}
\begin{proof} Let $k$ denote an algebraic closure of the fraction field of $R_1$. By the remarks above we get the
semisimplicity of $H_3 \otimes_{\varphi} k$ and, by Tits deformation theorem, that
$H_3 \otimes_{\varphi} k$ is isomorphic to $k^3 \oplus M_2(k)^3 \oplus M_3(k)$.
Moreover (see \cite{GECKPFEIFFER} theorem 7.4.6), the `decomposition map' between
$H_3 \otimes K$ and $H_3 \otimes_{\varphi} k$ induces an isomorphism
between simple modules, which implies that the morphism that we consider 
$\Phi_{H_3} \otimes_{\varphi} : H_3 \otimes_{\varphi} k \to k^3 \oplus M_2(k)^3 \oplus M_3(k)$
is a morphism from $H_3 \otimes_{\varphi} k$ to the sum of the matrix algebras
associated to its simple modules. Because $H_3 \otimes_{\varphi} k$
is semisimple, this morphism is indeed an isomorphism.
Since $H_3$ is free over $R_1$ the conclusion follows.
\end{proof}

Let $M$ be a $R$-module. Since the natural map $H_2 \to H_3$
is injective, we can identify $H_2$ with a $R$-subalgebra of $H_3$.

For $M$ a $R$-module, we let $MT_n(M)$ be the $R$-module
of $R$-linear maps $t : H_n \to M$ such that
$t(xy) = t(yx)$ for all $x ,y \in H_n$, and such that $t(xs_{n-1}) = t(xs_{n-1}^{-1})$,
for all $x$ in the image of the natural morphism $H_{n-1} \to H_n$.

\begin{prop} \label{propMT3} Let $M$ be a $R$-module. Then $MT_3(M)$ is isomorphic to $\Hom_R(R^4,M)$
under $t \mapsto t(1) e_1^* + t(s_1) e_2^* + t(s_1 s_2)e_3^* + t(s_1 s_2^{-1} s_1 s_2^{-1}) e_4^*$, 
where $(e_1,\dots,e_4)$ is the canonical basis of $R^4$ and the $e_i^* \in \Hom(R^4,M)$
are the obvious dual maps.
\end{prop}
\begin{proof} Let $t \in MT_3(M)$. Since $s_2$ is conjugated to $s_1$ and $t$ vanishes on $s_2 - s_2^{-1}$,
it also vanishes on $s_1- s_1^{-1}$, hence $t(s_1^{-1}) = t(s_1)$. Since $t$ also vanishes on
$s_1 s_2 - s_1 s_2^{-1}$, and because
 $$ \begin{array}{lcl}
 s_1^{-1} s_2^{-1} - s_1s_2 &=& (s_1^{-1} s_2^{-1} -s_1^{-1} s_2) + (s_1^{-1} s_2  - s_1s_2 )\\
 &=&(s_1^{-1} s_2^{-1} -s_1^{-1} s_2) + (s_1s_2s_1)(s_2^{-1} s_1  - s_2s_1 )(s_1s_2s_1)^{-1} \\
 \end{array}
 $$
we get that $t(s_1 s_2) = t(s_1 s_2^{-1}) = 
t( s_1^{-1} s_2^{-1})$. By proposition \ref{prop:H3AB} it follows that the map $MT_3(M) \to \Hom_R(R^4,M)$
described in the statement is injective (note that $t(s_1s_2^{-1}s_1s_2^{-1}) = t(s_2s_1^{-1} s_2 s_1^{-1}) = t(s_1^{-1} s_2 s_1^{-1} s_2)$
using $s_1s_2s_1$-conjugation and the invariance of traces under cyclic rotation). Conversely, because by proposition \ref{prop:H3AB} $H_3/[H_3,H_3]$
is a free module, it is possible to associate to each element in $ \Hom_R(R^4,M)$ a trace $t$ on $H_3$ satisfying
$t(s_1^{-1}) = t(s_1)$, $t(s_1 s_2) = t(s_1 s_2^{-1}) = 
t( s_1^{-1} s_2^{-1})$. Now, since $H_2$ is spanned by $1,s_1,s_1^{-1}$, the $R$-module spanned by the $x s_2 - x s_2^{-1}$ for $x \in H_2$ is spanned
by $s_2 - s_2^{-1}$, $s_1 s_2- s_1s_2^{-1}$, $s_1^{-1}s_2 - s_1^{-1}s_2^{-1}$, hence $t$ clearly vanishes on it.
\end{proof}

\section{Two BMW algebras as quotients of the braid groups}

\label{secttwobmwagebras}

We still denote $R = \Q[a,b,c,(abc)^{-1}]$ and let $S = R[(b+c)^{-1}]$, $S_{\pm} = S/(a \mp y)$.
There are two variants of the Birman-Wenzl-Murakami algebras, one which can be defined over $S_+$, the other one over $S_-$.
Usually, they are defined as a algebras over
$\Q[\alpha,\alpha^{-1}, q,q^{-1}, (q \pm q^{-1})^{-1}]$,
by generators $\sigma_1,\dots,\sigma_{n-1}$, braid relations between the $\sigma_i$'s, and three series of relations
involving the additional elements
$$
e_i = \frac{ \sigma_i^{-1} \pm \sigma_i}{q \pm q^{-1}} \mp 1,
$$
namely
\begin{enumerate}
\item $\sigma_i e_i = \alpha^{-1} e_i $
\item $e_i \sigma_{i+1} e_i = \alpha e_i$
\item $e_i \sigma_{i+1}^{-1} e_i = \alpha e_i$
\end{enumerate}
That these relations are enough to present the algebra originally
introduced in \cite{BW} was shown in \cite{WENZL}.
A classical remark is that, using conjugating properties inside the braid group,
these three \emph{series of relations} are equivalent to the three \emph{relations}
\begin{enumerate}
\item $\sigma_1 e_1 = \alpha^{-1} e_1$ 
\item $e_1 \sigma_{2} e_1 = \alpha e_1$
\item $e_1 \sigma_{2}^{-1} e_1 = \alpha^{-1} e_1$
\end{enumerate}
A slightly more convenient presentation for our purposes is to replace the
generators $\sigma_i$ by $s_i = \alpha^{-1} \sigma_i$. The formulas above become
$$
e_i = \frac{ \alpha^{-1} s_i^{-1} \pm \alpha s_i}{q \pm q^{-1}} \mp 1 = \frac{ \alpha^{-2} s_i^{-1}  \pm s_i}{\alpha^{-1}q \pm \alpha^{-1}q^{-1}} \mp 1,
$$
and
\begin{enumerate}
\item $s_1 e_1 = \alpha^{-2} e_1$ 
\item $e_1 s_{2} e_1 =  e_1$
\item $e_1 s_{2}^{-1} e_1 =  e_1$
\end{enumerate}
A classical consequence of the first relation is that $s_1$ and, therefore, all the $s_i$'s, satisfy a cubic relation,
and more precisely
$$
(s_i - \alpha^{-2})(s_i - \alpha^{-1} q) (s_i \mp \alpha^{-1} q^{-1}) = 0
$$
It follows that this algebra is actually defined over 
$$\Q[\alpha^{2},\alpha^{-2}, \alpha^{-1}q, \alpha^{-1}q^{-1},   (\alpha^{-1}q \pm \alpha^{-1}q^{-1})^{-1}]$$
which is isomorphic to $S_{\pm} = S/(a \mp bc)$ under $a \mapsto \alpha^{-2}$, $b \mapsto \alpha^{-1} q$,
$c \mapsto \pm \alpha^{-1} q^{-1}$. Using this isomorphism, we get that
$$
e_i =  \frac{  a s_i^{-1} \pm s_i}{b + c} \mp 1,
$$
and the defining relations become, in addition of the braid relations,
\begin{enumerate}
\item $(s_1 -a)(s_1-b)(s_1-c)=0$
\item $e_1 s_{2} e_1 =  e_1$
\item $e_1 s_{2}^{-1} e_1 =  e_1$
\end{enumerate}
and thus $BMW_n^{\pm}$ appears as the quotient of $H_n \otimes_{S}S_{\pm}$
by the ideal generated by two elements $\mathcal{S}_{\pm}$ and $\mathcal{S}'_{\pm}$, 
namely $e_1 s_2 e_1 - e_1$ and $e_1 s_2^{-1}e_1 - e_1$. Notice that each of these
elements, expressed in the $s_i$'s, is a linear combination of 12 terms originating from the
braid group. We call these the two 12-terms defining relations of the BMW algebras.
In this setting, the classical definition of the BMW algebras can be formulated as follows.

\begin{defi} $BMW_n^{\pm} = (H_n \otimes_R S_{\pm})/(\mathcal{S}_{\pm},\mathcal{S}'_{\pm})$.
\end{defi}

 We let $x = b+c$, $y = bc$. We recall the following easy consequences of the cubic relation
 and of the definition of $e_i$  :
 \begin{itemize}
 \item $ e_i^2 = \delta e_i$ with $\delta = \frac{1 \pm a \mp x}{x}$
 \item $e_i s_i = a e_i = e_i s_i$
 \end{itemize}
 
  We first prove that, if $x-a$
 is made invertible,
then these two 12-terms relations are equivalent, in other words that the defining ideal is generated
by either one of these two relations. 
\begin{prop} 
$$BMW_n^{\pm} \otimes_{S_{\pm}} S_{\pm}[(x-a)^{-1}] = \left( H_n \otimes_R S_{\pm}[(x-a)^{-1}] \right)/(\mathcal{S}_{\pm})
$$
{}
$$BMW_n^{\pm} \otimes_{S_{\pm}} S_{\pm}[(x-a)^{-1}] = \left( H_n \otimes_R S_{\pm}[(x-a)^{-1}] \right)/(\mathcal{S}'_{\pm})
$$
\end{prop}
\begin{proof}
We need to prove that the two relations $e_1 s_2^{-1} e_1 = e_1$ and
$e_1 s_2 e_1 = e_1$ are implied one by the other, inside $H_3 \otimes_R S_{\pm}[(x-a)^{-1}]$.
We assume that $e_1 s_2 e_1 = e_1$, and show $e_1 s_2^{-1} e_1 = e_1$, the proof of the converse
implication being similar.

We have $X = (e_1 s_2 e_1) s_2^{-1} e_1 = e_1 s_2^{-1} e_1$,
and also $X = e_1 (s_2 e_1 s_2^{-1}) e_1= e_1 s_1^{-1} e_2 s_1 e_1$, because the braid relations imply
$s_2 s_1^{u} s_2^{-1} = s_1^{-1} s_2^{u} s_1$ for all $u \in \Z$ and by expressing $e_i$ as a linear combination of $1,s_i$ and $s_i^{-1}$.
Now the cubic relation implies $s_1 e_1 = a e_1$ and $e_1 s_1^{-1} = a^{-1} e_1$, hence
$X = (e_1 s_1^{-1}) e_2 (s_1 e_1) = e_1 e_2 e_1$. By definition of $e_2$, this is
$$X = e_1 \left( \frac{  a s_2^{-1} \pm s_2}{x} \mp 1 \right) e_1=
  \frac{ a e_1s_2^{-1}e_1 \pm e_1 s_2e_1}{x} \mp e_1^2  
  =
  \frac{ a e_1s_2^{-1}e_1 \pm e_1 }{x} \mp \delta e_1 .
  $$
Altogether, this yields
$$
\left(1- \frac{a}{x} \right)e_1s_2^{-1}e_1 =  \mp \left(\delta -\frac{1}{x}  \right)  e_1 =
\mp \left(  \frac{1 \pm a \mp x - 1}{x}   \right)  e_1 = 
 \left(  \frac{ - a+ x }{x}   \right)  e_1 = 
\left( 1- \frac{  a }{x}   \right)  e_1 
$$
whence the conclusion.
\end{proof}

Recall that, inside $BMW_n^{\pm}$, we have $a = \pm y$.
\begin{prop} \label{propWBMWetendbien} We have $s_2^{-1} e_1 s_2^{-1} = a^{-2} s_1 e_2 s_1 = y^{-2} s_1 e_2 s_1$, and
$$
\frac{1}{x} s_2^{-1} s_1 s_2^{-1} - \frac{1}{xy^2} s_1 s_2 s_1 - s_2^{-1} + s_1^2 = \frac{1}{xy} s_1 s_2^{-1} s_1 - \frac{y}{x} s_2^{-1} s_1^{-1} s_2^{-1}.
$$
\end{prop}
\begin{proof}
We have $e_1 =  \frac{  a s_1^{-1} \pm s_1}{x} \mp 1$, that is
$as_1^{-1} \pm s_1 = x e_1 \pm x$ hence
$ s_1 =\pm x e_1+ x \mp a s_1^{-1}$. It follows that
$s_1 e_2 s_1 = \pm x e_1e_2 s_1+ xe_2 s_1 \mp a s_1^{-1}e_2 s_1$.
By the braid relations we have $s_1^{-1}e_2 s_1 = s_2e_1 s_2^{-1}$,
and we have $e_1e_2 s_1 = a e_1s_1^{-1} e_2 s_1
= a e_1s_2 e_1 s_2^{-1} = a e_1 s_2^{-1}$. Thus
$$s_1 e_2 s_1
= \pm x a e_1 s_2^{-1}+ x e_2 s_1 \mp a s_2e_1 s_2^{-1}.$$
Similarly,
$s_2^{-1} = \mp a^{-1} s_2 + a^{-1} x e_2 \pm x a^{-1}$,
hence
$s_2^{-1} e_1 s_2^{-1} = 
\mp a^{-1} s_2e_1 s_2^{-1} + a^{-1} x e_2e_1 s_2^{-1} \pm x a^{-1}e_1 s_2^{-1}$.
We have $e_2 e_1 s_2^{-1} = a^{-1} (e_2 s_2) e_1 s_2^{-1}= a^{-1} e_2 (s_2 e_1 s_2^{-1}) = 
a^{-1} (e_2 s_1^{-1} e_2) s_1 = 
a^{-1} e_2  s_1$ and
$$
\begin{array}{lcl}
s_2^{-1} e_1 s_2^{-1} &=& 
\mp a^{-1} s_2e_1 s_2^{-1} + a^{-2} x  e_2  s_1\pm x a^{-1}e_1 s_2^{-1} \\
&=& a^{-2} (  \pm x a e_1 s_2^{-1}+ x e_2 s_1 \mp a s_2e_1 s_2^{-1}) \\
&=& a^{-2} s_1 e_2 s_1.
\end{array}
$$
Using again $e_i =  \frac{  a s_i^{-1} \pm s_i}{x} \mp 1$ and $a = \pm y$ on both sides
one gets the conclusion by straightforward computation.
\end{proof}

Finally, following the method of Birman and Wenzl in \cite{BW}, we define Markov traces $t_n^{\pm} : BMW_n^{\pm} \to S_{\pm}$
by their images on words in the $\sigma_i's$, by closing the braid corresponding to it and
applying the Kauffman invariant of links, in its original or Dubrovnik variation (see \cite{KAUFREG}).
We recall that this is done by applying the skein relations
of Figure \ref{skeinkauf}, starting from the
additional conventional choice that the trivial knot diagram is mapped to 1, and then multiply by
$\alpha^r$ where $r$ is the writhe of the diagram, namely the number of crossings of the original braid
(in other terms, the image of the abelianization morphism $\ell : B_n \to \Z$ which maps $\sigma_i \mapsto 1$).
In particular, we have $t_n^{\pm}(s_1\dots s_{n-1}) = 1$, 
and 
$$
t_2^{\pm}(1) = \delta_K^{\pm} = \frac{y \mp x +1}{x}.
$$

The fact that the value of such a trace on braids lies inside the $S_{\pm}$
is a consequence of the following probably classical lemma.

\begin{lemma}
For all $\beta \in B_n$, $t_n^{\pm}(\beta) \in \Q[a,a^{-1},x,x^{-1}] \subset S_{\pm}$
\end{lemma}
\begin{proof}
Let $L$ be the closure of the braid $\beta$, and $\overrightarrow{D}$ an oriented link diagram
representing it. We prove more generally that, for an oriented link diagram $\overrightarrow{D}$,
the value of the Kauffman invariant $\overrightarrow{K}(\overrightarrow{D})$ lies in $\Q[a,a^{-1},x,x^{-1}]$. We do this
 by a double induction, first on the number of crossings,
and then, the number of crossings being fixed, on the minimal number of crossings needed to
be changed in order to get a diagram representing a trivial link. If $\overrightarrow{D}$ represents
a trivial link with $r$ components, we have $\overrightarrow{K}(\overrightarrow{D}) = t_n^{\pm}(1) = \delta_K^{\pm r} \in S_0$.
Otherwise,
by choosing a suitable crossing we can apply the first relation of Figure \ref{skeinkauf} to the corresponding
unoriented diagram $D$. Letting $D'$ the other diagram with the same number of
crossings, and $D_0$, $D_{\infty}$ the two other ones, and $K(D)$,$K(D')$, etc. the Kauffman polynomial
for unoriented links associated to them, we get
$K(D) = \mp K(D') \mp \eps (q \pm q^{-1}) (K(D_0) + K(D_{\infty}))$ for some $\eps \in \{ -1,1 \}$.
There exists 
oriented diagrams $\overrightarrow{D'},\overrightarrow{D_0},\overrightarrow{D_{\infty}}$
whose underlying unoriented diagrams are $D', D_0,D_{\infty}$. Then 
$\overrightarrow{K}(\overrightarrow{D})$
is equal to
$$
 \mp \overrightarrow{K}(\overrightarrow{D}') \alpha^{w(\overrightarrow{D})-w(\overrightarrow{D'})} \mp \eps (q\pm q^{-1}) \alpha  \left(\alpha^{w(\overrightarrow{D})-w(\overrightarrow{D_0})-1}
 \overrightarrow{K}(\overrightarrow{D}_0)
\pm \alpha^{w(\overrightarrow{D}_0)-w(\overrightarrow{D_{\infty}})-1}
\overrightarrow{K}(\overrightarrow{D}_{\infty})
\right)
$$
{}
Since the parity of the writhe only depends on the number of crossings the conclusion follows by induction.
\end{proof}

More precisely, this lemma shows that
the value of $t_n^{\pm}$ on such a braid belongs to the subalgebra of $S_{\pm}$
generated by $\alpha^{-2} = a$ and $\alpha^{-1}(q\pm q^{-1})  = b+c = x$
as well as their inverses. If $z = q \pm q^{-1}$, this subalgebra may also
be seen as the fixed subalgebra of $\Q[\alpha,\alpha^{-1},z,z^{-1}]$ fixed
by the involutive automorphism $\alpha \mapsto -\alpha$, $z \mapsto -z$.

Using the skein relations one gets in particular the following formula (which is the
value of the Kauffman polynomial on the figure-eight knot $4_1$).
$$
t_3^{\pm}(s_1 s_2^{-1}s_1 s_2^{-1}) = x^3(a^2 \pm a) + x^2 (a^2 + 2a \pm 1) - x(1 \pm a) - (\pm 1\pm a + a^{-1})
$$

\begin{figure} 
\begin{center}
\resizebox{10cm}{!}{\includegraphics{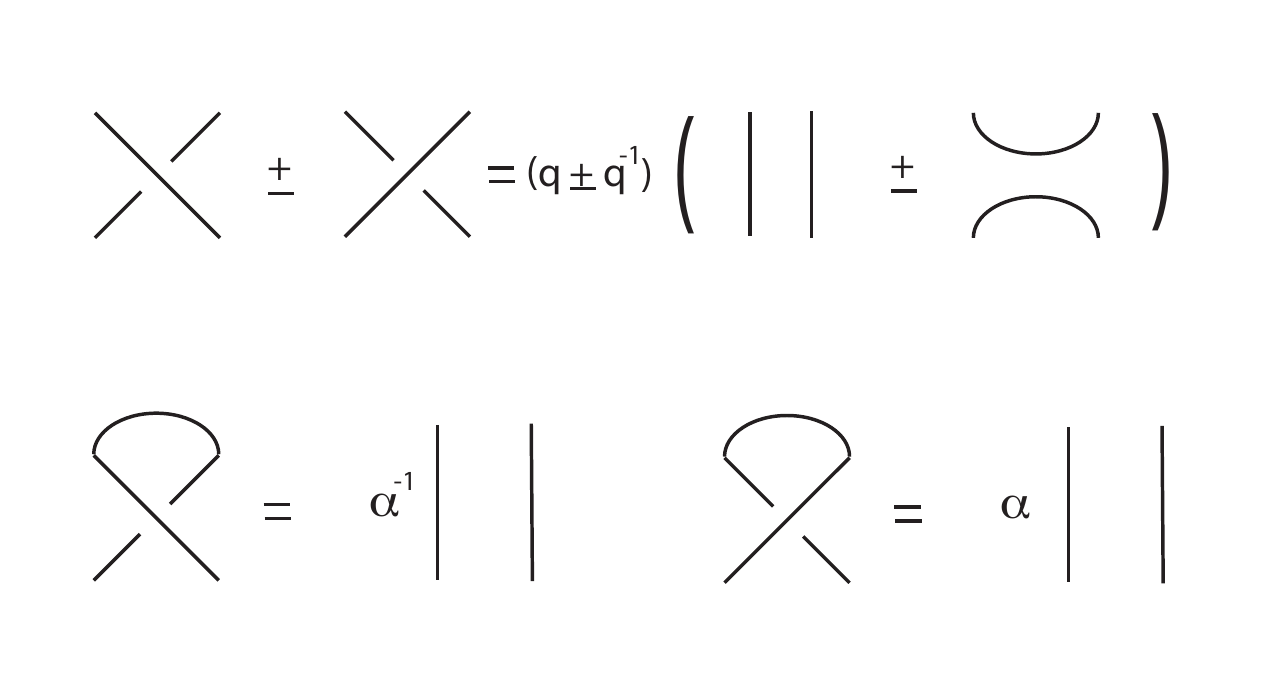}}
\end{center}
\caption{Skein relations for the Kauffman polynomial of unoriented links}
\label{skeinkauf}
\end{figure}

In terms of quantum groups, these two variants of the BMW algebras have their origin in the disctinction between
the symplectic and orthogonal groups acting on their standard module. Indeed, let $V$ denote the finite dimensional
complex vector space acted upon by the isometry group $G$ of some non-degenerate bilinear form. We assume $G$ is split
and fix a Cartan subalgebra of its Lie algebra $\mathfrak{g}$, and use Bourbaki conventions and notations for the weights (see \cite{LIE456}). Then $V$
is a fundamental module of highest weight $\varpi_1$. Then, Figure \ref{brattelitensor} represents the graph corresponding to the relation
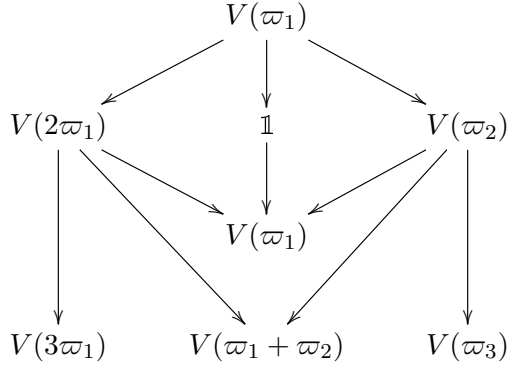
\begin{figure}
$$\xymatrix{
 & V(\varpi_1) \ar[dr] \ar[d] \ar[dl] & \\
 V(2 \varpi_1) \ar[dd] \ar[ddr] \ar[dr] & \un \ar[d] & V(\varpi_2) \ar[dl] \ar[ddl] \ar[dd] \\
 & V(\varpi_1) & \\
 V(3 \varpi_1) & V(\varpi_1 + \varpi_2) & V(\varpi_3) 
 }
$$
\caption{Decomposition of $V \otimes V$ and $V \otimes V \otimes V$}
 \label{brattelitensor}
\end{figure}
$x \to y$ meaning `$y$ appears as a constituent in $x \otimes V$' (which turns out to be a symmetric relation because $V$
is selfdual), and therefore is also the Bratteli diagram of the tower of centralizers algebra, which are well-known to be the
algebras of Brauer diagrams. The difference between the orthogonal and symplectic case is that $S^2 V = \un + V(2 \varpi_1)$
in the former case, while $\Lambda^2 V = \un + V(\varpi_2)$ in the latter -- where $V(\la)$ denotes the highest weight module corresponding
to $\la$, and $\un = V(0)$ is the trivial representation.Therefore, the quantum representation of the braid group obtained by monodromy
of the KZ 1-form
$$
\frac{h}{\ii \pi} \sum_{i<j} \Omega_{ij} \mathrm{dlog}(z_i-z_j)
$$
is such that the spectrum $\{ \tilde{a},\tilde{b},\tilde{c} \}$ of the Artin generators is given in table \ref{tablespectreosp} with $q = e^{h/2(m-2)}$,
\begin{table}$$
\begin{array}{|c|c|c|c|}
\hline
 & \un & 2 \varpi_1 & \varpi_2 \\
 & \tilde{a} & \tilde{b}\  \mathrm{ or }\  \tilde{c} & \tilde{b}\ \mathrm{ or } \ \tilde{c} \\
 \hline
 \hline
 SO(V) & q^{1-m} & q & -q^{-1} \\
 \hline
 SP(V) & -q^{1-m} & q^{-1} & -q \\
 \hline
 \end{array}
 $$
 \caption{Eigenvalues of the braid action}
  \label{tablespectreosp}
\end{table}
$m = \dim V$ in the orthogonal case, $m = - \dim V$ in the symplectic case (recall that
$\Omega_{ij}$ is the action on the tensor factors in position $(i,j)$ of $V^{\otimes n}$ of $\sum_k f_k \otimes f_k \in \mathfrak{g} \otimes \mathfrak{g}$ where
$f_1,f_2,\dots$ is an orthonormal basis of $\g$ w.r.t. the Killing form). Renormalizing
the eigenvalues by the formula $x = \tilde{x} q^{1-m}$, we get that $a = -bc$ in the orthogonal case,
while $a = bc$ in the symplectic case. Therefore the algebras $BMW_n^+$ and $BMW_n^-$ corresponds to the
symplectic and orthogonal groups, respectively. Moreover, $a$ is specialized to $\mp q^{2(1-m)}$.

\begin{center}
\end{center}
\begin{center}
\end{center}
\begin{center}
\end{center}

\section{A universal cover for the two BMW algebras, and a shorter defining relation}

We still denote $R = \Q[a,b,c,(abc)^{-1}]$, $S = R[(b+c)^{-1}]$. In all what follows, $\Q$ could be replaced by $\Z[\frac{1}{2}]$
without damage ; however the invertibility of $2$ is crucial at several steps.
Indeed, the BMW-algebras in characteristic 2 present very specific features (see \cite{CABANESMARIN} for results in this direction).
For $n \geq 3$, we let 
$I_n$ denote the ideal of $H_n$ generated by the elements
$$
\mathcal{R}_i = -b c s_i s_{i+1}^{-1} s_i + (bc)^2 s_{i+1}^{-1} s_i s_{i+1}^{-1} - s_i s_{i+1} s_i - (b+c) b^2c^2 s_{i+1}^{-2}
+ (b+c) s_i^2 + (bc)^3 s_{i+1}^{-1} s_i^{-1} s_{i+1}^{-1} 
$$
for $1 \leq i \leq n-2$.
Note that $\mathcal{R}_i = -b c s_i s_{i+1}^{-1} s_i + (bc)^2 s_{i+1}^{-1} s_i s_{i+1}^{-1} - s_i s_{i+1} s_i - (b+c) b^2c^2 s_{i+1}^{-2}
+ (b+c) s_i^2 + (bc)^3 s_{i}^{-1} s_{i+1}^{-1} s_{i}^{-1} 
$, and that each $\mathcal{R}_i$ is conjugated to $\mathcal{R}_1$ inside $H_n$. As a consequence $I_n$ is generated
as an ideal by $\mathcal{R}_1$.

\begin{defi} We let $\widetilde{BMW}_n = H_n/I_n$.
\end{defi}
We also note that $\mathcal{R}_i$ actually has coefficients in $R_0 = \Q[b,c,(bc)^{-1}] \subset R$. Let $\eps$
be
the involutive automorphism of the $\Q$-algebra $R_0$
defined by $b \mapsto c$, $c \mapsto b$. We have $R_0^{\eps} = \Q[x, y^{\pm 1}]$ with $y = bc$, $x = b+c$,
and $\mathcal{R}_i$ has coefficients in $R_0^{\eps}$. The relation $\mathcal{R}_i$ can thus be written
$$
0\equiv -y s_i s_{i+1}^{-1} s_i + y^2 s_{i+1}^{-1} s_i s_{i+1}^{-1} - s_i s_{i+1} s_i - xy^2 s_{i+1}^{-2}
+ x s_i^2 + y^3 s_{i}^{-1} s_{i+1}^{-1} s_{i}^{-1} 
$$
or (see Figure \ref{fig:6T})
{}
$$
 s_{i+1}^{-1} s_i s_{i+1}^{-1} \equiv \frac{1}{y} s_i s_{i+1}^{-1} s_i  +\frac{1}{y^2} s_i s_{i+1} s_i + x s_{i+1}^{-2}
- \frac{x}{y^2} s_i^2 - y s_{i}^{-1} s_{i+1}^{-1} s_{i}^{-1}. 
$$

The cubic relation $(s_i - a)(s_i^2 - x s_i + y) = s_i^3 - (a+x) s_i^2 + (y+ax) s_i - ay = 0$
implies $s_i^2 = (a+x) s_i - (y+ax) + ays_i^{-1}$, hence 
that
$s_{i+1} s_i^{-1} s_{i+1} = s_i^{-1} (s_i s_{i+1} s_i^{-1}) s_{i+1}
= s_i^{-1} s_{i+1}^{-1} s_{i} s_{i+1} s_{i+1}
= s_i^{-1} s_{i+1}^{-1} s_{i} s_{i+1} ^2  
= (a+x)  s_i^{-1} (s_{i+1}^{-1} s_{i} s_{i+1}) - (y+ax)  s_i^{-1} s_{i+1}^{-1} s_{i} + ay s_i^{-1} s_{i+1}^{-1} s_{i} s_{i+1}^{-1}
= (a+x)   s_{i+1} s_{i}^{-1} - (y+ax)  s_i^{-1} s_{i+1}^{-1} s_{i} + ay s_i^{-1} s_{i+1}^{-1} s_{i} s_{i+1}^{-1}
$.
Thus, $\mathcal{R}_i$ implies the following relation $\mathcal{R}'_i$ :
$$
\begin{array}{lcl}
s_{i+1} s_i^{-1} s_{i+1} &\equiv&
 (a+x)   s_{i+1} s_{i}^{-1} - (y+ax)  s_i^{-1} s_{i+1}^{-1} s_{i} 
 +a  s_{i+1}^{-1} s_i  +\frac{1}{y} a  s_{i+1} s_i 
 \\ & & + xay s_i^{-1} s_{i+1}^{-2}
- \frac{x}{y}a  s_i - ay^2 s_i^{-2} s_{i+1}^{-1} s_{i}^{-1}.
\end{array}$$

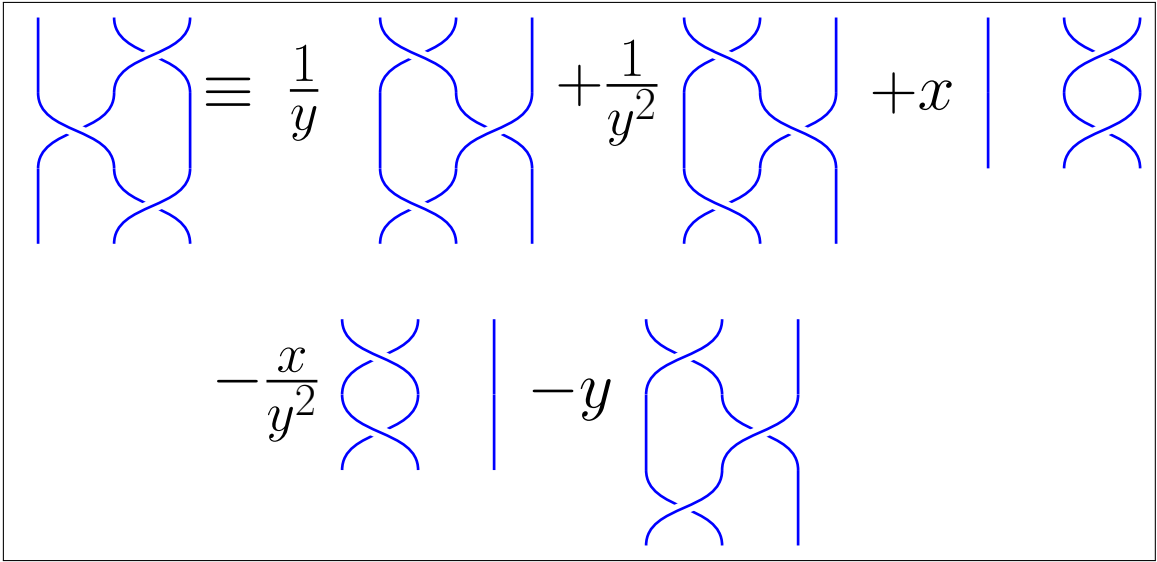
\begin{figure}
\fbox{
\hspace{-1cm}
\begin{tikzpicture} 
\braid[braid colour=blue,strands=3,braid start={(0,0)}]
{ \dsigma _2^{-1}  \dsigma_1 \dsigma_2^{-1}} 
\node[font=\Huge] at (3.5,-1) {\( \equiv \)};
\node[font=\Huge] at (4.5,-1) {\( \frac{1}{y} \)};
\braid[braid colour = blue,strands=3,braid start={(4.5,0)}]
{\dsigma_1 \dsigma_2^{-1} \dsigma_1}
\node[font=\Huge] at (8.5,-1) {\(+\frac{1}{y^2}\)};
\braid[braid colour = blue,strands=3,braid start={(8.5,0)}]
{\dsigma_1 \dsigma_2 \dsigma_1 }
\node[font=\Huge] at (12.5,-1) {\(+ x\)};
\braid[braid colour = blue,strands=3,braid start={(12.5,0)}]
{\dsigma_2^{-1} \dsigma_2^{-1} }
\node[font=\Huge] at (14,-1) {};
\node[font=\Huge] at (4,-5) {\( -\frac{x}{y^2} \)};
\braid[braid colour=blue,strands=3,braid start={(4,-4)}]
{  \dsigma_1 \dsigma_1} 
\node[font=\Huge] at (8,-5) {\(-y  \)};
\braid[braid colour = blue,strands=3,braid start={(8,-4)}]
{\dsigma_1^{-1} \dsigma_2^{-1} \dsigma_1^{-1}}
\end{tikzpicture}}
\caption{6-terms defining relation for $\widetilde{BMW}_3$}
\label{fig:6T}
\end{figure}

We let $\mathcal{H}_n(b,c)$ denote the usual Hecke algebra $R B_n/(s_1-b)(s_1-c)$. The natural projection
$R B_n \onto \mathcal{H}_n(b,c)$ obviously factorizes through $H_n$.

\begin{prop} {\ } \label{propstructgenWBMW}
\begin{enumerate}
\item The $R$-algebra morphism $H_n \onto \mathcal{H}_n(b,c)$ factorizes through $\widetilde{BMW}_n$.
\item If we abuse notations by letting $ \widetilde{BMW}_{n}$ also denote the
image of $\widetilde{BMW}_{n}$ inside $\widetilde{BMW}_{n+1}$ under the natural morphism
$s_i \mapsto s_i$, we have
$$ \widetilde{BMW}_{n+1}  = \widetilde{BMW}_{n} + \widetilde{BMW}_{n}s_n  \widetilde{BMW}_{n} +  \widetilde{BMW}_{n}s_n^{-1}
 \widetilde{BMW}_{n}.
 $$
 \item For all $n$,  $\widetilde{BMW}_{n}$ is a finitely generated $R$-module.
\end{enumerate}
\end{prop}
\begin{proof}
The defining relation of $\mathcal{H}_n(b,c)$ is
$s_i^2 = x s_i - y$, and 
$\mathcal{R}_i = -y s_i s_{i+1}^{-1} s_i + y^2 s_{i+1}^{-1} s_i s_{i+1}^{-1} - s_i s_{i+1} s_i - x y^2 s_{i+1}^{-2}
+ x s_i^2 + y^3 s_{i+1}^{-1} s_i^{-1} s_{i+1}^{-1}$ .  Inside $\mathcal{H}_n(b,c)$,
$ys_{i+1}^{-1} =  x -  s_{i+1}$ hence
$-ys_is_{i+1}^{-1}s_i = - xs_i^2 +  s_is_{i+1}s_i$ and
$\mathcal{R}_i =      y^2 s_{i+1}^{-1} s_i s_{i+1}^{-1}  - x y^2 s_{i+1}^{-2}
 + y^3 s_{i+1}^{-1} s_i^{-1} s_{i+1}^{-1}$. From $ys_{i}^{-1} =  x -  s_{i}$
 we get $y^3 s_{i+1}^{-1}s_{i}^{-1}s_{i+1}^{-1} =  xy^2s_{i+1}^{-2} - y^2 s_{i+1}^{-1}s_{i}s_{i+1}^{-1}$
 hence $\mathcal{R}_i = 0$ in $\mathcal{H}_n(b,c)$, which proves (i).
 We prove (ii) by induction, with $\WBMW_1 = R$, the case $n = 1$ being trivially
 true, as $\WBMW_2$ is spanned over $R$ by $1,s_1,s_1^{-1}$. The method
 is now similar to the one used in \cite{CABANESMARIN}, proposition 4.2.
 Let $U = \WBMW_n + \WBMW_n s_n \WBMW_n + \WBMW_n s_n^{-1} \WBMW_n
\subset \WBMW_{n+1}$. It is a $\WBMW_n$-submodule of $\WBMW_{n+1}$ containing $1$, so we
only need to prove $s_n U \subset U$. We have $s_n \WBMW_n \subset U$, so we need to
prove $s_n \WBMW_n s_n^{\pm 1} \WBMW_n \subset U$, and actually
only  $s_n \WBMW_n s_n^{\pm 1} \subset U$ is needed. By induction we know
$ \WBMW_n = \WBMW_{n-1}  + \WBMW_{n-1} s_{n-1} \WBMW_{n-1} + 
\WBMW_{n-1} s_{n-1}^{-1} \WBMW_{n-1} $
hence $s_n\WBMW_ns_n^{\pm 1}$ is equal to 
$$
\begin{array}{cl}
 & s_n\WBMW_{n-1}s_n^{\pm 1}  +s_n \WBMW_{n-1} s_{n-1} \WBMW_{n-1} s_n^{\pm 1}+ 
s_n\WBMW_{n-1} s_{n-1}^{-1} \WBMW_{n-1}s_n^{\pm 1} \\
=&\WBMW_{n-1}s_ns_n^{\pm 1}  + \WBMW_{n-1} s_ns_{n-1}s_n^{\pm 1} \WBMW_{n-1} + 
\WBMW_{n-1}s_n s_{n-1}^{-1}s_n^{\pm 1} \WBMW_{n-1}.
\end{array}
$$ It is thus sufficient to
prove that $s_n s_{n-1}^{-1}s_n \in U$, since $s_ns_{n-1}s_n^{\pm 1} = s_{n-1}^{\pm 1} s_n s_{n-1} \in U$
and $s_n s_{n-1}^{-1} s_n^{-1} = s_{n-1}^{-1} s_n s_{n-1} \in U$. This follows from relation $\mathcal{R}_n'$,
and proves (ii) by induction. (iii) is a trivial consequence of (ii).
\end{proof}

Inside $\WBMW_n \otimes_R S$ and $S$ we introduce the following elements
$$
e_i = \frac{a}{y} \left( \frac{y s_i^{-1} + s_i}{x} - 1 \right), \ \ \ \  \ \ \ \tilde{\delta} = \frac{a^2 - ax + y}{y}.
$$

\begin{prop} \label{proppresWBMWidemp}
We have 
$e_i s_i = a e_i$
and $s_{i+1}^{-1} e_i s_{i+1}^{-1} = y^{-2} s_i e_{i+1} s_i$. Moreover,
these two relations together with $e_i = \frac{a}{y} \left( \frac{y s_i^{-1} + s_i}{x} - 1 \right)$
provide a presentation of $\WBMW_n \otimes_R S$ with generators $s_1,\dots,s_{n-1}, e_1,\dots,e_{n-1}$,
and we have
$e_i^2 = \tilde{\delta}x^{-1} e_i$. 
\end{prop}
\begin{proof}
The fact that $e_i^2 = \tilde{\delta}x^{-1} e_i$ is a straightforward computation using the cubic relation and the
definition of $e_i$. Using the definition of $e_i$, $s_{i+1}^{-1} e_i s_{i+1}^{-1} = y^{-2} s_i e_{i+1} s_i$
is a rewriting of the defining relation $\mathcal{R}_i$, which proves the first claim. The fact that
these relations provide a presentation of $\WBMW_n$ follows from the fact that $e_i s_i = a e_i$,
$e_i^2 = \delta e_i$ (with $\delta = \tilde{\delta}/x$)
together
with the relation between $e_i$ and $s_i$ implies the cubic relation
on $s_i$, and then again $s_{i+1}^{-1} e_i s_{i+1}^{-1} = y^{-2} s_i e_{i+1} s_i$.
\end{proof}

It is thus possible to depict elements of $\WBMW_n \otimes_R S$ as diagrams in a similar flavour as
the elements of the usual $BMW$ algebras (see Figure \ref{figWBMW}).

\begin{figure}
\begin{center}
\resizebox{10cm}{!}{\includegraphics{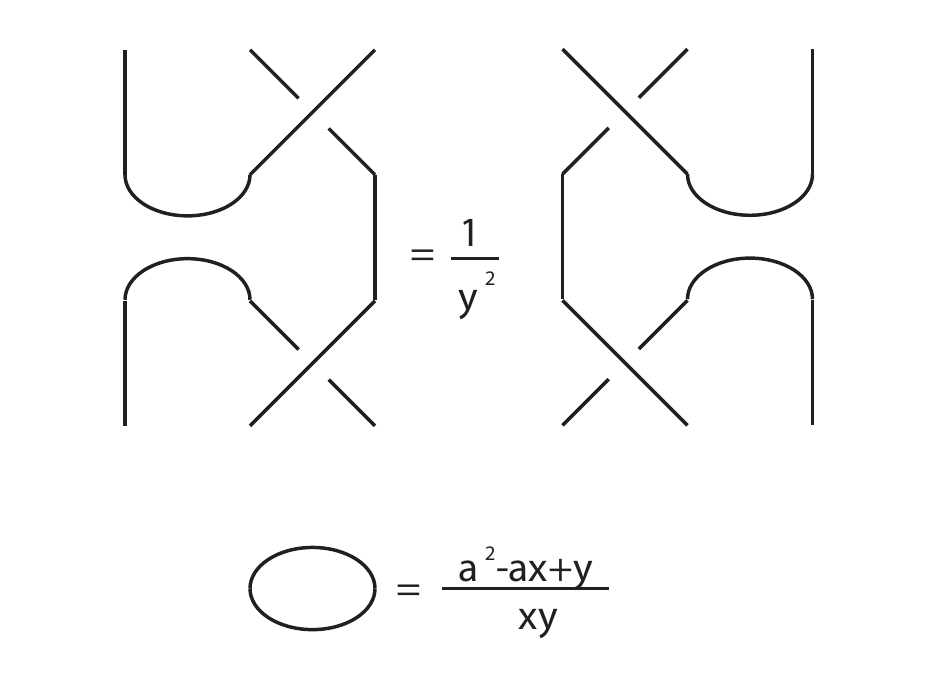}}
\end{center}
\caption{Diagrammatic relations for $\WBMW_n\otimes_R S$}
\label{figWBMW}
\end{figure}

\bigskip

For the sequel we need to compute the image of $\mathcal{R}_1$ inside the various representations of $H_3$. We get that
$S_b,S_c,U_{b,c},V$ all map $\mathcal{R}_1$ to $0$, while
$$
S_a : \mathcal{R}_1 \mapsto -{\frac { \left( -c+a \right)  \left( a-b \right)  \left( {a}^{2}-bc \right)  \left( {a}^{2}+bc \right) }{{a}^{3}}}
$$
{}
$$
U_{a,c} : \mathcal{R}_1 \mapsto 
 \left( \begin {array}{cc}  \left( a-b \right)  \left( ca+{b}^{2} \right) &-{\frac {c \left( a-b \right)  \left( ca+{b}^{2} \right) }{a}}
\\ \noalign{\medskip}-{\frac {c \left( a-b \right)  \left( ca+{b}^{2} \right) }{a}}&{\frac {{c}^{2} \left( a-b \right)  \left( ca+{b}^{2}
 \right) }{{a}^{2}}}\end {array} \right) 
$$

 $$
U_{a,b} : \mathcal{R}_1 \mapsto 
 \left( \begin {array}{cc}  \left( -c+a \right)  \left( ab+{c}^{2} \right) &-{\frac {b \left( -c+a \right)  \left( ab+{c}^{2} \right) }{a}}
\\ \noalign{\medskip}-{\frac {b \left( -c+a \right)  \left( ab+{c}^{2} \right) }{a}}&{\frac {{b}^{2} \left( -c+a \right)  \left( ab+{c}^{2}
 \right) }{{a}^{2}}}\end {array} \right) 
$$

Recall that the \emph{numerical invariant} of a (absolutely) semisimple $k$-algebra $A$
is the tuple $(n_1 \leq \dots \leq n_r)$ such that $A \otimes \overline{k} \simeq M_{n_1}(\overline{k}) \times
\dots \times M_{n_r}(\overline{k})$, where $\overline{k}$ denotes
any algebraically closed field containing $k$.

\begin{prop} \label{propisomBMWtilde}
 For all $n$, $\widetilde{BMW}_n \otimes K$ is a semisimple algebra
 whose 
 numerical invariant is the same as $BMW_n^+$. In particular
 it has dimension 
 $1.3.5.\dots.(2n+1)$.
\end{prop}
\begin{proof}
Let $\psi : R \to S_+[\la^{\pm 1}]$ be the algebra morphism defined by $a \mapsto \la bc$, $b \mapsto \la b$,
$c \mapsto \la c$.
We define a surjective morphism $S_+[\la^{\pm 1}] B_n \to BMW_n^+ \otimes_{S_+} S_+[\la^{\pm 1}]$ which maps
$\sigma_i \mapsto \la s_i$, where $\sigma_i$ is the $i$-th Artin generator of the braid group $B_n$.
 It is straightforward to check that this morphism factorizes according to the following diagram,
 where the map $S_+[\la^{\pm 1}] B_n \to H_n \otimes_{\psi} S_+[\la^{\pm 1}]$ is $\sigma_i \mapsto s_i$,
 and $H_n \otimes_{\psi} S_+[\la^{\pm 1}] \to 
 \WBMW_n \otimes_{\psi} S_+[\la^{\pm 1}] $ is naturally induced by the canonical map
$H_n \onto \WBMW_n$.
$$
\xymatrix{
S_+[\la^{\pm 1}] B_n \ar[r] \ar[d] &  BMW_n^+ \otimes_{S_+} S_+[\la^{\pm 1}] \\
H_n \otimes_{\psi} S_+[\la^{\pm 1}] \ar[ur] \ar[dr] & \\
& \WBMW_n \otimes_{\psi} S_+[\la^{\pm 1}] \ar[uu]
}
$$
Let $\tilde{K}$ the field of fractions of $S_+[\la^{\pm 1}]$. From this factorisation we
get  surjective morphisms $ H_n \otimes_{\psi} \tilde{K} \to \WBMW_n \otimes_{\psi} \tilde{K} \to BMW_n^+ \otimes_{S_+} \tilde{K}$.
The kernel $F_n$ of $H_n \otimes_{\psi} \tilde{K} \to BMW_n^+ \otimes_{S_+} \tilde{K}$ is
the ideal generated by the image of $F_3$ under the natural morphism $H_3 \otimes_{\psi} \tilde{K} \to H_n \otimes_{\psi} \tilde{K}$,
and similarly the kernel
$G_n$ of $H_n \otimes_{\psi} \tilde{K} \to \WBMW_n \otimes_{\psi} \tilde{K}$ is
generated by the image of $G_3$ under the natural morphism $H_3 \otimes_{\psi} \tilde{K} \to H_n \otimes_{\psi} \tilde{K}$.
In order to prove that the morphism $\WBMW_n \otimes_{\psi} \tilde{K} \to BMW_n^+ \otimes_{S_+} \tilde{K}$
is an isomorphism,
it is thus sufficient to check that $BMW_3^+ \otimes_{S_+} \tilde{K}$ and $\WBMW_3 \otimes_{\psi} \tilde{K}$ have the same
dimension. Since $BMW_3^+$ is free of rank 15 over $S_+$ we need to show that $\WBMW_3 \otimes_{\psi} \tilde{K}$ has
dimension 15. For example because of the computation of the Schur elements, we know by
section \ref{sectH3} that $H_3 \otimes_{\psi} \tilde{K}$ is semisimple, and isomorphic to a direct sum
of matrix algebras. Because of this, the ideal generated by an element is uniquely determined by the
collection of simple modules on which it vanishes. Because of the calculations above we know
that this element is nonzero exactly on $U_{a,c},U_{b,c},S_a$. It follows that the ideal has dimension
$1+2 \times 2^2 = 9$ hence the dimension of $\WBMW_3 \otimes_{\psi} \tilde{K}$ is $24-9 = 15$, as required.

We thus proved that $\WBMW_n \otimes_{\psi} \tilde{K}$ is semisimple, because $BMW_n^+ \otimes_{S_+} \tilde{K}$
is so, and that these two $\tilde{K}$-algebras have the same numerical invariant.
Now notice that $\tilde{K} = \Q(\la,b,c)$ is isomorphic to
$K=\Q(a,b,c)$ under
the isomorphisms
$$
\begin{array}{lcl}
 \tau^{-1} : & K \to \tilde{K} & a \mapsto \la bc, b \mapsto \la b, c \mapsto \la c \\
\tau : &  \tilde{K} \to K & \la \mapsto \frac{bc}{a}, b \mapsto \frac{a}{c}, c \mapsto \frac{a}{b}
\end{array}
$$
and we have $\WBMW_n \otimes_R K \simeq (\WBMW_n \otimes_{\psi} \tilde{K}) \otimes_{\tau } K$ as $K$-algebras.
This isomorphism $\tau$ between $ \tilde{K}$ and $K$ can be extended to an isomorphism between their algebraic closures $k$, $\tilde{k}$. If
$\WBMW_n \otimes_{\psi} \tilde{k} \simeq M_{n_1}(\tilde{k}) \oplus \dots \oplus  M_{n_r}(\tilde{k})$, then
$\WBMW_n \otimes_R k \simeq  M_{n_1}(k) \oplus \dots \oplus M_{n_r}(k)$ and the conclusion follows.
\end{proof}

\bigskip

We let $\overline{R} = R/(a^2-(bc)^2) = R/(a^2-y^2)$, $R_{\pm} = R/(a \mp y)$,
and $\overline{S} = S/(a^2-(bc)^2) = S/(a^2-y^2)$, $S_{\pm} = S/(a \mp y)$. By the Chinese Reminder Theorem
we have $\overline{R} \simeq R_+ \oplus R_-$, $\overline{S} \simeq S_+ \oplus S_-$. Let $\overline{BMW}_n = \widetilde{BMW}_n \otimes_R \overline{R}$.
From the preceding section we know that we have natural $\overline{R}$-algebra morphisms $\overline{BMW}_n \to
BMW_n^{\pm}$, hence a $\overline{R}$-algebra morphism $\overline{BMW}_n \to BMW_n^+ \oplus BMW_n^-$.

\begin{prop} \label{propisomBMWtilde2} The natural
morphism $\overline{BMW}_n \to BMW_n^+ \oplus BMW_n^-$ becomes an isomorphism
after tensorisation by $\overline{S}' = \overline{S}[ (bc-1)^{-1}]$.
\end{prop}

In other words, if $(bc-1)$ is assumed to be invertible, then the 6 terms relation
that we introduce here imply the defining 12 terms relations of the BMW algebras, and
thus we can use it to get a simpler definition of the BMW algebras as quotients of the group
algebra of the braid group.

Recall that the standard (12 terms) generators of the defining ideal of $BMW_n^+$ are
$$
\begin{array}{lcl}
\mathcal{S}_+ &:=& x^{-2}
s_1 s_2 s_1 -yx^{-1}s_1^{-1}s_2-yx^{-1}s_2s_1^{-1}+s_2-x^{-1}s_1s_2-x^{-1}s_2s_1
+yx^{-2}s_1s_2s_1^{-1}
\\ & & +yx^{-2}s_1^{-1} s_2 s_1+y^2x^{-2}s_1^{-1} s_2 s_1^{-1}-x^{-1}s_1-yx^{-1}s_1^{-1}+1
\end{array}
$$
and
$$
\begin{array}{lcl}
\mathcal{S}'_+ &:=& x^{-2}
s_1 s_2^{-1} s_1 -yx^{-1}s_1^{-1}s_2^{-1}-yx^{-1}s_2^{-1}s_1^{-1}+s_2^{-1}-x^{-1}s_1s_2^{-1}-x^{-1}s_2^{-1}s_1
+yx^{-2}s_1s_2^{-1}s_1^{-1}
\\ & & +yx^{-2}s_1^{-1} s_2^{-1} s_1+y^2x^{-2}s_1^{-1} s_2^{-1} s_1^{-1}-x^{-1}s_1-yx^{-1}s_1^{-1}+1
\end{array}
$$
and that the corresponding generators of the defining ideal of $BMW_n^-$ are
$$
\begin{array}{lcl}
\mathcal{S}_- &:=& x^{-2}
s_1 s_2 s_1 -yx^{-1}s_1^{-1}s_2-yx^{-1}s_2s_1^{-1}+s_2-x^{-1}s_1s_2-x^{-1}s_2s_1
+yx^{-2}s_1s_2s_1^{-1}
\\ & & +yx^{-2}s_1^{-1} s_2 s_1+y^2x^{-2}s_1^{-1} s_2 s_1^{-1}+x^{-1}s_1+yx^{-1}s_1^{-1}-1
\end{array}
$$
and
$$
\begin{array}{lcl}
\mathcal{S}'_- &:=& x^{-2}
s_1 s_2^{-1} s_1 -yx^{-1}s_1^{-1}s_2^{-1}-yx^{-1}s_2^{-1}s_1^{-1}+s_2^{-1}-x^{-1}s_1s_2^{-1}-x^{-1}s_2^{-1}s_1
+yx^{-2}s_1s_2^{-1}s_1^{-1}
\\ & & +yx^{-2}s_1^{-1} s_2^{-1} s_1+y^2x^{-2}s_1^{-1} s_2^{-1} s_1^{-1}+x^{-1}s_1+yx^{-1}s_1^{-1}-1
\end{array}
$$

\begin{proof}
We first notice that, since $\overline{R} \simeq R_+ \oplus R_-$, we
have $\overline{BMW}_n \otimes_{\overline{R}} \overline{S}'= \WBMW_n \otimes_R S'_+ \oplus \WBMW_n \otimes_R S'_-$,
where $S'_{\pm} = \overline{S}'/(a \mp bc)$,
so we only need to prove that the natural maps $\WBMW_n \otimes_R S'_{\pm} \to BMW_n^{\pm} \otimes_{R_{\pm}} S'_{\pm}$ are isomorphisms
of $S'_{\pm}$-algebras. For $n=2$ this is clearly true, so we can assume $n \geq 3$.

By definition, $BMW_n^+\otimes_{S_+} S'_+$ is the quotient of $H_n\otimes_S S'_+$ by the ideal generated by $\mathcal{S}$. The claim thus
amounts to saying that $\mathcal{S}$ is contained inside the twosided ideal of $H_n\otimes_S S'_+$ generated by $\mathcal{R}_1$. This
claim is substantiated by the next lemma, which concludes the proof of the proposition.

\begin{lemma} \label{lemRimpliesS} Inside $H_3 \otimes_R R_+$, we have 
$$
x^2(y-1) \mathcal{S}_+ =  \left( \frac{x+1}{y}  - \frac{1}{y} s_1 - s_2^{-1} \right) \mathcal{R}_1 s_2
$$
and
$$
x^2(y-1) \mathcal{S}'_+ =  \left( \frac{x+y}{y^2}  + \frac{y-1}{y} s_1^{-1} - \frac{1}{y^2} s_1 - s_2^{-1} \right) \mathcal{R}_1 s_2.
$$
Inside $H_3 \otimes_R R_-$, we have  
$$
x^2(y-1)\mathcal{S}_- = \left( \frac{1-x}{y}   -\frac{1}{y}   s_1- s_2^{-1} \right) \mathcal{R}_1 s_2
$$
and
$$
x^2(y-1)\mathcal{S}'_- = \left(  \frac{y-x}{y^2} + \frac{y-1}{y} s_1^{-1} - \frac{1}{y^2} s_1 - s_2^{-1} \right) \mathcal{R}_1 s_2.
$$

\end{lemma}
\begin{proof} In order to check equalities in $H_3 \otimes_R R_{\pm}$, we use the map $H_3 \otimes_R R_{\pm} \into R_{\pm}^3 \times M_2(R_{\pm})^3 \times M_3(R_{\pm})$ described in section \ref{sectH3},
which is injective by lemma \ref{lem:injH3},
and check by direct computation that both sides are equal.
\end{proof}
\end{proof}

We add a comment on how we got the equalities of lemma \ref{lemRimpliesS}, inside the algebra
$R_{\pm}^3 \times M_2(R_{\pm})^3 \times M_3(R_{\pm})$. First of all, the images of both $\mathcal{S}_{\pm}$ and $\mathcal{R}$ are $0$
on two of the 1-dimensional factors, on the factor $M_3(R_{\pm})$, and on two of the factors $M_2(R_{\pm})$. In order to get an expression
of $\mathcal{S}_{\pm}$ inside the ideal generated by $\mathcal{R}_1$, we only need to consider the map $H_3 \otimes_R R_{\pm} \to \End(S_{ \pm bc}) \oplus \End(U_{\pm bc,c}) \oplus \End(U_{\pm bc,b})$.
Moreover, because $\mathcal{R}_1$ and $\mathcal{S}$ have coefficients in $R^{\eps}$, we can restrict ourselves
to consider the map $H_3 \otimes_R R_{\pm} \to  \End(S_{ \pm bc}) \oplus \End(U_{\pm bc,c})$ and look for a linear combination with coefficients in $R^{\eps}_{\pm}$ of terms of
the form $g \mathcal{R}_1 g'$ with $g,g'$ elements of the braid group. Now, a direct computation shows that the image of $\mathcal{S}$ as well
as of $\mathcal{R}_1 s_2$ inside $\End(S_{\pm bc,b})$ are matrices of the form $\begin{pmatrix} \ast & 0 \\ \ast & 0 \end{pmatrix}$. Taking into account
their images in $S_{\pm bc}$, which belong to $\Q(x,y)=\Q(b,c)^{\eps}$, and because $\Q(b,c)^2 \simeq \Q(x,y)^4$ as a $\Q(x,y)$-vector space,
we get that $\mathcal{S}_{\pm}$ and the elements of $H_3 \mathcal{R}_1 s_2$ are uniquely determined by their image into a 5-dimensional
vector space over $\Q(x,y)$.

By multiplying $\mathcal{R}_1s_2$
on the left by suitable elements of $H_3$ one readily gets a basis
of this vector space, namely
$\mathcal{R}_1s_2$, $s_1^{-1}\mathcal{R}_1s_2$, $s_1\mathcal{R}_1s_2$, $s_2^{-1}\mathcal{R}_1s_2$, $s_2\mathcal{R}_1s_2$. Expressing the image of $\mathcal{S}_{\pm}$
in this basis we get the linear combinations of the lemma.

\medskip

Since $\bar{S}/(y-1) = \bar{S}/(a^2 -y) = S^{\dagger}/(a^2-y^2)$ with $S^{\dagger} = S/(a^2-y)$, the special case $a^2 = y = 1$ is a consequence of
the study of the algebra $BMW_n^{\dagger} = \WBMW_n \otimes_R S^{\dagger}$.
We let $K^{\dagger}$ denote the fraction field of $S^{\dagger}$.

\begin{prop} The dimension of $BMW_3^{\dagger} \otimes_{S^{\dagger}} K^{\dagger}$ is 16.
\end{prop}
\begin{proof} 
By the computations above
we now that 
the representation $S_a : H_3 \to R$
factorizes through $BMW_3^{\dagger}$
when tensored with $S^{\dagger}$, and that it is not the case for the representations $U_{a,c}$ and $U_{a,b}$.
Since $H_3 \otimes_R K^{\dagger}$ is semisimple, the conclusion follows, as $1^2+1^2+1^2+2^2+3^3 = 16$.
\end{proof}
\begin{cor} $\WBMW_3$ is not free over $S$.
\end{cor}

By the same method, using the freeness of $H_4,H_5$ proved in \cite{CUBIC5}, and modulo
the still conjectural
existence of convenient symmetrizing forms assumed in \cite{MALLE2}, we can
check from the Schur elements computed in \cite{MALLE2} that
$H_4 \otimes K^{\dagger}$ and $H_5 \otimes K^{\dagger}$ are semisimple, and
compute with the same method the dimension of $BMW_n^{\dagger} \otimes K^{\dagger}$ : the irreducible
representations corresponding to the ideal $I_n$ are the ones whose restriction to $BMW_3^{\dagger}$
contains a constituent isomorphic to $U_{b,c}$ or $U_{a,c}$. 
By this method we get that, if these symmetrizing forms exist, then
$$\dim_{K^{\dagger}} BMW_n^{\dagger} \otimes K^{\dagger} = 1+ \dim_K \WBMW_n \otimes K = 1 + \dim BMW_n^+
$$
for $n=3,4,5$. Note however that this formula is not valid for $n=2$. We leave the general case open.

Finally, we notice the following fact. Let $\eta \in \Aut(R)$ be defined by $a \mapsto -a$, $b \mapsto -b$, $c \mapsto -c$.
The fixed subring of $R$ is denoted $R^{\eta}$. It is straightforward to check that there is an involutive automorphism of $R^{\eta}$-algebra $\hat{E}$
of $H_n$ mapping $s_i \mapsto -s_i$ and acting as $\eta$ on elements of $R$. One checks easily that
$\hat{E}(\mathcal{R}_i) = - \mathcal{R}_i$,
$\hat{E}(\mathcal{S}_+) = - \mathcal{S}_-$,
$\hat{E}(\mathcal{S}_-) = - \mathcal{S}_+$. As a consequence we get the following.
\begin{prop} \label{propdefE}
There is an automorphism $E$ of the  $R^{\eta}$-algebra $\WBMW_n$ 
defined by
$r \mapsto \eta(r)$ for $r \in R$ and $s_i \mapsto -s_i$. It induces an automorphism
$\overline{E}$
of $\overline{BMW_n} \otimes_{\overline{R}} \overline{S}$ exchanging $BMW_n^+$ and $BMW_n^-$.
\end{prop}

The automorphism of proposition \ref{propdefE} relates the traces $t_n^+$ and $t_n^-$, as we will see in corollary \ref{cortnplusettnmoins} below.

\section{Traces on $\widetilde{BMW}_n$}

By definition, a Markov trace on the tower of algebras $(\WBMW_n)_{n \geq 1}$ with
values in a fixed $R$-module $M$ is a sequence $(t_n)_{n \geq 1}$ of $R$-linear maps $t_n = \WBMW_n \to M$
which are traces on each $\WBMW_n$ (that is, $t_n(xy)=t_n(yx)$ for $x,y \in \WBMW_n$) and which
satisfy the Markov conditions $t_{n+1}(\iota_n(x)s_n^{\pm 1}) = t_n(x)$ for each $x \in \WBMW_n$, where
$\iota_n$ denotes the natural map $\iota_n : \WBMW_n \to \WBMW_{n+1}$. 
We let $(t_n^H)$ denote the Markov trace induced by the Ocneanu trace on
the Hecke algebra $\mathcal{H}_n(b,c)$ through the factorization of proposition \ref{propstructgenWBMW} (i),
normalized by $t_1^H(1) = 1$.

\subsection{General inductive properties}

The following proposition is a corollary of proposition \ref{propstructgenWBMW} (ii).

\begin{prop} \label{unicitetn1} For all $n \geq 2$, $t_{n+1}$ is uniquely determined by $t_n$ and $t_{n+1}(1)$.
\end{prop}
\begin{proof}
In the sequel, we abuse notations by letting $\WBMW_r$ denote the image of $\WBMW_r$ inside $\WBMW_{n+1}$. 
We prove that, for all $r \leq n+1$, $t_{n+1}$ is uniquely determined by $t_n$ and by its value
on $\WBMW_r$. The case $r = 1$ is the statement of the proposition, since $\WBMW_1 = R$.
We prove this by descending induction on $r$, the case $r = n+1$ being trivial.
Because of proposition \ref{propstructgenWBMW} (ii) we know that
$\WBMW_{r+1} = \WBMW_r + \WBMW_r s_r \WBMW_r +  \WBMW_r s_r^{-1} \WBMW_r$.
If $x,y \in \WBMW_r$, we have $t_{n+1} (x s_r^{\pm 1} y) = t_{n+1} (yx s_r^{\pm 1} )$. Conjugating by 
$(s_1 s_2 \dots s_{n})^{n-r}$ maps $s_r^{\pm 1}$ to $s_n$ and $x,y$ to elements $x',y'$ of $\WBMW_n$. This yields
$ t_{n+1} (yx s_r^{\pm 1} ) =  t_{n+1} (y'x' s_n^{\pm 1} )= t_n(y'x')$. This proves that $t_{n+1}$ is determined by
$t_n$ and by its restriction to $\WBMW_r$. The conclusion follows by induction.
\end{proof}

\begin{prop} \label{prop:unicitetracegen} Let $t'_n$ be the composition of $t_n$ with $M \to M \otimes_R R[(a^2-bc)^{-1},(b+c)^{-1}]$. Then $(t'_n)$ is uniquely determined
by $ t'_1(1)$ and $t'_2(1)$.
\end{prop}
\begin{proof}
In view of the previous proposition, it is sufficient to prove that,
for all $n \geq 2$, $t'_{n+1}(1)$ is determined by $t'_n$.
We have that $\mathcal{R}_{n-1}$ belongs to
$$
H_n s_n H_n + H_n s_n^{-1} H_n - (b+c) b^2c^2 s_{n}^{-2}+ (b+c) s_{n-1}^2 + (bc)^2 s_n^{-1} s_{n-1} s_n^{-1}.
$$
Since $s_n^{-1} s_{n-1} s_n^{-1}$ is conjugated to $s_{n-1}^{-1} s_n s_{n-1}^{-1}$, the Markov property implies that 
$(b+c)t_{n+1}(s_{n-1}^2 - (bc)^2 s_n^{-2})$ is determined by $t_n$. It is straightforward to check that,
because of the cubic relation, $s_{n-1}^2 - (bc)^2 s_n^{-2}$ is equal to $\frac{(b+c)}{a}(a^2-bc).1$ plus a linear
combination of $s_{n-1}$, $s_{n-1}^{-1}$, $s_n$, $s_n^{-1}$, on which the value of $t_{n+1}$ is clearly determined
by $t_n$. It follows that $(b+c)^2(a^2-bc)t_{n+1}(1)$ is uniquely determined by $t_n$, and the conclusion follows.

\end{proof}

We may compare this statement with the stronger assertion one has on $BMW_n^+$.
\begin{prop} \label{prop:unicitetracegenplus} If $(T_n : BMW_n^+ \to S_+)$ is a Markov trace, then it is uniquely
determined by $T_1(1)$ and $T_2(1)$.

\end{prop}
\begin{proof}
As before it is sufficient to prove that,
for all $n \geq 2$, $T_{n+1}(1)$ is determined by $T_n$.
We have that the 12 terms relation can be written as
$$
\begin{array}{lcl}
e_{n-1} s_n e_{n-1} - e_{n-1}&=&x^{-2}
s_{n-1} s_{n} s_{n-1} -yx^{-1}s_{n-1}^{-1}s_{n}-yx^{-1}s_{n}s_{n-1}^{-1}+s_{n}-x^{-1}s_{n-1}s_{n}\\ & & -x^{-1}s_{n}s_{n-1}
+yx^{-2}s_{n-1}s_{n}s_{n-1}^{-1}+yx^{-2}s_{n-1}^{-1} s_{n} s_{n-1}
\\ & & +y^2x^{-2}s_{n-1}^{-1} s_{n} s_{n-1}^{-1}-x^{-1}s_{n-1}-yx^{-1}s_{n-1}^{-1}+1
\end{array}
$$ 
hence belongs to $  BMW_n^+ s_n BMW_n^+ + BMW_n^+ s_n^{-1} BMW_n^+ - x s_{n-1}-yx^{-1}s_{n-1}^{-1}+1$,
hence, since $T_{n+1}(s_{n-1}^{\pm 1}) = T_{n+1}(s_{n}^{\pm 1}) = T_n(1)$, we get that $T_{n+1}(1)$ is indeed determined by $T_n$ and the conclusion follows.
\end{proof}

\begin{cor} \label{cortnplusettnmoins} Over $BMW_n^+$, we have $t_n^+ = (-1)^{n-1} \eta \circ t_n^- \circ \overline{E}$.
\end{cor}
\begin{proof}
We first prove that the RHS defines a Markov trace on $BMW_n^+$, with values in $S_+$.
We denote $T_n$ this RHS.
We clearly have $T_n(\la \beta) = \la T_n( \beta)$, $T_n(\alpha+ \beta) = T_n(\alpha) + T_n(\beta)$ and $T_n(\alpha \beta) = 
T_n(\beta \alpha)$ for all braids $\alpha, \beta \in BMW_n^+$ and scalar $\la \in S_+$, and
also $T_{n+1}(\beta s_n^{\pm 1}) 
= (-1)^n t_{n+1}^-(\overline{E} (\beta s_n^{\pm 1}))
= (-1)^n t_{n+1}^-(\overline{E} (\beta) \overline{E}(s_n^{\pm 1})))
= (-1)^{n-1} t_{n+1}^-(\overline{E} (\beta) s_n^{\pm 1}))
= (-1)^{n-1} t_n^-(\overline{E} (\beta) )) = T_n(\beta)$.
By the proposition above we first need to prove that $T_2(1) = t_2^+(1)$ and $T_2(s_1) = t_2^+(s_1)$.
Recall that $t_2^{\pm}(1) =\delta_K^{\pm}$, $t_2^{\pm}(s_1) = 1$ with
$$
\delta_K^{\pm} = \frac{y \mp x + 1}{x}.
$$
Since $\eta(\delta_K^-) = - \delta_K^+$ we
get $T_2(1) =  \delta_K^+ = t_2^+(1)$ and $T_2(s_1) =1= t_2^+(s_1)$, and the
conclusion follows.

\end{proof}

This corollary proves that the `natural' diagram below
is commutative only up to a sign depending on $n$.

$$
\xymatrix{
 & \overline{BMW}_n \otimes_{\overline{R}} \overline{S}  \ar[dl] \ar[dr] & \\
 BMW_n^+ \ar[d]_{t_n^+} \ar[rr]^{\overline{E}} & & BMW_n^- \ar[d]^{t_n^-} \\
 S_+ \ar[dr] & & S_- \ar[ll]^{\eta} \ar[dl] \\
 & S & }
$$

\subsection{Restrictions to 3 strands, and the Ocneanu trace}

We recall that the notation $MT_3(M)$ was defined in section \ref{sectH3}. We note that
$$
t_3^H(1) = \left( \frac{y+1}{x} \right)^2 \ \  
t_3^H(s_1) =\frac{y+1}{x}  \ \  
t_3^H(s_1s_2) = 1 \ \  
$${}
and define additional elements $t_3^S$, $t_3^K$
of $MT_3(M)$ by
$$
t_3^S(1) = a^3,\ \ \ \ t_3^S(s_1) = t_3^S(s_1s_2) = 0, \ \ \ \ 
t_3^S((s_1s_2^{-1})^2)=  x^2(2a+x),
$$
{}
$$
t_3^K(1) = \delta_K^2,\ \ \ \ t_3^K(s_1) = \delta_K,\ \ \ \ t_3^K(s_1s_2) = 1,
\ \ \ \ \delta_K =  \frac{y^2 - ax+y}{xy}, 
$${}
$$
 t_3^K((s_1 s_2^{-1})^2) = x^3 y(y+1) + x^2(y^2 + 2a + \frac{a}{y}) - x(1+y) - (y + a + \frac{a}{y}).
$$
\begin{prop} \label{prop:MT3bmw} Let $t \in MT_3(M)$ factorizing through $\WBMW_3$.
\begin{enumerate}
\item  The trace $t$ is uniquely
determined by $t(1)$, $t(s_1)=t(s_2)$ and $t(s_1s_2)$. 
\item If $x(a^2-y^2)$ is invertible in $M$ then $t$ is  uniquely determined by $t(1)$ and $t(s_1s_2)$.
\item If $x^2(a^2-y)$ is invertible in $M$ then $t$ is  uniquely determined by $t(s_1)$ and $t(s_1s_2)$.
\item If $x(a^2-y^2)$ and $x^2(a^2-y)$ are invertible in $M$ then $t$ is uniquely determined by $t(s_1s_2)$.
\end{enumerate}
\end{prop}
\begin{proof} Let $t \in MT_3(M)$ factorizing though $\WBMW_3$. With the notations of proposition \ref{propMT3}
we write $t = \alpha e_1^* + \beta e_2^* + \gamma e_3^* + \delta e_4^*$. By direct calculation,
we check that the equation $t(\mathcal{R}_1s_1) = 0$ means
$$0=-{x}^{2} \left( {a}^{2}+ax+y \right) \alpha +
 {\frac {x \left( 2\, \left( y+1 \right) {a}^{3}+2\,x \left( y+1 \right) {a}^{2}+a \left( {x}^{2}+y \left( y+1
 \right)  \right) +yx \right) }{a}}\beta $${}$$+{\frac {- \left( y+1 \right) ^{2}{a}^{3}-x \left( y+1 \right) ^{2}{a}^{2}+a \left( -{x}^{2}-2\,{x}^{2}y+y \left( {y}^{
2}+y+1 \right)  \right) -yx \left( y+1 \right) }{a}} \gamma + {y}^{2} \delta
$$ 
and, since $y$ is invertible, this proves that $t$ is determined by $\alpha, \beta$ and $\gamma$,
meaning that $t$ is uniquely determined by $t(1)$, $t(s_1)$ and $t(s_1s_2)$, which proves (i).
Similarly, we get that $t(s_1^{-1}\mathcal{R}_1)=0$ reads
$$
(a^2-y^2) \left( x t(s_1) - (y+1) t(s_1s_2) \right) = 0
$$
and $t(\mathcal{R}_1)=0$ means
$$
x^2(a^2-y) t(1) = \frac{(a^2-y)(y+1) - x(y-1)a}{a} \left( x t(s_1) - (y+1) t(s_1s_2) \right) + \frac{x}{a}(y+1)(a^2-y) t(s_1)
$$
which easily implies (ii)-(iv).
\end{proof}

\begin{cor} Let $(t_n)$ be a Markov trace factoring through $(\WBMW_n)$ with values in $M$.
If  $x(a^2-y^2)$ and $x^2(a^2-y)$ are invertible in $M$ then $(t_n)$ is a multiple of $(t_n^H)$ composed
with some morphism $R \to M$.
\end{cor}
\begin{proof}
Immediate consequence of (iv) together with proposition \ref{prop:unicitetracegen}.
\end{proof}

Recall that $S^{\dagger} = S/(a^2-y) = R[x^{-1}]/(a^2-y)$ and $BMW_n^{\dagger} = \WBMW_n \otimes_R S^{\dagger}$.

\begin{prop} \label{propmarkovypas1} Let $t'_n$ be the composite of $t_n$ with $M \to M \otimes_R S^{\dagger}[ (y-1)^{-1}]$.
 Then $(t'_n)$ is a scalar multiple of the composite of $t_n^H$ with some morphism $R \to M$ and with $M \to M \otimes_R S^{\dagger}[ (y-1)^{-1}]$.
 \end{prop}
\begin{proof}
Inside $S^{\dagger}[(y-1)^{-1}]$ we have $a^2 -y^2 = y(1-y)$, hence by proposition \ref{prop:MT3bmw} (ii)
we have that $t'_3$ is uniquely determined by $t'_3(1)$ and $t'_3(s_1s_2)$,
and more precisely the equation $t'_3(s_1^{-1}\mathcal{R}_1)=0$ means
$t'(s_1) = \frac{y+1}{x} t'(s_1 s_2)$. Let $t_n^{\circ} = t'_n - t'_3(s_1s_2) t_n^H$.
Then $t_3^{\circ}(s_1s_2) = 0$ hence $t_3^{\circ}(s_1) = 0$ and therefore $t_3^{\circ} =  a^{-3}t_3^{\circ}(1) t_3^S$.
Let $\tau = t_4^{\circ} \circ \iota_3$, where $\iota_3 : \WBMW_3 \to \WBMW_4$ is the natural map. Again by
 \ref{prop:MT3bmw} (ii) we know that $\tau$ is uniquely determined by $\tau(1)$ and $\tau(s_1s_2)$. Moreover
 $\tau(s_1s_2) = t_4^{\circ}(s_1s_2) = t_4^{\circ}(s_2s_3) = t_3^{\circ}(s_2) = t_3^S(s_2) = 0$.
 It follows that $\tau(s_1) = 0$, since $\tau(s_1) = \frac{y+1}{x} \tau(s_1 s_2)$.
 But $\tau(s_1) =  t_4^{\circ}(s_1) =t_4^{\circ}(s_3) = t_3^{\circ}(1)$, hence $t_3(1) = 0$ and
 $t_3^{\circ} = 0$. We prove by induction on $n$ that $t_n^{\circ}(1) = 0$. Assume we
 know that $t_r^{\circ}(1) = 0$ for all $r \leq n$, and consider the composite $\tau$ of $t_{n+1}^{ \circ}$
 with the natural map $\WBMW_3 \to \WBMW_{n+1}$. We have $\tau(s_1s_2) = t_{n+1}^{\circ}(s_1 s_2)
 = t_{n+1}^{\circ}(s_{n-1} s_{n})= t_{n-1}^{\circ}(1) = 0$. It follows as before that $\tau(s_1) = 0$,
 hence $t_n^0(1) = t_{n+1}^0(s_n) = t_{n+1}^0(s_1) = \tau(s_1) = 0$. It follows that $t_n^{\circ}(1) = 0$,
 and the claim follows by induction because of proposition \ref{unicitetn1}.
\end{proof}

\subsection{The Kauffman trace}

Recall that we defined in \S \ref{secttwobmwagebras} two Markov traces $t_n^{\pm} : BMW_n^{\pm} \to S_{\pm}$. Using the Chinese
Remainder Theorem isomorphism $\overline{S} \simeq S_+ \oplus S_-$ they
can be patched together into a Markov trace $\overline{BMW}_n \otimes_{\overline{R}}  \overline{S} \to
\overline{S}$, which extends to a Markov trace $t_n^K : \WBMW_n \to \overline{S}$.

$$
\xymatrix{
  &\overline{BMW}_n\otimes_{\overline{R}}S_+ \ar[r]&  BMW_n^+ \ar[r]^{t_n^+} & S_+ \ar[dr] & \\
 \overline{BMW}_n\otimes_{\overline{R}} \overline{S} \ar[ur] \ar[dr] \ar@{.>}[rrrr]^{t_n^K}& & & & \overline{S} \\
  & \overline{BMW}_n\otimes_{\overline{R}}S_- \ar[r]& BMW_n^- \ar[r]_{t_n^-} & S_- \ar[ur] & \\
}
$$

Note that the value of $t_3^K((s_1 s_2^{-1})^2)$ define before indeed matches
this new definition, because of the
computation of the Kauffman invariant of the figure-eight knot we did in section \ref{secttwobmwagebras}.

\begin{prop} Let $t'_n$ be the composite of $t_n$ with $M \to M \otimes_R \overline{S}[ (a^2-y)^{-1}]$.
 Then $(t'_n)$ is a linear combination of the composites of $t_n^H$ and $t_n^K$ with some morphism $\overline{S} \to M \otimes_R \overline{S}[ (a^2-y)^{-1}]$.
\end{prop}
\begin{proof}
By proposition \ref{prop:unicitetracegen} we know that $t'_n$ is uniquely determined by $t'_1(1) = t'_3(s_1s_2)$ and $t'_2(1) = t'_3(s_1)$. It is thus
sufficient to show that $t_3^H$ and $t_3^K$ induce a basis of $( \overline{S}[ (a^2-y)^{-1}])^2$ under $t'_3 \mapsto (t'_3(s_1s_2),t'_3(s_1))$.
The corresponding $2 \times 2$ matrix has for determinant
$$
\left| \begin{array}{cc} \delta_H & 1 \\ \delta_K & 1 \end{array} \right| = \delta_H - \delta_K = \frac{a}{y} \in ( \overline{S}[ (a^2-y)^{-1}])^{\times}
$$
with $\delta_H = \frac{y+1}{x}$, and this concludes the proof.
\end{proof}

\begin{remark} It is well-known that both the Kauffman polynomial and
the HOMFLYPT polynomial both specialize to the Jones polynomial
(see e.g. \cite{LICKO} p. 180). It could have been expected that this coincidence would have appeared in the previous proposition.
It clearly does not, since any specialized value of the invariants defined by our
traces $t_n^H$ and $t_n^K$
cannot coincide on the 2-components unlink (we always have $\delta_H \neq \delta_K$).
What happens is that, according to \cite{LICKO} proposition 16.6, the trace $t_n^K$ provides the Jones
polynomial when specialized to $\alpha = -t^{\frac{-3}{4}}$, $q = t^{\frac{1}{4}}$, that
is $a = q^{-6}$, $b = -q^{-2} = -t^{\frac{-1}{2}}$, $c = -q^{-4} = -t^{-1}$ while, according to \cite{LICKO} proposition 16.5, $t_n^H$
provides the Jones polynomial when $\{b,c \}$ is specialized to $\{ - t^{\frac{1}{2}}, t^{\frac{3}{2}} \}$, and the
only value of $t$ for which these two parametrizations coincide is $t^{\frac{1}{2}} = -1$, in which
case we get $b+c = 0$, which is forbidden here.
\end{remark}

\subsection{An additional trace when $y=a^2= 1$}

Let $S^{\dagger\dagger}  = S/(a^2-1,y-1) = S^{\dagger}/(a^2-y^2) = \overline{S}/(a^2-y)$,
and $BMW_n^{\dagger\dagger} = \WBMW_n \otimes_R S^{\dagger \dagger}$.

We first deal with the very special case $x = -2a$.
\begin{prop} \label{propTracesXm2a}
Let $(x_n)_{n \geq 1}$ denote a sequence with values in some $S^{\dagger\dagger}$-module
satisfying $(x+2a) M = 0$. Then there exists a Markov trace $(t_n^X)$ on $(\widetilde{BMW}_n)$
with values in $M$ such that, for every braid $g$, $t_n^X(g) = x_{\# \hat{g}}$, where
$\hat{g}$ denotes the closure of the braid $g$ and $\# L$ denotes the number of
components of the link $L$. Moreover, every Markov trace on $(\widetilde{BMW}_n)$ with
values in $M$ is of that form.
\end{prop}
\begin{proof}
We first note that the relations $(s_i-a)(s_i^2 + 2a s_i +1) = (s_i -a)(s_i+a)^2 = (s_i^2-1)(s_i+a) = 0$
hold true in $M$. Therefore, the action of $BMW_n^{\dagger\dagger}$ on $M$ factors
through the group algebra $\Q[a]/(a^2-1) \mathfrak{S}_n$ of the symmetric group. Moreover, the
natural map $H_n \otimes S^{\dagger\dagger}/(x+2a) \to \Q[a]/(a^2-1) \mathfrak{S}_n$, where $H_n$
is the cubic Hecke algebra defined by $(s_i^2-1)(s_i+a) = 0$, is clearly surjective, and $\mathcal{R}_1$
is easily checked to map to $0$. Therefore, the map $\widetilde{BMW}_n \otimes S^{\dagger\dagger}/(x+2a) \to \Q[a]/(a^2-1) \mathfrak{S}_n$
is surjective.

We show that there exists a Markov trace $(t_n^X)$ on $((\Q[a]/(a^2-1)) \mathfrak{S}_n)$
fulfilling the conditions of the statement. Since the formula
$t_n^X(g) = x_{\# \hat{g}}$ clearly defines an invariant of links, and therefore a Markov
trace on the tower of algebras of the braid groups, it is sufficient to prove that
$t_n^X$ vanishes on the defining ideal of $\Q[a]/(a^2-1) \mathfrak{S}_n$ for any given $n$.
This ideal is the linear span of the $A s_i^2 B - AB$ for $A,B$ two arbitrary braids on $n$
strands. Since the closures of $A s_i^2 B$ and $AB$ have the same number of components,
we indeed get that $(t_n^X)$ is a well-defined Markov trace on $(\widetilde{BMW}_n)$.
Since $t_n^X(1) = x_n$, the fact that all Markov traces are obtained this way is a consequence of proposition \ref{unicitetn1}.

\end{proof}

We can now state a general statement.
\begin{prop} {\ }  \label{propclassmarkovddagger3}
\begin{enumerate}
\item There exists a Markov trace $(t_n^{\dagger\dagger})$ on $\WBMW_n$ with values in $S^{\dagger\dagger}$, given by 
$t_n^{\dagger\dagger}(\beta) = a^n \psi_n(\beta)$, where $\psi_n : \WBMW_n \to S^{\dagger\dagger}$ is an algebra morphism
defined by $s_i \mapsto a$.
\item Let $t'_3$ be the composite of $t_3$ with $M \to M \otimes_R S^{\dagger\dagger}[ (a-x)^{-1},(2a-x)^{-1}]$.
 Then $(t'_3)$ is a linear combination of the composites of $t_3^H$, $t_3^K$, $t_3^{\dagger\dagger}$ 
 with
 coefficients inside $M \otimes_R S^{\dagger\dagger}[ (a-x)^{-1},(2a-x)^{-1}]$.
\end{enumerate}
\end{prop}
\begin{proof}
We first check that the $R$-algebra morphism $\psi_n : H_n \to S^{\dagger\dagger}$ defined by $s_i \mapsto a$
indeed factorizes through $\WBMW_n$, namely that $\psi_n(\mathcal{R}_1) = 0$, by direct calculation. Then
$t_n^{\dagger\dagger}(x) = a^n \psi_n(x)$ clearly defines a trace for every $n$, and we need to check
the Markov property, namely that $t_{n+1}^{\dagger\dagger}(x s_{n}^{\pm 1}) = t_n^{\dagger\dagger}(x)$
for all $x \in \WBMW_n$. This holds because 
$t_{n+1}^{\dagger\dagger}(x s_{n}^{\pm 1}) = a^{n+1} \psi_{n+1}(x s_{n}^{\pm 1})
= a^{n+1} \psi_{n+1}(x) a^{\pm 1}
= a^{n+1} \psi_{n}(x) a^{\pm 1} 
= a^{n+1} \psi_{n}(x) a
= a^{n+2} \psi_{n}(x)
= a^n \psi_n(x) = t_n^{\dagger\dagger}(x)$. This proves (i).
Note that $t_3^{\dagger\dagger}(1) = a^3$, $t_3^{\dagger\dagger}(s_1) = a^2=1$, $t_3^{\dagger\dagger}(s_1s_2) = a$,  
We know that $t_3'$ is uniquely determined by its value on $1,s_1,s_1s_2$. It has to be a linear
combination of $t_3^H$, $t_3^K$ and $t_3^{\dagger\dagger}$ if and only if 
$$
\Delta = \left| \begin{array}{ccc}  a^3 & a^2 & a \\ \delta_H^2 & \delta_H & 1 \\ \delta_K^2 & \delta_K & 1 \\ \end{array} \right| 
=a  \left| \begin{array}{ccc}  a^2 & a & 1 \\ \delta_H^2 & \delta_H & 1 \\ \delta_K^2 & \delta_K & 1 \\ \end{array} \right| 
$$
is invertible. Since $\delta_H = (y+1)/x = 2/x$, $\delta_K = \frac{2}{x} - a$, we get
$$
\Delta = \frac{2}{x^2}(2a-x)(a-x)
$$
whence the conclusion of (ii).

\end{proof}

We will show below (see corollary \ref{corclasstracesddagger}) that part (ii) actually holds true for every $n$, provided
that $x+2a$ is also assumed to be invertible.

We note that, when specialized to a field, $x=a, y=1,a^2=1$ imply $x \in \{-1,1 \}$, hence $\{ b,c \} = \{-j, -j^2 \}$ with $j = \exp(\frac{2 \ii \pi}{3})$
if $a=1$, and $\{ b,c \} = \{ j, j^2 \}$ if $a=-1$ ; likewise,
$x=2a, y=1,a^2=1$ imply $x \in \{-2,2 \}$, hence $b=c$, and $b,c \in \{-1,1 \}$,
hence either $a=b=c=1$ or $a=b=c=-1$.
In these four cases, we have $a = \pm bc$, and a possible additional trace on $\WBMW_n$ cannot factorize through
$BMW_n^{\pm}$, as is immediately checked on the 12-terms relation (note that,
by substracting a linear combination of the two ordinary traces,
we can assume in the first two cases that this trace satisfies $t_3(1) = t_3(s_1) = 0$, $t_3(s_1s_2) = 1$, while in the
latter cases we can assume $t_3(1) = 1$, $t_3(s_1) = t_3(s_1s_2)=0$).

\section{A central extension of $BMW$}

\subsection{Definition}
We define an algebra $F_n$ over $A = \Q[a,x,x^{-1}]/(a^2 = 1)$ by generators
$s_1,\dots,s_{n-1}$, $e_1,\dots,e_{n-1},C$ and relations
\begin{enumerate}
\item $s_i s_{i+1} s_i  = s_{i+1} s_i s_{i+1}$, $s_i s_j = s_j s_j$
\item $(s_i -a)(s_i^2 -x s_i+1)=0$
\item $e_i = a \left( \frac{ s_i^{-1} + s_i}{x} - 1 \right)$
\item $s_i e_i = a e_i$
\item $e_i s_{i+1} e_i =  e_i+C$
\item $e_i s_{i-1} e_i =  e_i+C$
\item $e_i s_{i+1}^{-1} e_i =  e_i+C$
\item $e_i s_{i-1}^{-1} e_i =  e_i+C$
\item $s_i C = C s_i = a C$.
\end{enumerate}
Letting  $\tilde{\delta} = 2 - ax $, we have $e_i^2= \tilde{\delta}x^{-1} e_i$. 
Immediate consequences of these relations are
$s_i^{-1} C = C s_i^{-1} = a C$, $e_i C = C e_i = x^{-1} \tilde{\delta}C$,
$C^2 = (x^{-2} \tilde{\delta}^2 a - x^{-1} \tilde{\delta})C = x^{-1} \tilde{\delta}(a x^{-1} \tilde{\delta} - 1) C
= 2x^{-2} \tilde{\delta}( a-x) C$. Note
that, in the specializations $x = a$ and $x = 2a$, we have $C^2 = 0$.

The following is easily checked
\begin{prop} \label{propbasiqueFn}
\begin{enumerate}
\item The $S^{\dagger\dagger}$-algebra $F_n \otimes_A S^{\dagger\dagger}$ is a quotient of $H_n \otimes_R S^{\dagger\dagger}$ through $s_i \mapsto s_i$ and $A \to S^{\dagger\dagger}$ being given by $x \mapsto b+c$. This quotient factorizes through $BMW_n^{\dagger\dagger}$.
\item There is a surjective morphism of $S^{\dagger\dagger}_{\pm}$-algebras $F_n \otimes_A S^{\dagger\dagger}_{\pm} \onto BMW_n^{\pm} \otimes_{S_{\pm}} S^{\dagger\dagger}_{\pm}$ satisfying
$s_i \mapsto s_i$, $e_i \mapsto e_i$, $C \mapsto 0$. Its kernel is the linear span of $C$.
\item The automorphism and antiautomorphism of $A$-algebras of the group algebra $A B_n$
defined by $s_i \mapsto s_i^{-1}$ induce an automorphism and a antiautomorphism of $F_n$.
\item Every Markov trace $(t_n)$ factorizing through $F_n$ is uniquely determined by $t_3(1)$, $t_3(s_1)$
and $t_3(s_1s_2)$.
\end{enumerate}
\end{prop}

\begin{proof}
We start with (i). Only the fact that we have a factorization through $BMW_n^{\dagger\dagger}$ requires a justification.
According to proposition \ref{proppresWBMWidemp}, it is sufficient to show that $s_1 e_2s_1$ and $s_2^{-1} e_1 s_2^{-1}$
are mapped to the same element. Using the same computations as in the proof of proposition \ref{propWBMWetendbien}
we easily get that both are sent to $x e_1 s_2^{-1} + ax C + x e_2 s_1 - s_2 e_1 s_2^{-1}$, and this proves the claim. 
(ii) is easy, because the linear span of $C$ is clearly a two-sided ideal of $F_n$. (iii) is easily checked from the defining relations of $F_n$.
We now prove (iv). From the arguments of proposition \ref{prop:unicitetracegenplus}  one
easily gets that such a Markov trace is uniquely determined by $t_3$ together with the collection of the $t_n(C)$,
since $C = e_{n-1} s_n e_{n-1} - e_{n-1}$.  Because $Cs_n = aC$ we get $a t_{n+1}(C) = t_n(C)$, whence the Markov trace is uniquely determined by $t_3$, and
therefore by $t_3(1)$, $t_3(s_1)$ together with $t_3(s_1s_2)$ by proposition \ref{prop:MT3bmw}.

\end{proof}

\begin{cor} If $(t_n)$ is a Markov trace factoring through $F_n$, then the associated link invariant
does not distinguish mirrors and does not detect non-invertible links.
\end{cor}
\begin{proof} This follows from the items (iii) and (iv) of the proposition : the mirror of the closed braid $\hat{\beta}$
is the closure of the image of $\beta$ under the automorphism $s_i \mapsto s_i^{-1}$,  the mirror of the inverse is the closure of the image of $\beta$
under the antiautomorphism $s_i \mapsto s_i^{-1}$. Since $t_3(s_1) = t_3(s_1^{-1})$,
$t_3(1) = t_3(1^{-1})$, $t_3(s_1s_2) = t_3(s_1^{-1} s_2^{-1})$, such a Markov trace coincides by (iv) with its 
composite with the (anti-)automorphisms defined in (iii), whence the conclusion. 
\end{proof}

\begin{prop} \label{proprelrecutraceFn} If $(t_n)$ is a Markov trace factoring through $F_n$, then
\begin{enumerate}
\item $\forall n \geq 3 \ \ t_{n+1}(C) = a t_n(C)$
\item $\forall n \geq 2\ \ t_{n+1}(C) = a t_{n+1}(1) - a x^{-1}(\tilde{\delta} + 2)t_n(1) +2a \tilde{\delta} x^{-2} t_{n-1}(1)$
\item For all $n \geq 1$,
$$t_{n+3}(1)=   \left( a+\frac{\tilde{\delta} + 2}{x}\right)t_{n+2}(1) -\frac{1}{x} \left(2 \frac{\tilde{\delta}}{x}+a(\tilde{\delta} + 2) \right) t_{n+1}(1) 
 +2a \frac{\tilde{\delta}}{x^2} t_{n}(1)$$
\end{enumerate}
\end{prop}
\begin{proof} (i) is trivially deduced from $s_n C = a C$ and the Markov property.
From $C = e_{n-1} s_n e_{n-1} - e_{n-1}$ we get
$t_{n+1}(C) = t_{n+1}(e_{n-1}^2 s_n) - t_3(e_{n-1}) = \tilde{\delta}x^{-1}t_{n+1}(e_{n-1} s_n) - t_{n+1}(e_{n-1})
= \tilde{\delta}x^{-1}t_n(e_{n-1}) - t_{n+1}(e_n)$ by the Markov property and because $e_{n-1}$ and $e_n$ are conjugates.
Expanding $e_i = (a/x)(s_i + s_i^{-1}) - a$ and using the Markov property again we
get
$t_{n+1}(C) = \tilde{\delta}x^{-1}(2a/x t_{n-1}(1) - a t_n(1)) - (2a/x t_n(1) - a t_{n+1}(1)) =
a t_{n+1}(1) - a x^{-1}(\tilde{\delta} + 2)t_n(1) +2a \tilde{\delta} x^{-2} t_{n-1}(1)$, namely (ii).
By (i) we know $t_{n+1}(C) = a t_{n+2}(C)$ hence, by (ii), we get
$t_{n+1}(C) = a t_{n+2}(C) =  t_{n+2}(1) -  x^{-1}(\tilde{\delta} + 2)t_{n+1}(1) +2 \tilde{\delta} x^{-2} t_{n}(1)$
 hence, again by (ii),  $t_{n+2}(1) -  x^{-1}(\tilde{\delta} + 2)t_{n+1}(1) +2 \tilde{\delta} x^{-2} t_{n}(1) = 
a t_{n+1}(1) - a x^{-1}(\tilde{\delta} + 2)t_n(1) +2a \tilde{\delta} x^{-2} t_{n-1}(1)$
hence 
$$t_{n+2}(1)=   \left( a+\frac{\tilde{\delta} + 2}{x}\right)t_{n+1}(1) -\frac{1}{x} \left(2 \frac{\tilde{\delta}}{x}+a(\tilde{\delta} + 2) \right) t_{n}(1) 
 +2a \frac{\tilde{\delta}}{x^2} t_{n-1}(1)$$
which proves (iii).
\end{proof}

At this stage, $C$ could well be $0$. We now prove that this is not the case.
\subsection{A genuine extension of $BMW_n^{\pm}$}

Using the abelianization morphism $B_n \onto \Z$ we can define a 3-dimensional $H_n \otimes_R S^{\dagger \dagger}$-module
by
$$
s_i \mapsto \begin{pmatrix}  a & 1 & 0 \\ 0 & b & 1 \\ 0 & 0 & b^{-1} \end{pmatrix}
$$
It is easily checked that $\mathcal{R}_1$ acts by $0$ in this module, hence we get a 3-dimensional
$BMW_n^{\dagger\dagger}$-module. We get that $e_i = a((s_i + s_i^{-1})/x - 1)$
is mapped to
$$
\left[ \begin {array}{ccc} {\frac { \left( ab-1 \right)  \left( -b+a
 \right) }{{b}^{2}+1}}&{\frac {ab-1}{{b}^{2}+1}}&{\frac {b}{{b}^{2}+1}
}\\ \noalign{\medskip}0&0&0\\ \noalign{\medskip}0&0&0\end {array}
 \right] 
$$
while $e_i s_{i+1} e_i - e_i$ and $e_i s_{i+1}^{-1} e_i - e_i$ are both mapped
to
$$
2 \frac{ab-b^2-1}{(b^2+1)^2}
\begin{pmatrix} (ab-1)(a-b) & ab-1 & b \\ 0 & 0 & 0 \\ 0 & 0 & 0
\end{pmatrix}
$$
{}
This proves that this module induces a $F_n\otimes_A S^{\dagger\dagger}$-module,
and therefore a $F_n\otimes_A S^{\dagger\dagger}_{\pm}$-module,
which do not factorize through $BMW_n^{\pm}$.
It follows that $BMW_n^{\dagger\dagger}$ and $F_n$ are genuine extensions of the Birman-Wenzl-Murakami algebra.
We need to prove a similar result for the specializations appearing in the previous section.
However, for one of them the above argument does not work, because it corresponds
to a root of $ab-b^2-1$. 
Nevertheless, we know by proposition  \ref{propstructgenWBMW} (iii) that $BMW_n^{\dagger\dagger}$ is finitely
generated as a $S^{\dagger\dagger}$-module. Because of this, the dimension of every specialization
is at least the dimension of $BMW_n^{\dagger\dagger}$ over the field of fractions of $S^{\dagger\dagger}$
(this classical fact follows for instance from Nakayama's lemma, by replacing $S^{\dagger\dagger}$ by its
localization at the defining ideal of the specialization). This proves the following.

\begin{prop} \label{propFnVraieextension} For every morphism $\la : S^{\dagger \dagger} \to \C$,
$BMW_n^{\dagger \dagger} \otimes_{\la} \C$ and $(F_n\otimes_A S^{\dagger\dagger})\otimes_{\la} \C$ have dimension at least $1 + \dim BMW_n^+$.
\end{prop}

\begin{cor} For every morphism $\la : S^{\dagger \dagger} \into \C$,
the image of $C$ inside 
$(F_n\otimes_A S^{\dagger\dagger})\otimes_{\la} \C$ 
is nonzero.
\end{cor}

\begin{cor} \label{cordimFn}  The $A$-module $F_n$ is free of rank $1+\dim BMW_n^{\pm}$
\end{cor}
\begin{proof}
Let us denote $A_{\pm} = \Q[a,x,x^{-1}]/(a\mp 1) \simeq \Q[x,x^{-1}]$. By the natural $A$-module
decomposition $A \simeq A_+ \oplus A_-$ it is enough to prove that $F_n^{\pm} = F_n \otimes_A A_{\pm}$
is a free $A_{\pm}$-module of rank $1+N$ with $N = \dim BMW_n^{\pm}$.  Now notice that $BMW_n^{\pm} \otimes_{S_{\pm}} S^{\dagger\dagger}_{\pm}$ is actually defined by a presentation with coefficients in $A_{\pm}$, and that
the corresponding $A_{\pm}$-form is free of rank $N$ by \cite{MORTONWASSERMANN}. Therefore, as in proposition \ref{propbasiqueFn} (ii),
we get a surjective $A_{\pm}$-morphisme $F_n^{\pm} \onto A_{\pm}^N$ by mapping $C$ to $0$. Letting $s_{\pm}$
denote a section of this morphism, we get a surjective morphism $u_{\pm} : A_{\pm}^{1+N} \simeq (A_{\pm} C)\oplus A_{\pm}^N \to
F_n^{\pm}$ by $(\la C,m) \mapsto \la C + s_{\pm}(m)$. Letting $K_{\pm}$ denote the quotient field
of $A_{\pm}$, it follows that $u_{\pm} \otimes_{A_{\pm}} K_{\pm}$ is surjective. But proposition \ref{propFnVraieextension}
implies that $F_n^{\pm} \otimes_{A_{\pm}} K_{\pm}$ has dimension $1+N$. Therefore
the source and target of $u_{\pm} \otimes_{A_{\pm}} K_{\pm}$ have the same dimension hence $u_{\pm} \otimes K_{\pm}$
is injective. Since the source of $u_{\pm}$ is a free module this implies that $u_{\pm}$ is injective and this proves the claim.

\end{proof}

\begin{cor} \label{corFnFnp1}  The natural algebra morphism $F_n \to F_{n+1}$ is injective.
\end{cor}
\begin{proof}
As before, using the notations of the previous proof, it is equivalent to show that the natural maps $F_n^{\pm} \to F_{n+1}^{\pm}$ are
injective. This is true because the following diagram of horizontal short exact sequences is commutative, and because its two extremal vertical arrows
are known to be injective.
$$
\xymatrix{
0 \ar[r] & A_{\pm} C \ar[r]\ar[d] & F_n^{\pm} \ar[r] \ar[d]& BMW_n^{\pm} \ar[r] \ar[d]& 0 \\
0 \ar[r] & A_{\pm} C \ar[r] & F_{n+1}^{\pm} \ar[r] & BMW_{n+1}^{\pm} \ar[r] & 0 \\
}
$$
\end{proof}
\subsection{The algebra $F_n$ as a specialization of $\WBMW_n$}

The goal of this section is to prove the following theorem.

\begin{theor} \label{theoisomddaggerFn} The morphism of $S^{\dagger\dagger}$-algebras $BMW_n^{\dagger\dagger} \to F_n \otimes_A S^{\dagger\dagger}$
induces an isomorphism after tensorisation by $S^{\dagger\dagger}[(x+2a)^{-1}]$.
\end{theor}

We define $\mathcal{S}_ i = e_i s_{i+1} e_i - e_i$, 
$\mathcal{S}'_ i = e_i s_{i+1}^{-1} e_i - e_i \in H_n \otimes_R S^{\dagger\dagger}$, and also
$\widehat{\mathcal{R}_i} = (s_i s_{i+1} s_i)\mathcal{R}_i (s_i s_{i+1} s_i)^{-1}$,
and similarly $\widehat{\mathcal{S}_i} = (s_i s_{i+1} s_i) (\mathcal{S}_i) (s_i s_{i+1} s_i)^{-1}$.
The two formulas below hold inside $H_3 \otimes S^{\dagger\dagger}$ and can be checked
computationally by using the morphism $\Phi_{H_3}$, as we did for lemma \ref{lemRimpliesS} (these
formulas were found by a similar procedure, too).

\begin{equation}
\mathcal{S}_1- \mathcal{S}'_1 
= \frac{1}{x} \mathcal{R}_1 s_2 -\frac{1}{x^2}s_1^{-1} \mathcal{R}_1 s_2 
- \frac{1}{x^2}s_1 \mathcal{R}_1 s_2 = \frac{1}{x^2}(x - s_1^{-1} - s_1) \mathcal{R}_1 s_2 
= \frac{1}{(b+c)^2} (s_1 -b)(s_1-c) s_1^{-1} \mathcal{R}_1 s_2  \label{eqssprime}
\end{equation}
\begin{equation}
2x^2 \left(  \mathcal{S}_{1} - \widehat{\mathcal{S}}_{1}\right)=
 (x- s_1^{-1} - s_1) \mathcal{R}_1 s_2 -  \mathcal{R}_1 + (-x +  s_2^{-1} + s_2) \widehat{\mathcal{R}_1} s_1
+  \widehat{\mathcal{R}_1} \label{eqssym}
\end{equation}

We now want to show that $\mathcal{S}_1 - \mathcal{S}_2$ also belongs to the ideal generated by $\mathcal{R}_1$. For this
we need to work inside $H_4 \otimes S^{\dagger\dagger}$. Since the computations become quite complicated, we specialize
at $a=1$. There is no loss of generality in doing this, as we justify it now. The natural decomposition $S^{\dagger\dagger}
= S^{\dagger\dagger}_+ \oplus S^{\dagger\dagger}_-$ induces a $\Q$-vector space decomposition $H_n \otimes_R S^{\dagger\dagger} = H_n^+ \oplus
H_n^-$ with projectors $p_+,p_-$ given by the multiplication by $(a-1)/2$ and $(a+1)/2$. It is straightforward to
check that the involutive automorphism $E^{\dagger}$ of $H_n \otimes_R S^{\dagger\dagger}$ induced by
$\hat{E}$ and $\eta$ (see the notations of proposition \ref{propdefE} and before) exchanges $H_n^+$ and $H_n^-$, maps
$\mathcal{S}_i \mapsto -\mathcal{S}_i$, $\mathcal{R}_i \mapsto -\mathcal{R}_i$ and intertwines $p_+$ and $p_-$ up
to a sign (that is: $E^{\dagger} \circ p_+ = - p_- \circ E^{\dagger}$,
$E^{\dagger} \circ p_- = - p_+ \circ E^{\dagger}$). Because of this, any expression of $p_+(\mathcal{S}_1 - \mathcal{S}_2) = (\mathcal{S}_1)_+
- (\mathcal{S}_2)_+$ immediately yields an expression of $p_-(\mathcal{S}_1 - \mathcal{S}_2)$
and therefore of $\mathcal{S}_1 - \mathcal{S}_2$.

We now use the fact that $H_4$ is a free $R$-module of rank 648 in order to do explicit computations. More precisely,
the computations are made as follows. First of all, we build a basis of $H_4$ as follows. We
consider the collection $\mathcal{B}_2$ of 27 words in $s_i, s_i^{-1}$ given in proposition 4.8 of \cite{CUBIC5}.
They form a basis of $H_4$ as a $H_3$-module. Together with the list of 24 words $\mathcal{B}_1$ 
given by proposition \ref{prop:H3libre} (ii), which induces a basis of $H_3$, we get a
collection $\mathcal{B}_3 = \{ g_1 g_2 ; (g_1, g_2) \in \mathcal{B}_1 \times \mathcal{B}_2\}$
of $24 \times 27 = 648$ words which induces a basis of $H_4$. From this and the implicit rewriting rules of \cite{CUBIC5} we
build an explicit regular representation $H_4 \into Mat_{648}(R)$ and therefore an injective map
 $\Phi_{H_4} : H_4 \otimes_R S_+^{\dagger \dagger} \to Mat_{648}(S_+^{\dagger \dagger})$, that we use in order to check equalities.

Letting $(\mathcal{S}_1)_{\pm} = \mathcal{S}_{\pm}$,
we let $(\mathcal{S}_2)_{\pm} = (s_1 s_2)( \mathcal{S}_1)_{\pm} (s_1 s_2)^{-1}$.
Inside $H_4 \otimes_R R/(a-1,y-1)$ we find that $2 (x+2)^2 x^4 ((\mathcal{S}_1)_+ -( \mathcal{S}_2)_+)$
can be expressed as a sum of 161 terms obviously belonging to $I_4$, see Figures \ref{S12a}, \ref{S12b}, \ref{S12c}.

These terms were found as follows. For computational reasons (and the limited power of the computers we have at disposal) it is too difficult to
compute the linear span of $I_4$ inside the field of fractions $\Q(x)$ or $\Q(b)$, so we need to circumvent this obstacle by computing inside specialisations in $x$.
For some specific value of $x$ we get a basis $\mathcal{B}^1$ of the linear span of the image of $I_4$ inside $\Q^{648}$.
The one we get is made of terms of the form $g \mathcal{R}_1 g'$ where $g,g'$ are products of elements $s_i,s_i^{-1}$
(chosen inside the basis of $H_4$ mentionned above), or $g \mathcal{R}_2 g'$ or $g \widehat{\mathcal{R}_1} g'$ or $g \widehat{\mathcal{R}_2} g'$.
The same elements form a basis of the specializations of $H_4$ for infinitely many values of $x$. We chose a number of values for which
we got an expression of $\mathcal{S}_1 - \mathcal{S}_2$ as linear combination of the elements of $\mathcal{B}^1$. Assuming that these
coefficients should be rational fractions in $x$ whose denominators have low degree, we get these rational fractions by interpolation. We then check that the
corresponding equality is correct by direct computation inside $H_4 \otimes \Q(x)$. It so happens that the choices we made in this process provide an expression of $(x-1)(x+2)^2(x^2+x-1)x^4 (\mathcal{S}_1 - \mathcal{S}_2)$. Recall that, in order to deal with the odd cases of the previous
section, we need to specialize at $x=1$.  For this alone, we need to start again with this specialization, and we get this time, as a linear combinations of another basis $\mathcal{B}^2$, a polynomial expression of $(x+1)(x+2)^2x^4 (\mathcal{S}_1 - \mathcal{S}_2)$. By Bezout theorem both results combined provide an expression of $2 (x+2)^2 x^4 ((\mathcal{S}_1)_+ -( \mathcal{S}_2)_+)$
as a linear combination of $\mathcal{B}^1 \cup \mathcal{B}^2$, and this is the result that is expressed here (the cardinality of
$\mathcal{B}^1 \cup \mathcal{B}^2$ is 161).

Needless to say, one could have hoped to get a nicer expression. Unfortunately we failed to find one.

The $x^4$ factor in this expression prevents the specialization $x=0$, which is to be expected. The $x+2$ factor prevents the specialization
$x=-2a$, which was not expected, but is explained by the proposition below. This proposition implies that the morphism under consideration
does not induce an isomorphism under this specialization, since the algebras $BMW_4^{\pm}$ have dimension 105.
Note that $(S^{\dagger\dagger}_{\pm})/(x+2a) \simeq \Q[b]/(b+ b^{-1} \pm 2) \simeq \Q[b]/(b \pm 1)^2$.

\begin{prop} \label{propdimBMWm2a}
The $\Q$-algebras $BMW_4^{\dagger\dagger} \otimes_{S^{\dagger\dagger}} \Q[b]/(b \pm 1)$ both have dimension 115.
\end{prop}
\begin{proof}
Direct computations prove that their defining ideals inside $H_4 \otimes_R R/(a\mp 1,b\pm 1, c \pm 1) \simeq \Q^{648}$ have dimension $533$,
hence the corresponding quotients have dimension $648-533 = 115$.
\end{proof}
\begin{figure}
$$
\begin{array}{ll}
 & (x+2)(x^7+4x^6-x^5-12x^4-5x^3+9x^2-1)x\mathcal{R}_1+ (-x^7-4x^6+8x^4+2x^3-8x^2-2x+6)\mathcal{R}_1s_3^{-1}\\ +&
  (2x+4)\mathcal{R}_1s_3^{-1}s_2^{-1}s_1^{-1}
  +   (x+2)(x^2-x-1)(x^5+3x^4-3x^3-5x^2+3x-2)x\mathcal{R}_1s_1 \\
+&  (x+2)(x^2-x-1)x^2\mathcal{R}_1s_1s_3^{-1}+ (-x-2)(x^2-x-1)(x^3+3x^2-2x-3)x^2\mathcal{R}_1s_1s_3 \\+&
  (-x^8-5x^7-3x^6+13x^5+17x^4-4x^3-10x^2+6x+6)\mathcal{R}_1s_3+(2x+4)\mathcal{R}_1s_3s_2s_1\\+& (-x-2)(x^3+2x^2+2x-1)x\mathcal{R}_2+ 
  (x^5+4x^4+3x^3-3x^2-2x-2)\mathcal{R}_2s_1^{-1}\\+& (x+1)(x+2)(x^2-x-1)(x^4+4x^3+2x^2-2x-4)x\mathcal{R}_2s_1^{-1}s_3+ 
  (x^5+4x^4+3x^3-3x^2-2x-2)\mathcal{R}_2s_1\\ +& (x+1)(x+2)(x^2-x-1)(x^4+4x^3+2x^2-2x-4)x\mathcal{R}_2s_1s_3\\ + & 
  (-x-1)(x+2)(x^2-x-1)(x^4+3x^3-2x^2-x+2)x\mathcal{R}_2s_2\\ + & (-x-2)(x^5+2x^4-6x^3-7x^2+2x+9)x\widehat{\mathcal{R}}_1+
  (2x^7+9x^6+2x^5-26x^4-29x^3-x^2+12x+4)\widehat{\mathcal{R}}_1s_3^{-1}\\ + &  (-2x-4)\widehat{\mathcal{R}}_1s_3^{-1}s_2^{-1}s_1^{-1}+ 
  (x^8+5x^7+4x^6-11x^5-25x^4-14x^3+9x^2+14x+4)\widehat{\mathcal{R}}_1s_3\\+&( -2x-4)\widehat{\mathcal{R}}_1s_3s_2s_1+
  (-x-1)(x+2)(x^6+3x^5-4x^4-9x^3+2x^2+9x+1)x\widehat{\mathcal{R}}_2\\ + &
  (x^9+7x^8+14x^7-4x^6-47x^5-49x^4+5x^3+34x^2+10x-4)\widehat{\mathcal{R}}_2s_1^{-1}\\+& 
  (x^9+7x^8+14x^7-3x^6-44x^5-52x^4+32x^2+12x-4)\widehat{\mathcal{R}}_2s_1\\ + & 
  (-x^8-5x^7-2x^6+17x^5+25x^4+3x^3-24x^2-20x-2s_3^{-1})\mathcal{R}_1+ (2x+4)(x^2+x+2)s_3^{-1}\mathcal{R}_1s_3^{-1}\\ + &
  (-x^6-3x^5-4x^4-2x^3+7x^2+10x+4)s_3^{-1}\mathcal{R}_1s_1^{-1}\\+& (-x^8-5x^7-x^6+14x^5+9x^4-11x^3-10x^2+4x+4)s_3^{-1}\mathcal{R}_1s_1\\+& 
  (x-1)(x+2)(x^2-x-1)(x^2+4x+2)xs_3^{-1}\mathcal{R}_1s_1s_3\\ +& (-x^8-4x^7+x^6+12x^5+12x^4+5x^3-4x^2-12x-10)s_3^{-1}\mathcal{R}_1s_2\\+& 
  (x+1)(x^5+3x^4-4x^3-9x^2+2x+10)xs_3^{-1}\mathcal{R}_1s_3+ (2x+4)s_3^{-1}\mathcal{R}_1s_3s_2^{-1} \\ + & 
  (-x^8-5x^7-3x^6+13x^5+25x^4+15x^3-5x^2-12x-4)s_3^{-1}\widehat{\mathcal{R}}_1 \\ + &
  (-x-2)(x^5+3x^4-3x^3-7x^2-4x+2)s_3^{-1}\widehat{\mathcal{R}}_1s_3^{-1}+
  (x+2)(x^4+x^3+2x^2-3)xs_3^{-1}\widehat{\mathcal{R}}_1s_2^{-1}\\ + & 
  (x^7+4x^6-2x^5-14x^4-15x^3-2x^2+14x+8)s_3^{-1}\widehat{\mathcal{R}}_1s_1^{-1}\\+& 
  (x^5-4x^4-12x^3-4x^2+10x+6)s_3^{-1}\widehat{\mathcal{R}}_1s_1\\+&( -x^7-4x^6+10x^4+12x^3+3x^2-4x-6)s_3^{-1}\widehat{\mathcal{R}}_1s_3+ 
  (-x-1)(x+2)(x^2+x-1)xs_3^{-1}\widehat{\mathcal{R}}_2\\+& (x^5+4x^4+3x^3-3x^2-2x-2)s_3^{-1}\widehat{\mathcal{R}}_2s_1^{-1}+ 
  (x^5+4x^4+3x^3-3x^2-2x-2)s_3^{-1}\widehat{\mathcal{R}}_2s_1\\+& (x-1)(x^2+4x+2)s_3^{-1}s_2^{-1}\widehat{\mathcal{R}}_1s_3+
  (2x+4)s_3^{-1}s_2^{-1}s_1\mathcal{R}_2+ (-2x-4)s_3^{-1}s_2^{-1}s_1\widehat{\mathcal{R}}_2\\+& 
  (-x-2)(x^5+3x^4-3x^3-7x^2-4x+2)s_3^{-1}s_2\mathcal{R}_1s_3^{-1}+( -x^5-3x^4+4x^3+10x^2+4x-4)s_3^{-1}s_2\mathcal{R}_1s_3\\+&
  (x+1)(x^5+3x^4-4x^3-8x^2+6x+11)x^2s_2^{-1}\mathcal{R}_2+( -x^5-4x^4-5x^3-x^2+2x-2)s_2^{-1}\mathcal{R}_2s_1^{-1}\\+&
  (-x^7-4x^6-x^5+6x^4+6x^3-x^2-2)s_2^{-1}\mathcal{R}_2s_1+ (x+2)(x^4+3x^3-6x+1)xs_2^{-1}\widehat{\mathcal{R}}_1\\+&
  (-x^4-3x^3+3x+2)xs_2^{-1}\widehat{\mathcal{R}}_1s_3^{-1}+ (-x^4-4x^3-3x^2+5x+4)xs_2^{-1}\widehat{\mathcal{R}}_1s_3\\+&
  (x+1)(x^5+2x^4-6x^3-x^2+10x+6)xs_2^{-1}\widehat{\mathcal{R}}_2\\+& 
  (-x^8-5x^7-4x^6+10x^5+18x^4+6x^3-7x^2-4x-2)s_2^{-1}\widehat{\mathcal{R}}_2s_1^{-1}\\+& 
 ( -x^8-5x^7-4x^6+10x^5+18x^4+6x^3-7x^2-4x-2)s_2^{-1}\widehat{\mathcal{R}}_2s_1+( -2x-4)s_2^{-1}s_1\mathcal{R}_2s_1^{-1}\\+&
  (2x+4)s_2^{-1}s_3\widehat{\mathcal{R}}_1s_3^{-1}+ (-x^7-4x^6+8x^4-7x^2-x-4)xs_1^{-1}\mathcal{R}_1\\+& 
  (-x^8-3x^7+5x^6+13x^5-9x^3+2x^2+4x-2)s_1^{-1}\mathcal{R}_1s_3^{-1}+( x^7+4x^6-9x^4-3x^3+4x^2-2)s_1^{-1}\mathcal{R}_1s_3\\+& 
  (-x^7-3x^6-2x^5+2x^4+7x^3+2x^2-4x+2)s_1^{-1}\mathcal{R}_2\\+& (x^7+4x^6-2x^5-10x^4+x^3+2x^2-7x-6)xs_1^{-1}\mathcal{R}_2s_2^{-1}\\+&
  (x^8+4x^7-2x^6-10x^5+x^4+2x^3-7x^2-4x+4)xs_1^{-1}\mathcal{R}_2s_2\\+& (-x^7-4x^6+2x^5+10x^4-x^3-2x^2+7x+6)xs_1^{-1}\mathcal{R}_2s_3\\+&
  (-x+1)(x+1)(x^4+2x^3-5x^2+2x+14)xs_1^{-1}\widehat{\mathcal{R}}_1\\+&
  (-x^7-4x^6+11x^4+15x^3+5x^2-2x-6)s_1^{-1}\widehat{\mathcal{R}}_1s_3^{-1}\\+&
  (-x-1)(x^6+2x^5-5x^4-3x^3-5x^2+2x+6)s_1^{-1}\widehat{\mathcal{R}}_1s_3 \\ + & \dots
  \end{array}
  $$
   \caption{For $a=1$, $2 (x+2)^2 x^4 (\mathcal{S}_1 - \mathcal{S}_2)$ to be continued.}
   \label{S12a}
  \end{figure}
  \begin{figure}
  $$
  \begin{array}{ll}& \dots\\
  +& 
  (-x^7-4x^6+2x^5+10x^4-x^3+15x+14)xs_1^{-1}\widehat{\mathcal{R}}_2\\+& 
  (-x^7-3x^6-2x^5+2x^4+7x^3+2x^2-4x+2)s_1^{-1}\widehat{\mathcal{R}}_2s_3^{-1} \\+& 
  (2x^6+7x^5-6x^4-11x^3+2x^2-10x-4)xs_1^{-1}\widehat{\mathcal{R}}_2s_2^{-1}\\+&
  (-x+1)(x+1)^2(x^2-x-1)(x^2+4x+2)s_1^{-1}\widehat{\mathcal{R}}_2s_1^{-1}\\+&
  (-x^7-4x^6+10x^4+11x^3-2x^2-6x-4)s_1^{-1}\widehat{\mathcal{R}}_2s_1+( -2x-4)s_1^{-1}\widehat{\mathcal{R}}_2s_1s_2^{-1}\\+&
  (2x^8+8x^7-3x^6-24x^5-10x^4+12x^3+6x^2-2x-2)s_1^{-1}\widehat{\mathcal{R}}_2s_2\\+&
  (x^7+4x^6-x^5-13x^4-10x^3+x^2-6)s_1^{-1}s_3^{-1}\mathcal{R}_1s_3^{-1}+ (-x^3-5x^2-2x-2s_1^{-1})s_3^{-1}\mathcal{R}_1s_3\\+&
  (x^6+4x^5-10x^3-13x^2-4x-2)s_1^{-1}s_3^{-1}\widehat{\mathcal{R}}_1s_3^{-1}\\+&
  (x^6+3x^5-3x^4-7x^3-6x^2-2x+4)s_1^{-1}s_3^{-1}\widehat{\mathcal{R}}_1s_3\\+& 
  (-x^7-4x^6+x^5+13x^4+10x^3-x^2+6)s_1^{-1}s_2^{-1}\mathcal{R}_2s_1^{-1}\\+& (-x^7-3x^6+5x^5+13x^4-12x^2+8)s_1^{-1}s_2^{-1}\mathcal{R}_2s_1
+(  2x+4)s_1^{-1}s_2^{-1}s_3^{-1}\mathcal{R}_1+( -2x-4)s_1^{-1}s_2^{-1}s_3^{-1}\widehat{\mathcal{R}}_1
\\+&
  (2x+4)s_1^{-1}s_2^{-1}s_3\mathcal{R}_1+ (-2x-4)s_1^{-1}s_2^{-1}s_3\widehat{\mathcal{R}}_1\\+&
  (-x-2)(x^2-x-1)(x^5+3x^4-4x^3-2x^2+3x-4)xs_1^{-1}s_2\mathcal{R}_2\\+&
  (x+1)(x^5+3x^4-4x^3-10x^2-4x+4)xs_1^{-1}s_3\mathcal{R}_1s_3^{-1}+ (-x-1)(x+2)(x^2+2x-2)s_1^{-1}s_3\mathcal{R}_1s_3\\+&
  (-x+1)(x+1)^2(x^2-x-1)(x^2+4x+2)s_1^{-1}s_3\mathcal{R}_2s_1^{-1}\\+&( -x^7-4x^6+10x^4+11x^3-2x^2-6x-4)s_1^{-1}s_3\mathcal{R}_2s_1\\+& 
  (x+1)(x^6+4x^5-x^4-13x^3-9x^2+2x+6)s_1^{-1}s_3\widehat{\mathcal{R}}_1s_3^{-1}\\+& 
  (x+2)(x^6+3x^5-4x^4-9x^3-x^2+2x+6)s_1^{-1}s_3\widehat{\mathcal{R}}_1s_3\\+&
  (-x^8-4x^7+2x^6+17x^5+5x^4-22x^3-18x^2+2x+10)xs_1\mathcal{R}_1\\+& (-x^7-4x^6-x^5+10x^4+18x^3+7x^2-2x-6)s_1\mathcal{R}_1s_3^{-1}\\+&
  (x^8+3x^7-5x^6-13x^5+5x^4+21x^3+9x^2-8x-6)s_1\mathcal{R}_1s_3\\+&(-x^6-2x^5-x^4-4x^3-8x^2-8x+2)s_1\mathcal{R}_2
\\+&  (x^5+3x^4+4x^3-3x^2-9x-2)xs_1\mathcal{R}_2s_2^{-1}+ (2x+4)s_1\mathcal{R}_2s_1^{-1}s_2\\+& 
  (x^7+5x^6+2x^5-9x^4-11x^3-4x^2-2)xs_1\mathcal{R}_2s_2\\+& (x^7+4x^6-2x^5-16x^4-12x^3+10x^2+17x+4)xs_1\mathcal{R}_2s_3\\+&
  (x^8+5x^7+4x^6-10x^5-23x^4-20x^3+3x^2+19x+12)xs_1\widehat{\mathcal{R}}_1\\+& 
  (-x^7-6x^6-8x^5+11x^4+34x^3+22x^2-5x-10)xs_1\widehat{\mathcal{R}}_1s_3^{-1}\\+&
  (-2x^7-9x^6-3x^5+24x^4+33x^3+7x^2-15x-8)xs_1\widehat{\mathcal{R}}_1s_3\\+&
  (x^7+4x^6-2x^5-16x^4-12x^3+10x^2+17x+4)xs_1\widehat{\mathcal{R}}_2\\+& (-x^6-2x^5-x^4-2x^3+2)s_1\widehat{\mathcal{R}}_2s_3^{-1}\\+&
  (-x^6-3x^5+6x^4+13x^3+3x^2-14x-10)xs_1\widehat{\mathcal{R}}_2s_2^{-1}\\+& 
  (-x^8-5x^7-3x^6+15x^5+24x^4-2x^3-20x^2-8x+8)s_1\widehat{\mathcal{R}}_2s_1^{-1}\\+&
 ( -x^8-5x^7-3x^6+14x^5+21x^4+x^3-15x^2-6x+6)s_1\widehat{\mathcal{R}}_2s_1\\+&
  (x^9+5x^8+3x^7-14x^6-22x^5+3x^4+20x^3+4x^2-8x-2)s_1\widehat{\mathcal{R}}_2s_2\\+&
  (x+2)(x^2-x-1)(x^5+5x^4+3x^3-2x^2-2x-2)s_1s_3^{-1}\mathcal{R}_1\\+& (x^6+4x^5-10x^3-13x^2-4x-2)s_1s_3^{-1}\mathcal{R}_1s_3^{-1}\\+&
  (-x^7-3x^6+5x^5+12x^4-5x^3-16x^2-6x+4)s_1s_3^{-1}\mathcal{R}_1s_3\\+&
  (x^7+5x^6+4x^5-10x^4-21x^3-11x^2+2x+6)s_1s_3^{-1}\widehat{\mathcal{R}}_1s_3^{-1}\\+&
  (2x+2)(x^2+4x+2)(x^4-x^3-3/2x^2+1)s_1s_3^{-1}\widehat{\mathcal{R}}_1s_3 \\+ & \dots
  \end{array}
  $$
   \caption{For $a=1$, $2 (x+2)^2 x^4 (\mathcal{S}_1 - \mathcal{S}_2)$, continued.}
      \label{S12b} 
  \end{figure}
\begin{figure}
  $$
  \begin{array}{ll}+& 
  (-x-1)(x^5+3x^4-3x^3-7x^2-6x+2)s_1s_2^{-1}\mathcal{R}_2s_1^{-1}+ (2x+4)xs_1s_2^{-1}\mathcal{R}_2s_1\\+&
  (x-1)(x+1)^2(x^2-x-1)(x^2+4x+2)s_1s_2^{-1}\widehat{\mathcal{R}}_2s_1^{-1}\\+&
  (x-1)(x+1)^2(x^2-x-1)(x^2+4x+2)s_1s_2^{-1}\widehat{\mathcal{R}}_2s_1\\+& (-x-2)(x^2-x-1)(x^4+5x^3+3x^2-x-4)xs_1s_2\mathcal{R}_2\\+& 
  (x-1)(x+1)^2(x^2-x-1)(x^2+4x+2)s_1s_2\mathcal{R}_2s_1^{-1}\\+& (x-1)(x+1)^2(x^2-x-1)(x^2+4x+2)s_1s_2\mathcal{R}_2s_1\\+&
  (-x+1)(x+1)^2(x^2-x-1)(x^2+4x+2)s_1s_2\widehat{\mathcal{R}}_2s_1^{-1}\\+&
  (-x+1)(x+1)^2(x^2-x-1)(x^2+4x+2)s_1s_2\widehat{\mathcal{R}}_2s_1+ (2x+4)s_1s_2s_3^{-1}\mathcal{R}_1\\+& 
  (-2x-4)s_1s_2s_3^{-1}\widehat{\mathcal{R}}_1+ (2x+4)s_1s_2s_3\mathcal{R}_1+( -2x-4)s_1s_2s_3\widehat{\mathcal{R}}_1\\+&
  (x+2)(x^2-x-1)(x^4+x^3-4x^2-2)s_1s_3\mathcal{R}_1\\+& (x+1)(x^6+4x^5-x^4-13x^3-9x^2+2x+6)s_1s_3\mathcal{R}_1s_3^{-1}\\+&
  (x+2)(x^5+2x^4-6x^3-5x^2+6)s_1s_3\mathcal{R}_1s_3\\+&( -x^8-5x^7-3x^6+15x^5+24x^4-2x^3-20x^2-8x+8)s_1s_3\mathcal{R}_2s_1^{-1}\\+&
  (-x^8-5x^7-3x^6+14x^5+21x^4+x^3-15x^2-6x+6)s_1s_3\mathcal{R}_2s_1\\+&
  (x+1)^2(x^2+4x+2)(x^4-x^3-x^2-x+1)s_1s_3\widehat{\mathcal{R}}_1s_3^{-1}\\+&
  (x+2)(x^6+4x^5-x^4-13x^3-9x^2+4x+8)xs_1s_3\widehat{\mathcal{R}}_1s_3\\+& (-x-2)(x^5+2x^4-6x^3-7x^2+4x+5)xs_2\mathcal{R}_1+
  (x^3-2x-2)(x^4+5x^3+5x^2-x-1)s_2\mathcal{R}_1s_3^{-1}\\+& (x^6+3x^5-4x^4-13x^3-6x^2+4x+2)s_2\mathcal{R}_1s_3\\+&
  (x+1)(x^7+4x^6-x^5-12x^4-5x^3+8x^2+6x+2)xs_2\mathcal{R}_2\\+& (-x^8-5x^7-4x^6+10x^5+18x^4+6x^3-7x^2-4x-2)s_2\mathcal{R}_2s_1^{-1}\\+&
  (-x^8-5x^7-4x^6+10x^5+18x^4+6x^3-7x^2-4x-2)s_2\mathcal{R}_2s_1\\+& (x^6+4x^5-x^4-12x^3-4x^2+13x+11)x^2s_2\widehat{\mathcal{R}}_2\\+&
  (-x^9-6x^8-9x^7+7x^6+34x^5+31x^4-9x^3-26x^2-6x+8)s_2\widehat{\mathcal{R}}_2s_1^{-1}\\+&
  (-x^9-6x^8-9x^7+7x^6+34x^5+31x^4-9x^3-26x^2-6x+8)s_2\widehat{\mathcal{R}}_2s_1+( -2x-4)s_2s_3^{-1}\mathcal{R}_1s_3\\+&
  (2x+4)s_2s_1^{-1}\widehat{\mathcal{R}}_2s_1+ (-x^7-3x^6+4x^5+14x^4+6x^3-10x^2-12x-2)s_3\mathcal{R}_1\\+&
  (-x-1)(x+1)(x^5+2x^4-6x^3-4x^2+4)s_3\mathcal{R}_1s_3^{-1}+ (-x-1)(x^5+2x^4-x^3-x^2-4x-4)s_3\mathcal{R}_1s_1^{-1}\\+&
  (-x^7-x^6+7x^5+2x^4-13x^3-10x^2+4x+4)s_3\mathcal{R}_1s_1\\+& (x-1)(x+2)(x^2-x-1)(x^2+4x+2)xs_3\mathcal{R}_1s_1s_3\\+&
  (-x+1)(x^8+6x^7+9x^6-6x^5-29x^4-28x^3-16x^2-14x-10)s_3\mathcal{R}_1s_2\\+& (x+2)(x^3+3x^2-6)s_3\mathcal{R}_1s_3\\+&
  (-x-1)(x+2)(x^6+3x^5-4x^4-9x^3+2x^2+7x+1)xs_3\mathcal{R}_2\\+& (x^8+6x^7+7x^6-13x^5-27x^4-9x^3+10x^2+2x-4)s_3\mathcal{R}_2s_1^{-1}\\+&
  (x^8+6x^7+8x^6-10x^5-30x^4-14x^3+8x^2+4x-4s_3)\mathcal{R}_2s_1\\+& (-x^7-3x^6+3x^5+11x^4+13x^3+7x^2-2x-4)s_3\widehat{\mathcal{R}}_1\\+& 
  (-x-1)(x^7+4x^6-10x^4-12x^3-3x^2+4x+6)s_3\widehat{\mathcal{R}}_1s_3^{-1}+ (-2x-4)s_3\widehat{\mathcal{R}}_1s_3^{-1}s_2\\+&
  (x^5+3x^4+x^3-x-8)xs_3\widehat{\mathcal{R}}_1s_2^{-1}+( x^6+x^5-7x^4-5x^3-4x^2+10x+8)s_3\widehat{\mathcal{R}}_1s_1^{-1}\\+&
  (-x^6-2x^5+2x^4-3x^3-4x^2+6x+6)s_3\widehat{\mathcal{R}}_1s_1\\+& (-x-1)(x+2)(x^6+3x^5-4x^4-9x^3+4x+4)s_3\widehat{\mathcal{R}}_1s_3\\+&
  (x-1)(x+1)(x^2+4x+2)s_3s_2^{-1}\widehat{\mathcal{R}}_1s_3^{-1}+ (x-1)(x+2)(x^2+4x+2)s_3s_2^{-1}\widehat{\mathcal{R}}_1s_3\\+&
  (-x-1)(x^5+3x^4-4x^3-10x^2-4x+4)s_3s_2\mathcal{R}_1s_3^{-1}+ (x+1)(x+2)(x^2+2x-2)s_3s_2\mathcal{R}_1s_3\\+&
  (2x+4)s_3s_2s_1^{-1}\mathcal{R}_2+ (-2x-4)s_3s_2s_1^{-1}\widehat{\mathcal{R}}_2.
 \end{array}
 $$
 \caption{For $a=1$, $2 (x+2)^2 x^4 (\mathcal{S}_1 - \mathcal{S}_2)$, last part.}
    \label{S12c}
 
 \end{figure}
 
We will elaborate a bit more on the case $x = -2a$ in section \ref{subsectm2a}. For now, let $\pi : BMW_n^{\dagger \dagger}\otimes S^{\dagger\dagger}[(x+2a)^{-1}] \onto F_n \otimes_A S^{\dagger\dagger}[(x+2a)^{-1}]$ be the obvious
 projection. We need to find a section $f$ such that $f \circ \pi = \mathrm{Id}$, and for this it is enough
 to check that the natural projection $f : H_n \otimes S^{\dagger\dagger}[(x+2a)^{-1}] \onto BMW_n^{\dagger \dagger}\otimes S^{\dagger\dagger}[(x+2a)^{-1}] $ factorizes through $F_n \otimes_A S^{\dagger\dagger}[(x+2a)^{-1}] $. Using the decomposition $S^{\dagger\dagger} = S^{\dagger\dagger}_+ \oplus
 S^{\dagger\dagger}_-$ and the Galois automorphism mapping $a \mapsto -a$,
   it is sufficient to check this for $f_+ : H_n \otimes S^{\dagger\dagger}_+[(x+2)^{-1}] \onto BMW_n^{\dagger \dagger}\otimes S_+^{\dagger\dagger}[(x+2)^{-1}] $. It is clear that the relators associated to $(i)-(iv)$ are mapped to $0$.
 Now $C = e_1 s_2 e_1 - e_1$ is mapped to $\mathcal{S}_1$, and the relations \ref{eqssprime}, \ref{eqssym} and the fact that
 $\mathcal{S}_1- \mathcal{S}_2 \in I_n$ (hence $\mathcal{S}_i - \mathcal{S}_{i+1} \in I_n$ for all $i$) then imply, using conjugation by
 elements of the braid groups, that the relators associated to (v) and (viii) are also mapped to $0$. Since $s_1 e_1 = e_1 s_1 = a e_1$
 it is clear that $s_1 \mathcal{S}_1 =  \mathcal{S}_1s_1 = a s_1$ and this implies $s_i C - a C \mapsto 0, Cs_i - a C \mapsto 0$ for all $i$,
 which justifies that the 9th type of relator is also mapped to $0$. This completes the proof of the theorem.

\begin{cor} \label{corclasstracesddagger}
Let $(t_n)$ be a Markov trace on $(BMW_n^{\dagger\dagger})$ with value in some $S^{\dagger\dagger}$-module $M$, and $t'_n$
its composite with $M \to M \otimes S^{\dagger\dagger}[(a-x)^{-1},(2a-x)^{-1},(2a+x)^{-1} ]$. Then $t'_n$ is a linear combination
$t_n^H$, $t_n^K$ , $t_n^{\dagger\dagger}$
with coefficients in $M \otimes_R S^{\dagger\dagger}[(a-x)^{-1},(2a-x)^{-1},(2a+x)^{-1} ]$.
\end{cor}
\begin{proof}
We know from proposition \ref{propclassmarkovddagger3} (ii) that $t'_3$ is a linear combination of $t_3^H$, $t_3^K$, $t_3^{\dagger \dagger}$.
The statement is then a consequence of theorem \ref{theoisomddaggerFn} together with proposition \ref{propbasiqueFn} (iv).
\end{proof}
\subsection{Computations inside $F_n$}

Following the arguments of Wenzl in \cite{WENZL} p. 400 we can derive additional useful relations
inside $F_n$.

\begin{lemma} \label{lemformulesFsee}Assume $|j-i| = 1$. Then we have
\begin{enumerate}
\item $s_j s_i e_j = e_i s_j s_i$
\item $s_i e_j e_i = a s_j^{-1} e_i + C$, $e_i e_j s_i = a e_i s_j^{-1} + C$
\item $e_i e_j e_i = e_i + \frac{2a}{x} C$
\item $e_i s_j s_i = a e_i e_j -C$
\end{enumerate}
\end{lemma}
\begin{proof}
By the braid relations, we have $s_j s_i s_j^{\alpha} = s_i^{\alpha} s_j s_i$ whenever $\alpha \in \{ -1,0,1 \}$.
Since $e_k = a ((s_k^{-1} + s_k)/x-1)$ we get (i).
Using the defining relations of $F_n$ and (i) we get
$s_i e_j e_i  = s_j^{-1} (s_j s_i e_j) e_i = s_j^{-1}  e_i s_j (s_i e_i) = a s_j^{-1}  (e_i s_j  e_i) = 
a s_j^{-1}  (C+e_i) = a s_j^{-1}  C+a s_j^{-1} e_i =   C+a s_j^{-1} e_i$ hence the first part of (ii). The
second part is similar. Now $e_i e_j e_i$ is equal to
$$a e_i \left( \frac{s_j + s_j^{-1}}{x}-1 \right) e_i
= a  \left( \frac{1}{x} e_is_je_i + \frac{1}{x} e_is_j^{-1}e_i-e_i^2 \right) 
= a  \left( \frac{1}{x} (e_i + C) + \frac{1}{x} (e_i + C)-\frac{\tilde{\delta}}{x} e_i \right) $$
hence $e_i e_j e_i = 
 a  \left( \frac{ax}{x}e_i + \frac{2}{x}C  \right) $ which proves (iii).
 In order to prove (iv), we use that, because of (iii),
 $e_i s_j s_i = (e_i e_j e_i - \frac{2a}{x} C) s_j s_i =
 e_i (e_j e_i s_j) s_i- \frac{2a}{x} Cs_j s_i =
 e_i (a e_j s_i^{-1} + C) s_i- \frac{2a}{x} C =
  a e_ie_j s_i^{-1}s_i + e_iCs_i - \frac{2a}{x} C =
  a e_ie_j  + \frac{\tilde{\delta}}{x} a C - \frac{2a}{x} C =
  a e_ie_j  + \frac{2-ax-2}{x} a C $ which proves (iv).
 
\end{proof}

\subsection{A central extension of the Temperley-Lieb algebra}

\begin{defi} We define a unital algebra $\widetilde{TL}_n$
over $A = \Q[a,x,x^{-1}]/(a^2 = 1)$ by generators $e_1,\dots,e_{n-1},C$
and relations
\begin{enumerate}
\item $e_i^2 = \tilde{\delta}x^{-1} e_i$
\item $e_i C = C e_i = \tilde{\delta} x^{-1} C$
\item $C^2 = 2 x^{-2} \tilde{\delta}(a-x) C$
\item $e_i e_j = e_j e_i $ if $|j-i| \geq 2$
\item $e_i e_j e_i = e_i + \frac{2a}{x} C$ if $|j-i|  = 1$
\end{enumerate}
\end{defi}

We have a natural morphism $\widetilde{TL}_n \to F_n $ of unital $A$-algebras.
The next proposition shows that $\widetilde{TL}_n$ can be identified
with a subalgebra of $F_n$, and is a genuine extension of the ordinary
Temperley-Lieb algebra $TL_n$ defined as the quotient of $\widetilde{TL}_n$ by the
two-sided ideal generated by $C$.

\begin{prop} The natural morphism $\widetilde{TL}_n \otimes_A S^{\dagger\dagger} \to F_n  \otimes_A S^{\dagger\dagger}$
is injective. The $S^{\dagger\dagger}$-module $\widetilde{TL}_n \otimes_A S^{\dagger\dagger}$ is free of rank $1+Cat_n$,
where $Cat_n$ denotes the $n$-th Catalan number.
\end{prop}
\begin{proof}
Let $\mathcal{B}$ the set of words in the $e_i$'s that provides the
usual basis of the Temperley-Lieb algebra, namely `increasing products
of increasing strings', see \cite{JONES} p. 27. Their
image inside $\overline{BMW}_n \otimes_{\overline{R}} S^{\dagger\dagger}$
through $\widetilde{TL}_n \otimes_A S^{\dagger\dagger} \to F_n  \otimes_A S^{\dagger\dagger} \to \overline{BMW}_n \otimes_{\overline{R}} S^{\dagger\dagger}$
is a linearly independent subset of Kauffman's tangle algebra, see \cite{MORTONWASSERMANN}.
It follows that $\mathcal{B} \sqcup \{ C \}$ is linearly independent in $\widetilde{TL}_n \otimes_A S^{\dagger\dagger}$. For,
if such a linear combination $\sum_{b \in \mathcal{B} } \la_b b + \la_C C$ was $0$, then its image inside $ \overline{BMW}_n \otimes_{\overline{R}} S^{\dagger\dagger}$
would also be $0$. But this image is equal to the image of $\sum_{b \in \mathcal{B} } \la_b b$, which is zero only if $\la_b = 0$
for all $b \in \mathcal{B}$. But then $\la_C C$ is mapped to $\la_C C \in F_n \otimes_A S^{\dagger\dagger}$, and we know
that this is zero only if $\la_C = 0$. The remaining assertions are then obvious.
\end{proof}

We let $A_1 = A/(x-a)$, and $\widetilde{TL}_n(1) = \widetilde{TL}_n \otimes_A A_1$.
Note that, inside $\widetilde{TL}_n(1)$, we have $e_i^2 = a e_i$.
We let $\bar{e}_i$ denote the image of $e_i \in \widetilde{TL}_n$
under the natural projection $\widetilde{TL}_n \to TL_n$.

\begin{figure} 
\begin{center}
\resizebox{10cm}{!}{\includegraphics{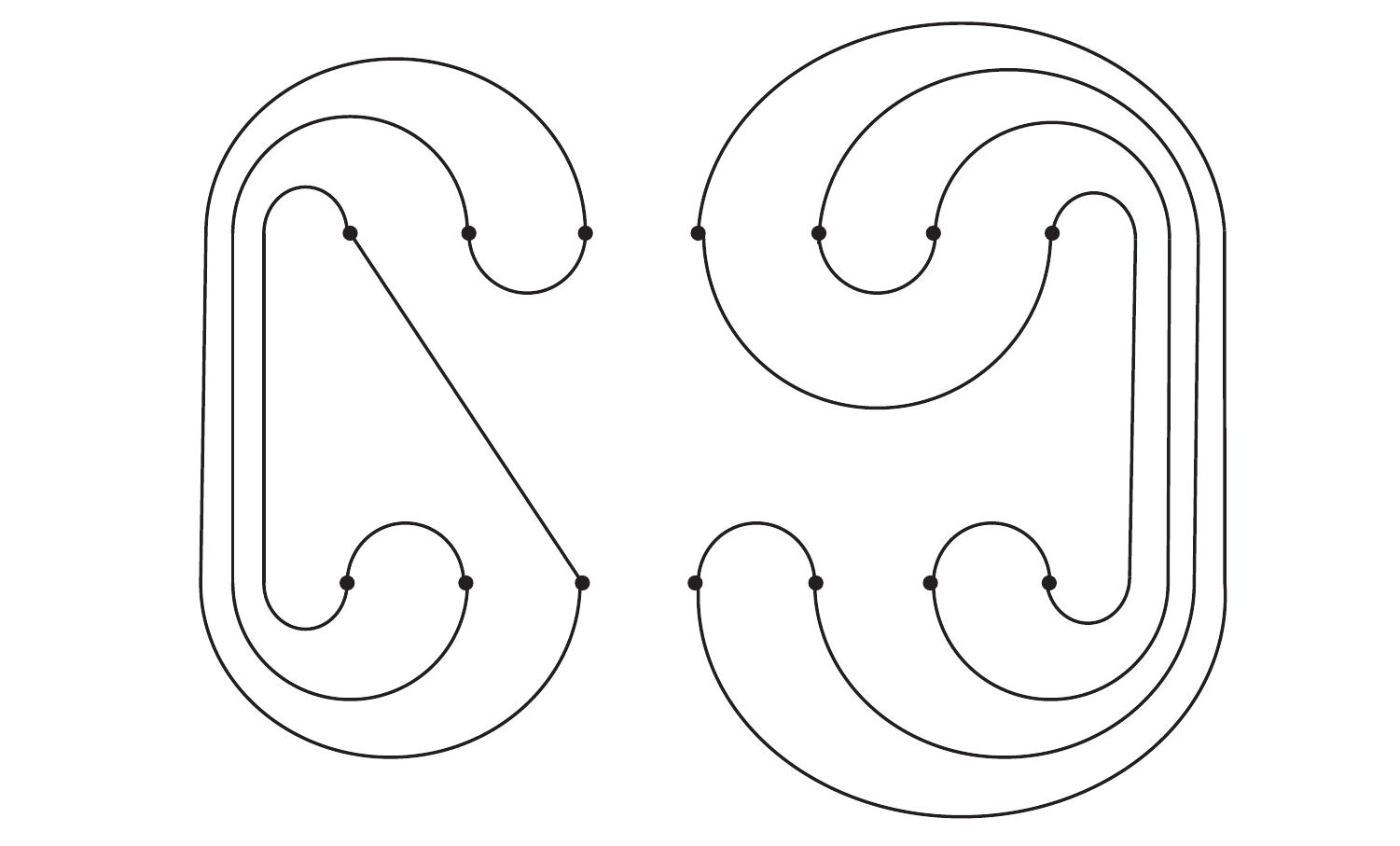}}
\end{center}
\caption{Closure of a Temperley-Lieb diagram with 2 components.}
\label{figclosureTL}
\end{figure}

\begin{prop} There exists a family of traces $t_n : \widetilde{TL}_n(1) \to A_1$
satisfying $t_n(C) = -a^{n+1}$, and
$$
t_n(e_{i_1}\dots e_{i_k}) = a^{k+n} \left( \mathcal{N}(\bar{e}_{i_1}\dots \bar{e}_{i_k}) - k \right)
$$
where $\mathcal{N}(\bar{e}_{i_1}\dots \bar{e}_{i_k})$ denotes the number of connected
components of the diagrammatic closure of $\bar{e}_{i_1}\dots \bar{e}_{i_k} \in TL_n$ (see Figure \ref{figclosureTL})
\end{prop}
\begin{proof}
We fix $n$, and we prove that this formula indeed provides a trace on $\widetilde{TL}_n(1)$. We first note that
the formula $t_n(e_{i_1}\dots e_{i_k}) = a^{k+n} \left( \mathcal{N}(\bar{e}_{i_1}\dots \bar{e}_{i_k}) - k \right)$
provides a linear form on the free $A_1$ algebra on $e_1,\dots,e_{n-1}$, and that
it is indeed a trace on this algebra, because 
$\mathcal{N}(e_{i_1}\dots e_{i_k})$ is invariant under
cyclic permutation of the $e_{i_r}'s$ as is easily seen, for instance by representing the vertices
of the diagram on a circle. It is easily checked that the formula $t_n(C)= -a^{n+1}$
extends this trace to the sum of this algebra with the (non-unital) 1-dimensional algebra
spanned by $C$, defined by $C^2 =0$ and $C e_i = e_i C=aC$ for all $i$.

It then remains to check that the defining relations of $\widetilde{TL}_n(1)$
as a quotient module are mapped to $0$ under $t_n$. Let $x,y$ be words in the $e_i$'s,
and assume $x e_i y$ has length $k$.
Then, by definition of $t_n$, we have 
$$t_n(x e_i^2 y) = a^{k+n+1}(  \mathcal{N}(\bar{x} \bar{e}_i^2 \bar{y})- k -1)
= a^{k+n+1}(  \mathcal{N}(\bar{x} \bar{e}_i \bar{y})+1- k -1) = a t_n (x e_i y)
$$
hence the ideal generated by $e_i^2 -a e_i$ is mapped to $0$.
Similarly, let $i,j$ such that $|j-i| = 1$. Then 
$$
\begin{array}{lclcl}
t_n(x e_i e_j e_i y) &=& a^{k+2+n}(  \mathcal{N}(\bar{x} \bar{e}_i\bar{e}_j \bar{e}_i \bar{y})- k -2)
&=& a^{k+n}(  \mathcal{N}(\bar{x} \bar{e}_i \bar{y})- k -2) \\
&=& t_n(x e_i y) - 2 a^{k+n}
& =&t_n(x e_i y) + 2 t_n(xCy)\end{array}
$$
and this is equal to $t_n (x (e_i + 2 C) y)$
hence the ideal generated by $e_ie_je_i - (e_i+2C)$ is mapped to $0$.
Assume now $|j-i| \geq 2$. Since $\bar{e}_i \bar{e}_j = \bar{e}_j \bar{e}_i$
it is clear that $t_n$ vanishes on the ideal generated by the
$e_i e_j - e_j e_i$. The conclusion follows.

\end{proof}

We let $A_0 = A/(x-2a)$, and $\widetilde{TL}_n(0) = \widetilde{TL}_n \otimes_A A_0$.
Note that, inside $\widetilde{TL}_n(0)$, we have $\tilde{\delta} = 0$
whence $e_i^2 = C e_i = e_i C = 0$, $e_i e_j e_i = e_i + C$ whenever $|j-i| = 1$.

\begin{prop} \label{proptracesTL0}
Let $n \geq 3$ and $u_n,v_n \in A_0$.
There exists a trace on $\widetilde{TL}_n(0)$
defined by $t_n(1) = v_n$, $t_n(C) = - u_n$,
$t_n(e_i) = u_n$ for all $i \in [1,n-1]$ and
$$
t_n(e_{i_1}\dots e_{i_k}) = 0 \mbox{\ if \ } k \geq 2.
$$
\end{prop}
\begin{proof}
Similar to the previous proof, only easier.
\end{proof}

Note that such a trace is never symmetrizing, because $t_n((e_1-e_2)x) = 0$
for every $x \in \widetilde{TL}_n(0)$.

We finally notice that the extension is not split in one of the two special cases we are interested in, therefore
providing a non-zero Hochschild cohomology 2-class inside $\mathrm{HH}^2(TL_n(1),\Q[a]/(a^2-1))$.

\begin{prop} \label{propTLsplit}
Let $k$ be a field of characteristic $0$ and $\varphi : A \onto k$ be a morphism of $\Q$-algebras.
Assume $n \geq 3$. The natural short exact sequences $0 \to k C \to \widetilde{TL}_n \otimes_{\varphi} k \to TL_n \otimes_{\varphi} k \to 0$
splits if and only if $\varphi(x-a) \neq 0$. In that case, the splitting is given by the
map
$$
\bar{e}_i \mapsto e_i - \left( \frac{\varphi(x)}{2\varphi(a-x)} \right) C.
$$
Moreover, this splitting is unique.
\end{prop}

\begin{proof}
For the proof, we identify $x,a$ with their value under $\varphi$, and $\widetilde{TL}_n$, $TL_n$
with their specialization. Let $\check{e}_i$ denote the image of $\bar{e}_i$ under such a splitting.
We have $\check{e}_i = e_i + \la_i C$ for some $\la_i \in k$. For a given $i$, the relation $\check{e}_i^2 = \frac{\tilde{\delta}}{x}
\check{e}_i$ is equivalent to  the equation
$$
(*) \ \ \ \ \frac{\la_i \tilde{\delta}}{x} \left( 1 + \frac{2 \la_i}{x} (a-x) \right) = 0.
$$
If $\la_i = 0$ for some $i$, then, choosing $j$ with $|j-i| = 1$ we get that
$\check{e}_i \check{e}_j \check{e}_i = e_i \check{e}_j \check{e}_i$ is equal to $\check{e}_i$
iff $2a + \la_j \tilde{\delta}^2/x = 0$, hence $\tilde{\delta} \neq 0$ and $\la_j = -2ax/\tilde{\delta}^2$.
But then the equation $\check{e}_j^2 = (\tilde{\delta}/x)\check{e}_j$ implies $\la_j x/\tilde{\delta} = 0$,
a contradiction. Therefore $\la_i = 0$ for every $i$.

If $\tilde{\delta} = 0$, that is $x = 2a$, for every $i,j$ with $|j-i| = 1$ we have
$\check{e}_i \check{e}_j \check{e}_i = e_i e_j e_i = e_i + C$, therefore $\la_i = 1 = \frac{-x}{2(a-x)}$,
and the formula $e_i \mapsto e_i +C$ is easily checked to provide a splitting in this case.

We can thus assume $\tilde{\delta} \neq 0$, and $\la_i \neq 0$ for all $i$. Then $(*)$
implies $x \neq a$ and $\la_i = \frac{-x}{2(a-x)}$. The fact that this formula provides a splitting
is again checked by direct computation, and this proves the claim.
\end{proof}

\subsection{Splittings}

In this section we show that our extension of the BMW-algebra is not split
exactly in the two cases we are interested in.

\begin{prop} \label{propsplittingsBMW} Let $k$ be a field of characteristic $0$ and $\varphi: S^{\dagger\dagger}_{\pm} \to k$
a morphism of $\Q$-algebras. Then, for $n \geq 3$, the natural short exact sequence
$0 \to k C \to (F_n \otimes_A S^{\dagger\dagger}_{\pm} ) \otimes_{\varphi} k
\to (BMW_n^{\pm})\otimes S^{\dagger\dagger}_{\pm} ) \otimes_{\varphi} k \to 0$
splits iff $\varphi(Q)$ has a root in $k$, where
$$
Q(\la) = x^4(1+ u \tilde{\delta} + u^2 \tilde{\delta}) \in S^{\dagger\dagger}_{\pm} [\la].
$$
and $u = 2 \la a (a-x)x^{-2}$.
If this is the case, each one of the roots $\la$ provides a splitting $s_i \mapsto s_i + \la C$, and these
are the only two possible splittings. In particular, if $k$ is algebraically closed, then the short
exact sequence splits iff $\varphi((x-a)(x-2a)) \neq 0$, and it admits exactly one
splitting iff $\varphi(x+2a) = 0$.
\end{prop}

\begin{proof}
In the proof we work inside the specializations, and identify $x,a,\dots$ with their images under $\varphi$.
Let us assume that there is a splitting, given by $e_i \mapsto \check{e}_i$, $s_i \mapsto \check{s}_i$.
This splitting provides a splitting of the extension of the Temperley-Lieb subalgebra, and therefore,
by proposition \ref{propTLsplit},
one needs to have $x \neq a$, and $\check{e}_i = e_i - \frac{x}{2(a-x)} C$. We have $\check{s}_i = s_i + \la_i C$
for some $\la_i \in k$. Moreover, the equation $(s_i-a)(s_i^2-xs_i+1) = 0$ implies
that $s_i^{-1} = a s_i^2 -(ax+1)s_i + (x+a)$. Since this equation also holds for $\check{s}_i$,
we get $\check{s}_i^{-1}  = a \check{s}_i^2 -(ax+1)\check{s}_i + (x+a)$. Expanding $\check{s}_i = s_i + \la_i C$
we get $\check{s}_i^{-1} = s_i^{-1} + \la_i a (a-x) (1+2\la_i \tilde{\delta}/x^2) C$. From this we then get that the
equation $\check{e}_i = a\left( \frac{\check{s}_i + \check{s}_i^{-1}}{x} -1 \right)$
imposes $Q(\la_i) = 0$. We now consider the braid relation $ \check{s}_i \check{s}_j \check{s}_i =
\check{s}_j \check{s}_i \check{s}_j = 0$ when $|j-i| = 1$. We get
$\check{s}_i \check{s}_j \check{s}_i = s_i s_j s_i + (2 \la_i + \la_j)C + (\la_i^2 + 2 \la_i \la_j) C^2 + \la_i^2 \la_j C^3$
and therefore
$\check{s}_i \check{s}_j \check{s}_i - \check{s}_j \check{s}_i \check{s}_j$ is equal to
$$
(\la_i - \la_j)\left( 2 \frac{\tilde{\delta}}{x^2}(a-x) \la_i+a\right) \left( 2 \frac{\tilde{\delta}}{x^2}(a-x) \la_j+a\right) C.
$$
Therefore, either $\la_i = \la_j$, or $\tilde{\delta} \neq 0$ and one of the two values $\la_i$ or $\la_j$
is equal to $\la_0 = -ax^2/(2 \tilde{\delta}(a-x))$. Since $Q(\la_0) = x^4/\tilde{\delta}$ this is excluded hence
$\la_i = \la_j$. This proves that $\la_i$ is independent of $i$, hence the splitting has the form $\check{s}_i = s_i + \la C$
with $\la$ independent of $i$. It then remains to prove that this formula, with $\la$ a root of $Q$, provides a splitting.
The relation $\check{s}_i \check{s}_j = \check{s}_j \check{s}_i$ when $|j-i| \geq 2$
is clear, and therefore the only two relations that remain to be checked are $\check{e_i} \check{s}_j^{\eps} \check{e}_i = \check{e}_i$
for $|j-i| = 1$ and $\eps \in \{-1,1\}$. For this we first check that $\check{e}_i C = 0$ by direct computation. Then,
$\check{e}_i \check{s}_j \check{e}_i = \check{e}_i(s_j + \la C) \check{e}_i
= \check{e}_i s_j \check{e}_i + \la \check{e}_i C\check{e}_i 
= \check{e}_i s_j \check{e}_i + 0$ and, expanding $\check{e}_i = e_i - \frac{x}{2(a-x)} C$, we  get
$\check{e}_i \check{s}_j \check{e}_i =  \check{e}_i$. The case $\eps = -1$ is similar, and this concludes the proof.

\end{proof}

\begin{remark}
In cohomological terms, the non splitting in the cases $x=a$ and $x=2a$ provides a non-zero cohomology 2-class in the
Hochschild cohomology of these specializations of the BMW algebras with values in the one-dimensional bimodule
given by $s_i \mapsto a$ (which factorizes through $BMW_n^{\pm}$ in these cases). 
If $x\not\in \{a,2a,-2a \}$ and the sequence splits, necessarily in two different ways, then the two splittings
afford to distinct $BMW_n^{\pm}$-bimodule structures on $k$, namely $s_i C = C s_i = bC$
and $s_i C = C s_i = c C$.
\end{remark}

\begin{remark} \label{remarkretractTL} Another natural question is
for which specializations $\varphi : A \to k$ (with $k$ a field of characteristic $0$)
the natural morphism $\widetilde{TL}_n \otimes_{\varphi} k \to F_n \otimes_{\varphi} k$
admits a retraction. A straightforward computation shows that this holds if and only if
$\varphi(x+2a) = 0$ and that, in this case, there is exactly one retraction. It
is given by $e_i \mapsto e_i$, $s_i \mapsto -e_i-a$, $s_i^{-1} \mapsto -e_i - a$, $C \mapsto C$.
\end{remark}
\subsection{A central extension of the $(-1)$-Hecke algebras}

We introduce the two-sided ideal $F_n^+$ of $F_n$ generated by $e_1,\dots,e_{n-1},C$,
and we let $F_n^{++}$ denote the ideal $(F_n^+)^2$.

Inside $F_n/F_n^{++}$ we have $e_i = -C$ for all $i$, and therefore $\tilde{\delta} x^{-1} C = e_i C = -C^2
= -2 \tilde{\delta}(a-x)C$, that is $\left( \frac{\tilde{\delta}}{x} \right)^2 a C = 0$. 
We thus let $\overline{F}_n$ denote the quotient of $F_n(0) = F_n \otimes_A A_0$ by the ideal
$(A_0 F_n^+)^2$. It is spanned over $A_0$ by elements $E_w, w \in \mathfrak{S}_n$,
and $C$. Indeed, one can easily check the following formula, when we also denote $s_1,\dots,s_{i-1}$
the Coxeter generator of the symmetric group :
if $\ell(s_i w) = \ell(w)+1$ then $s_i.E_w = E_{s_i w}$, otherwise
$s_i.E_w = -2a^{\ell(w)} C + 2a E_w - E_{s_i w}$ ;
moreover $C^2 = 0$ and $C.E_w = a^{\ell(w)} C$.
Actually, a similar algebra can be associated to every Coxeter system, as we show now.

\begin{theor} \label{theoHeckeExt} Let $(W,S)$ be a Coxeter system, and $k$ a field of characteristic $\neq 2$. The formulas
$$
\left\lbrace \begin{array}{lcll} s.E_w &=& E_{s w} & \mbox{ if } \ell(s w) = \ell(w)+1 \\
&=&  -2a^{\ell(w)} C + 2a E_w - E_{s w} & \mbox{otherwise} \\
s.C &=& a C
\end{array} \right.
$$
for all $s \in S$, $w \in W$, define a representation of the Artin-Tits group $B$ associated
to $(W,S)$ on the free module over $k[a]/(a^2-1)$ spanned by $C$ and the $E_w, w \in W$. When $W$ is finite, the image of the group algebra of $B$ inside this representation is a free module of rank $1+|W|$.
In all cases, this image projects onto the Iwahori-Hecke algebra of $(W,S)$ defined by the relation
$(s-a)^2 = 0$ for all $s \in S$, with kernel the linear span of $\tilde{C} = -(s-a)^2$ for an arbitrary choice of $s \in S$. When
$W$ admits a single conjugacy class of reflections,
this algebra is the quotient of the group
algebra of $B$ by the relations $(t-a)(s-a)^2 = (s-a)^2 (t-a) = 0$ for all $s,t \in S$.
\end{theor}
\begin{proof}
For every $s \in S$ we introduce the endomorphism $R_s$ defined by
$$
\left\lbrace \begin{array}{lcll} R_s(E_w) &=& E_{ws} & \mbox{ if } \ell( ws) = \ell(w)+1 \\
&=&  -2a^{\ell(w)} C + 2a E_w - E_{ ws} & \mbox{otherwise}
\end{array} \right.
$$
and $R_s(C) =aC$.
We want to prove that the action of $\underbrace{stst\dots}_{m_{st}}$ on each $E_w$ coincides
with the action of $\underbrace{tsts\dots}_{m_{st}}$. 
As in the classical proof for the usual Hecke algebras (see \cite{LIE456}, ex. 23 a) in ch. IV \S 2) we check that $R_s, R_t$ commute with the actions of $s$ and $t$ by a straightforward computation, only using that the two conditions
$\ell(swt) = \ell(w)$ and $\ell(sw) = \ell(wt)$, when met at the same time, imply $sw = wt$. From
this and the obvious fact that 
$$\underbrace{stst\dots}_{m_{st}}.E_1 = E_{\underbrace{stst\dots}_{m_{st}}} = 
E_{\underbrace{tsts\dots}_{m_{st}}} = \underbrace{tsts\dots}_{m_{st}}.E_1 
$$
we deduce that, writing any $w$ as a reduced expression $t_{i_1}\dots t_{i_r}$,
and letting $R_{\underline{w}} = R_{t_1}\dots R_{t_r}$, we get
$$\underbrace{stst\dots}_{m_{st}}.E_w=
\underbrace{stst\dots}_{m_{st}}. R_{\underline{w}}.E_1
=
R_{\underline{w}}\underbrace{stst\dots}_{m_{st}}. E_1 
=
R_{\underline{w}}\underbrace{tsts\dots}_{m_{st}}. E_1 
=
\underbrace{tsts\dots}_{m_{st}}. R_{\underline{w}}E_1 
=
\underbrace{tsts\dots}_{m_{st}}. E_w
$$ 
and this proves the first claim.
Let $\tilde{H}$ denote the image of the group algebra of $B$ in this representation. By the same argument,
we get that the morphism $g \mapsto g.E_1$ induces an injective module morphism between $\tilde{H}$
and the linear span $\mathcal{E}$ of the $E_w$ and $C$. Letting $\overline{\mathcal{E}}$ denote the
quotient of $\mathcal{E}$ by the linear span of $C$, we get an action of $\tilde{H}$ on $\overline{\mathcal{E}}$
which factorises through the regular representation of the usual Hecke algebra $H$ of $(W,S)$.
Letting $\overline{E}_w \in \overline{\mathcal{E}}$ denote the image
of $E_w \in \mathcal{E}$, we get  
therefore a surjective map $x \mapsto x. \overline{E}_1$ from $\tilde{H}$ onto
a free module 
with basis the $\overline{E}_w, w \in W$. If $W$ is finite, we deduce
from this that the rank of $\tilde{H}$ is $1+|W|$. Since $\tilde{H} \to \mathcal{E}$ is injective we know that the kernel of $\tilde{H} \to
\overline{E}$ is the linear span of $C$, and this kernel coincides with the kernel of $\tilde{H} \to H$ by the faithfulness of
the regular representation. If $W$ is finite, we deduce
from this that the rank of $\tilde{H}$ is $1+|W|$. 
We have $(s-a).E_1 = E_s - a E_1$ and $(s-a)^2.E_1 = -2aC+2aE_s -E_1 - a E_s - a E_s + E_1= -2aC$.
Letting $\tilde{C} \in \tilde{H}$ denote the action of $-(s-a)^2$, we get $\tilde{C}.E_1 = 2a C$, $\tilde{C}.C = 0$
hence $\tilde{C}^2 = 0$. Also note that, since $\tilde{C}.E_1 = -2a C$ does not depend on the choice of $s \in S$,
the definition of $\tilde{C}$ does not depend on the choice of $s$ either. Since $(t-a).C = 0$ for all $t \in S$ we get $(t-a)\tilde{C}.E_1 = 0$.
Moreover $\tilde{C}.E_t = \tilde{C} R_t(E_1) = R_t(\tilde{C}.E_1) = 2a R_t(C) = 2C$. Therefore
$\tilde{C}.(t-a).E_1 = \tilde{C}.E_t - a \tilde{C}.E_1 = 2 C - 2 C = 0$. Let now $\hat{H}$ denote the quotient of the
group algebra $k[a]/(a^2-1)B$ of $B$ by the relations $(s-a)^2(t-a) = (t-a)(s-a)^2 = 0$ for all $s,t \in S$. We proved that the
natural surjective morphism $k[a]/(a^2-1)B \to \tilde{H}$ factors through $\hat{H}$. For $s \in S$, let $\hat{C}_s$ denote the image of
$-(s-a)^2$ inside $\hat{H}$. Since $(t-a) \hat{C}_s = \hat{C}_s (t-a)=0$ we get $t \hat{C}_s = a \hat{C}_s = \hat{C}_s t$
hence $t \hat{C}_st^{-1} = \hat{C}_s$ for all $t \in S$
hence $b \hat{C}_sb^{-1} = \hat{C}_s$ for all $b \in B$. 
If $W$ has a single conjugacy class of reflections
this implies that $\hat{C}_s$ does not depend of $s \in S$, because in that case all the elements of $S$
are conjugated one to the other inside $B$, and $b \hat{C}_s b^{-1} = \hat{C}_{bsb^{-1}}$ whenever $bsb^{-1} \in S$.
Therefore we note $\hat{C} = \hat{C}_s$. By the above we know that its linear span is a two-sided ideal of $\hat{C}$, and
it is clear that the composite map $\hat{H} \to \tilde{H} \to H$ factors through $\hat{H}/\langle \hat{C} \rangle \to H$.
But $\hat{H}/\langle \hat{C} \rangle $ is the quotient of $k[a]/(a^2-1)B$ by the relations $(s-a)^2 =0$ for all $s \in S$
hence this map is an isomorphism. By the short five-lemma this implies that $\hat{H} \to \tilde{H}$ is also
an isomorphism.

\end{proof}

\begin{remark} These extensions of the $(-1)$-Hecke algebras are 
non-split ; indeed, a splitting would provide elements $\hat{s} = s + \la C$
for some $\la \in k$ such that $(\hat{s}-a)^2=0$. But an easy computation
shows that $(\hat{s}-a)^2=(s-a)^2 \neq 0$. Therefore, these extensions 
provide natural non-zero Hochschild 2-cohomology classes in the cohomology of these
Hecke algebras with values in the trivial bimodule afforded
by the obvious augmentation map. When $W = \mathfrak{S}_n$ and $k$ has characteristic $0$, it
is proved in \cite{BEM} (theorem 1.1) that the corresponding (Hochschild) cohomology group has dimension $1$. A natural
question is then whether our extension spans this 2nd cohomology group whenever $W$ has a single
class of reflections.
\end{remark}

\begin{remark} When $W$ has several classes of reflections, then
the quotient of $k[a]^2/(a^2-1) B$
by the relations $(s-a)^2(t-a) = (t-a)(s^2-a) = 0$ defines a larger algebra. This algebra projects
onto the usual Hecke algebra and the kernel of the projection is a two-sided nilpotent ideal of rank
the number $r$ of conjugacy classes of reflections. The action of $k[a]/(a^2-1)B$ on this
algebra admits a similar description : a basis is given by elements $E_w, w \in W$ together
with elements $C_s, s \in S$ such that $C_s = C_{s'}$ whenever $s$ is a conjugate of $s'$,
and the action itself is given by the formulas $s.E_w = E_{sw}$ is $\ell(sw) = \ell(w)+1$
and $s.E_w = -2a^{\ell(w)} C_s + 2a E_w - E_{sw}$ otherwise. The proof is similar
and left to the reader.
\end{remark}

When there is a single conjugacy class of reflections, one may wonder if we could reduce
the number of relations by asking for e.g. $(t-a)(s-a)^2 = 0$ for all $s,t \in S$, but \emph{not} for
$(s-a)^2(t-a)=0$ for all $s,t \in S$. The answer is positive, as we show now.

\begin{prop} If $(W,S)$ is an irreducible Coxeter system 
with a single conjugacy class of reflections
and $char. k = 0$, 
then the algebra $\tilde{H}(W,S)$ is the quotient of the group algebra of
$k[a]/(a^2-1)B$ by the relations $(t-a)(s-a)^2 = 0$ for $s,t \in S$. The corresponding
ideal is also generated by the relations $(s-a)^2(t-a) = 0$ for $s,t \in S$. If $(W,S)$ is simply laced, 
these statements remain valid under the weaker assumption
$char. k \neq 2$.
\end{prop}
\begin{proof}
We first give two proofs of this  statement in the important simply-laced case. These proofs
are straightforward and explicit, in contrast with the proof of the general case that
we provide after that.

In this simply-laced case, the statement can be reduced to the special case where $W$ has type $A_2$.
If $char. k = 0$, by using $k[a]/(a^2-1) \simeq k[a]/(a-1)\oplus k[a]/(a+1) \simeq k \oplus k$
and the fact that $\tilde{H}$ is in this case a quotient of the Hecke algebra a specialization of $H_3$,
we can compute the dimension of the ideal generated by these relations, and concludes in this
way that the ideal coincides with the ideal generated by the relations $(t-a)(s-a)^2 = (s-a)^2(t-a)=0$ for $a \in \{-1,1 \}$.
By this method one may actually get explicit expressions over $\Q$ whose denominators are powers of $2$, thus
getting the conclusion for every field of characteristic $\neq 2$. We provide an alternative, \`a la Coxeter
argument. We compute inside the quotient $A$ of $k B_n$ by the relations $(t-a)(s-a)^2 = 0$ for
all $s,t \in S$.  Again for all $s,t \in S$, since $(s-a)^3 = 0$
we get $s^{-1} = a s^2 - 3 s +3a$ and therefore the identity $s(t-a)^2 = a (t-a)^2$ implies $s^{-1}(t-a)^2 = a(t-a)^2$.
We assume that $(W,S)$ has type $A_2$ and we let $S = \{s,t \}$.
Then, $sts = tst$ implies $(t-a)^2 = (ts)^{-1}(s-a)^2 ts = s^{-1} t^{-1}(s-a)^2 ts = (s-a)^2 ts$. Let then $X = (s-a)^2 t - 2a(t-a)^2  + (t-a)^2 s$.
We get $X.s = (s-a)^2 ts + (t-a)^2 (s^2 -2as) = (t-a)^2 (s^2-2as +1) = (t-a)(t-a)(s-a)^2 = 0$. Since $s$ is invertible
this implies $X = 0$, that is $(s-a)^2 t = 2a(t-a)^2  - (t-a)^2 s$. Symmetrically we get $(t-a)^2s = 2a(s-a)^2  - (s-a)^2 t$
hence $2a (s-a)^2 = (t-a)^2 s + (s-a)^2 t = (t-a)^2 s + 2a(t-a)^2  - (t-a)^2 s = 2a(t-a)^2$. 
Thus $(s-a)^2 = (t-a)^2$ hence $(s-a)^2 t = (t-a)^2 t = t(t-a)^2 = a(t-a)^2 = a (s-a)^2$ and symmetrically $(t-a)^2 s = a(t-a)^2$.
The conclusion follows.

We now provide an argument for the general case in characteristic $0$. First of all, we may as in our
very first argument assume that $H$ and $\tilde{H}$ are defined over $k$ 
and $a \in \{-1,1 \} \subset k$. Secondly, we can assume $a=1$, for
the map $s \mapsto -s$ defines an isomorphism between the two
variations of $\tilde{H}$ that permutes consistently the ideals under
consideration. Thus, $H$ is the quotient of $k B$ by the relations $(s-1)^2 = 0$
for $s \in S$,
while $\tilde{H}$ is the quotient of $k B$ by the relations $(t-1)(s-1)^2 =(s-1)^2(t-1)= 0$
for $s,t \in S$. We let $\hat{H}$ denote the quotient of $k B$ by the relations $(t-1)(s-1)^2 = 0$ for $s,t \in W$. 
For $s \in S$ we let $C_s = (s-1)^2 \in \hat{H}$. 

We want to show that $\hat{H} = \tilde{H}$ and for this
we can assume that $W = \langle s,t \rangle$ is of type $I_2(m)$ with $m$ odd.
Indeed, recall that our assumption that there is only one conjugacy class of reflections 
implies that every two elements $s,t \in S$ are connected by a chain $s=s_1,s_2,\dots,s_r = t$ such that $\langle s_i, s_{i+1} \rangle$ is a finite dihedral group of odd type. Therefore, if we can prove our statement for each $\langle s_i, s_{i+1} \rangle$,
we get $C_{s_i} s_{i+1} = 
C_{s_i}= C_{s_i} s_i $ ;
since $s_i$ and $s_{i+1}$ are conjugates inside
$\langle s_i, s_{i+1} \rangle$ this proves $C_{s_i} = C_{s_{i+1}}$, and by induction $C_s = C_t$. But then $C_s(t-1) = C_t(t-1) = 0$,
which proves the claim. We thus assume from now on that $W = \langle s,t \rangle$ is of type $I_2(m)$ with $m$ odd.

Then $\hat{H} = \tilde{H}$ is equivalent to
saying that $\hat{H}$ acts on $k C_s$ by right multiplication through its
(unique) 1-dimensional representation $s,t \mapsto 1$. Indeed, if
this is the case we have $(s-1)^2(t-1) = 0$ and this implies $(t-1)^2(s-1)=0$
since these two expressions are conjugated by an element of $B$.
This is what we prove now.

First notice that this action factorizes through $H$. Indeed, whenever $f$ is the
image in $\hat{H}$ of an element of $B$, we have $g(s-1)^2 = (s-1)^2$
and $g(t-1)^2 = (t-1)^2$, and therefore
$$
C_s g s^2 = C_s g(s-1)^2 + 2 C_s g s - C_s g = C_s (s-1)^2 + 2 C_s g s - C_s g = 2 C_s g s - C_s g
$$
and similarly $C_s g t^2 = 2 C_s g t - C_s g$
since $C_s (t-1)^2 = (s-1)(s-1)(t-1)^2 = 0$. Since $\hat{H}$ is spanned
by the image of $B$, this proves that the right action on $C_s \hat{H}$ factorizes
through $H$. Also note that $k$ can be assumed to be algebraically closed. We now need
to make a few remarks on $H$.

Recall from e.g. \cite{GECKPFEIFFER} that the generic Hecke algebra admits two 1-dimensional
irreducible representations and $(m-1)/2$ two-dimensional ones, that can be defined
by explicit formulas. It is straightforward to check that the specializations at $q=-1$ of
the 2-dimensional ones
$$
s \mapsto \begin{pmatrix} -1 & 0 \\ 1 & -1 \end{pmatrix}
\ \ \ \ \ \ t \mapsto \begin{pmatrix} -1 & c_j \\ 0 & -1 \end{pmatrix}
$$
with $c_j = -2-(\zeta^j + \zeta^{-j})$ where $\zeta$ is a primitive $m$-th root of $1$
and $1 \leq j \leq (m-1)/2$ are pairwise non-isomorphic irreducible representations of $H$.
Note in passing that $st$ is mapped to a matrix of trace $c_j+2$ and determinant $1$,
and therefore is conjugated to $\mathrm{diag}(- \zeta^j, - \zeta^{-j})$. Since the two $1$-dimensional irreducible representations become one, it follows that the Jacobson
radical $J(H)$ has dimension $1$.

We claim that $J(H)$ is spanned by $X = \sum_{w \in W} (-1)^{\ell(w)} T_w \in H$. For
this, we first note that $X \mapsto 0$ under the $1$-dimensional representation $s,t \mapsto 1$. Then, letting $X_s = T_1 - T_t  + T_{st} - T_{tst} + ...$
we have $X = X_s - X_s.T_s$ and therefore $X.T_s = X_s.T_s - X_s.T_s^2 = X_s.T_s
-2 X_s.T_s + X_s = X$. Similarly, with obvious notations, $X = X_t - X_t.T_t$
and therefore $X.T_t = X$. It follows that $k X$ is a $1$-dimensional two-sided
ideal of $H$ whose image under $H/J(H) \simeq k \oplus M_2(k)^{(m-1)/2}$
cannot be the two-sided ideal $k$. Since all the other proper 2-sided ideals
of $H/J(H)$ have dimension at least $4$, this proves that its image inside $H/J(H)$ is
$0$ hence $k X = J(H)$.

From this we deduce that the right action of $H$ on $C_s \hat{H}$ factorizes
through $H/J(H)$. Indeed, letting $Y_s = T_1 - T_t + T_{ts} - T_{tst}+\dots$
we have $X = Y_s - T_s.Y_s$ hence $C_s.X = C_s.Y_s - C_s.s.Y_s= C_s Y_s - C_s Y_s = 0$, since $C_s.s = (s-1)^2(s-1)+(s-1)^2 = (s-1)^2 = C_s$.
Finally, since $(st)^{m} = (ts)^{m} $ is central, we
have $C_s(st)^{m} = (st)^{m} C_s = C_s$ hence the two-sided ideal of $H$
generated by $(st)^{m} -1$ acts by $0$. But the image of $(st)^{m}$
inside each 2-dimensional irreducible representation of $H$ is
$\mathrm{diag}(-\zeta^j, -\zeta^{-j})^{m} = -1$ hence, because $-2 \neq 0$ in $k$, the ideal
generated by the image of $(st)^{m}-1$ inside $H/J(H) \simeq k \oplus M_2(k)^{(m-1)/2}$ is $M_2(k)^{(m-1)/2}$. This proves that the right action of $H$ on $C_s$ factorizes through $k$, and this proves the claim.

The second statement about the relations $(s-a)^2(t-a) = 0$ is obviously similar.
\end{proof}

Thus $\overline{F}_n = F_n(0)/(A_0 F_n^+)^2$ is the extension of the theorem corresponding to $W = \mathfrak{S}_n$.
Notice that the natural map $\overline{F}_n \to \overline{F}_{n+1}$ is into for all $n \geq 1$.
We now prove that there is indeed a Markov trace on $F_n$ factorizing through $\overline{F}_n$.
Our proof is essentially an adaptation of Jones's proof of existence for the Ocneanu trace (see \cite{JONESANNALS} theorem 5.1).
\begin{theor} \label{theoexistencetracedeltanul} There exists a unique family of traces $t_n : \overline{F}_n \to A_0$
satisfying $t_{n+1}(xs_n^{\pm 1}) = t_{n}(x)$ for all $x \in \overline{F}_{n-1}$ and 
$t_2(C) = 1$.
\end{theor}
\begin{proof}
Because $t_{n+1}(C) = a t_{n+1}(C s_n)$ and $C \in \overline{F}_{n-1}$, the condition $t_2(C) = 1$ implies $t_n(C) = a^n$
for all $n \geq 2$.
We recall that every element of $\mathfrak{S}_{n+1}$ admits a reduced expression of
the form $wy_k$ with $w$ a reduced expression of some element in $\mathfrak{S}_n$
and $y_k = s_n s_{n-1}\dots s_k$ with $1 \leq k \leq n+1$, with the convention $y_{n+1} = 1$.
We assume that $t_n$ is uniquely defined with a trace satisfying the Markov property, and
we show from this that $t_{n+1}$ is also uniquely defined. We let $\hat{y}_k = s_{n-1}\dots s_k$
for $1 \leq k \leq n$.

First of all, we note that
$1 = \frac{a}{2} (s_i + s_i^{-1}) - a e_i= \frac{a}{2} (s_i + s_i^{-1}) + a C$, hence the
Markov property imposes that, for all $x \in \overline{F}_n$, we have
$$t_{n+1}(x) = at_{n}(x) +at_{n}(x C) = at_{n}(x) + a^{1+\ell(w)+n}$$
if $x$ is given by the reduced expression $w$. Now $\overline{F}_{n+1} = \overline{F}_n
\oplus \bigoplus_{k < n+1} \overline{F}_n y_k$.
Therefore, $t_{n+1}$ is uniquely defined by its value on $\overline{F}_n$, that we already
defined, and on the $\overline{F}_n y_k$ for $k < n+1$.
But the Markov property imposes $t_{n+1} ([w] y_k) = t_{n+1} ([w] s_n\hat{y}_k) =  t_{n} ([w] \hat{y}_k)$, so we take
this as a definition.

We need to prove that $t_{n+1}$ is a trace and that is satisfies the Markov condition.
We start by the latter property, and actually we prove first that
$t_{n+1}(x s_n^{\pm 1} y) = t_n(xy)$ for all $x ,y \in \overline{F}_n$. First of all we
note that $t_{n+1}(x s_n \hat{y}_k) = t_n(x \hat{y}_k)$ for all $x \in \overline{F}_n$,
since it holds for $x = C$ as well as all the $x = [w]$ for $w$ a reduced expression
in $\mathfrak{S}_n$. We can then restrict ourselves to proving that
$t_{n+1}(x s_n y) = t_n(xy)$ for all $y$ of the form $[w] \hat{y}_k$ for $w$ a reduced
expression in $\mathfrak{S}_{n-1}$ and $1 \leq k \leq n$. Then
$t_{n+1}(x s_n y) = t_{n+1}(x s_n [w] \hat{y}_k) = t_{n+1}(x [w] s_n \hat{y}_k) = t_{n+1}(x [w] y_k)
= t_n(x[w] \hat{y}_k) = t_n(xy)$. We now prove that $t_{n+1}(x s_n^{-1} y) = t_n(xy)$
under the same assumptions on $x,y$. We can assume $x = [w]$ and $y = [m]$
for $w,m$ reduced expression. We then notice that $s_n^{-1} = 2a + 2 e_n - s_n = 2a - 2C - s_n$
and therefore $t_{n+1}(x s_n^{-1} y) = 2at_{n+1}(xy) - 2 a^{\ell(w)+ \ell(m)+n} - t_{n+1} (x s_n y)
= 2at_{n+1}(xy) - 2 a^{\ell(w)+ \ell(m)+n} - t_{n+1} (x s_n y)
= 2 t_n(xy) + 2a^{n+\ell(w)+\ell(m)} - 2 a^{\ell(w)+ \ell(m)+n} - t_n(xy) = t_n(xy)$.

We now prove that $t_{n+1}$ is a trace. We need to prove $t_{n+1}(s_i x) = t_{n+1}(x s_i)$
for all $i \leq n$. We first assume $i < n$. If $x \in \overline{F}_n$ this is an immediate
consequence of the relation between $t_{n+1}(x)$ and $t_n(x)$. If not, we can assume
$x = [w] s_n \hat{y}_k$. Then $t_{n+1}(s_i x) = t_{n+1}(s_i [w] s_n \hat{y}_k) = t_n(s_i [w]  \hat{y}_k)$
by the Markov property. Since $t_n$ is a trace this is equal to
$t_n(s_i [w]  \hat{y}_k) = t_n( [w]  \hat{y}_ks_i) =  t_{n+1}([w] s_n \hat{y}_k s_i) = t_{n+1}(xs_i)$.

We now let $i = n$. If $x \in \overline{F}_n$ this is a consequence of the Markov property :
$t_{n+1}(s_n x) = t_n(x) = t_{n+1}(x s_n)$. If not, we can assume $x = u s_{n} v$
with $u,v \in \overline{F}_{n}$. Then  $t_{n+1}(s_n x) = t_{n+1}(s_nu s_n v)$. 

\begin{itemize}
\item If $u,v \in \overline{F}_{n-1}$ this is equal to $t_{n+1}(s_nu vs_n) =  t_{n+1}(us_nv s_n) =  t_{n+1} (x s_n)$.
\item
If $u \in \overline{F}_{n-1}$ and $v \not\in \overline{F}_{n-1}$ this is equal to 
$t_{n+1}(us_n^2 v) = -2a t_{n+1}(uCv) - t_{n+1}(uv) + 2a t_{n+1}(us_nv)$
since $s_n^2 = -2a C - 1 + 2a s_n$,
and therefore to $-2a^{\ell(u)+\ell(v)+n+1}  - t_{n+1}(uv) + 2a t_{n+1}(us_nv) =
-2a^{\ell(u)+\ell(v)+n+1}  - t_{n+1}(uv) + 2a t_{n}(uv) = 
 -2a t_{n+1}(uCvs_n)  - t_{n+1}(uv) + 2a t_{n+1}(uvs_n)$.
On the other hand we can write $v = [w] s_{n-1} [w']$ with $w,w'$ reduced expressions
in $\mathfrak{S}_{n-1}$. Then
$t_{n+1} (us_n v s_n) = t_{n+1}( u s_n[w] s_{n-1} [w'] s_n)
= t_{n+1}( u [w]s_n s_{n-1}  s_n[w'])
= t_{n+1}( u [w]s_{n-1} s_{n}  s_{n-1}[w'])
=t_{n}( u [w]s_{n-1}^2[w'])
= -2at_{n}( u [w]C[w']) - t_{n}( u [w][w']) + 2a t_{n}( u [w]s_{n-1}[w'])
= -2at_{n}( u [w]C[w']) - t_{n}( u [w][w']) + 2a t_{n}( u v)
= -2at_{n+1}( u Cvs_n) - t_{n+1}( u v) + 2a t_{n+1}( u vs_n)
$
since $t_n(u [w][w']) = at_{n-1}(u[w][w']) + a^{n+\ell(u) +\ell(w)+\ell(w')} = 
at_{n}(u[w]s_{n-1}[w']) + a^{n+\ell(u) +\ell(w)+\ell(w')} = t_{n+1}(uv)$.
\item the case $u\not\in \overline{F}_{n-1}$ and $v \in \overline{F}_{n-1}$ is similar and left to the reader.
\item We now assume $u = [w] s_{n-1} [w']$ and $v = [m] s_{n-1} [m']$.
Then 
$$
\begin{array}{lcl}
t_{n+1} (s_n u s_{n} v) &=& t_{n+1}(s_n [w] s_{n-1} [w'] s_n [m] s_{n-1} [m'])\\
&=& t_{n+1}( [w] s_ns_{n-1} s_n[w']  [m] s_{n-1} [m'])\\
&=& t_{n+1}( [w] s_{n-1}s_{n} s_{n-1}[w']  [m] s_{n-1} [m'])\\
&=& t_{n}( [w] s_{n-1}^2[w']  [m] s_{n-1} [m'])
\end{array}
$$
 by the Markov property.
Similarly, $$t_{n+1} (u s_n v s_n) = t_n([w] s_{n-1}[w'] [m] s_{n-1}^2 [m']).$$
Expanding $s_{n-1}^2$ we get
$t_{n}( [w] s_{n-1}^2[w']  [m] s_{n-1} [m']) = 
-2a t_{n}( [w] C[w']  [m] s_{n-1} [m']) - t_{n}( [w] [w']  [m] s_{n-1} [m'])
+ 2a t_{n}( [w] s_{n-1}[w']  [m] s_{n-1} [m'])$
while $t_n([w] s_{n-1}[w'] [m] s_{n-1}^2 [m']) = 
-2a t_n([w] s_{n-1}[w'] [m] C [m']) - t_n([w] s_{n-1}[w'] [m]  [m'])
+ 2a t_n([w] s_{n-1}[w'] [m] s_{n-1} [m'])$.
We have $t_{n}( [w] C[w']  [m] s_{n-1} [m']) = t_n([w] s_{n-1}[w'] [m] C [m']) = a^{n+1+\ell(w)+\ell(w')+\ell(m)+\ell(m')}$
$t_{n}( [w] [w']  [m] s_{n-1} [m']) = t_{n-1}( [w] [w']  [m]  [m'])
= t_{n}( [w]s_{n-1} [w']  [m]  [m'])$ by the Markov property,
whence $t_{n+1} (s_n u s_{n} v) = t_{n+1} (u s_n v s_n)$ and the conclusion.
\end{itemize}
\end{proof}

\subsection{The case $x = -2a$}
\label{subsectm2a}

Let $\mathcal{B}_n$ denote the $\Q[a]/(a^2-1)$-algebra $BMW_4^{\dagger\dagger} \otimes S^{\dagger\dagger}/(x+2a)$
and $\mathcal{B}_n^{\pm} = \mathcal{B}_n \otimes \Q[a]/(a \mp 1)$. 
Inside $\mathcal{B}_n$ we have $x=-2a$, $e_i^2 = 2a e_i$, $\tilde{\delta} = 4$ and $(s_i-a)(s_i+a)^2 = (s_i+a)(s_i^2-1) = 0$.
We specialize $H_4$ accordingly, and let $U^{\pm}$, $V^{\pm}$ denote the kernels of its projection onto $\mathcal{B}_4^{\pm}$
and $BMW_4^{\pm}$, respectively. We denote $\mathcal{I}_n^{\pm} = \Ker(\mathcal{B}_n^{\pm} \onto BMW_n^{\pm})$ and
we identify $\mathcal{I}_4^{\pm}$ with the vector space $U^{\pm}/V^{\pm}$.
By proposition \ref{propdimBMWm2a} we know that they have dimension $115-105 = 10$. As we noticed in the
proof of proposition \ref{propTracesXm2a}, we have natural morphisms $\mathcal{B}_n^{\pm} \onto \Q \mathfrak{S}_n$.

We let $C_i^{\pm}$ denote the image of $\mathcal{S}_i$ inside $\mathcal{B}_n^{\pm}$. We have $(C_i^{\pm})^2 = 6a C_i^{\pm} = \pm 6 C_i^{\pm}$.
$\dim \mathcal{B}_n^{\pm} = 115$. The quotient $\mathcal{F}_n^{\pm}$ of $\mathcal{B}_n^{\pm}$ by the ideal generated by $C_1^{\pm} - C_2^{\pm}$ is $F_n \otimes A/(a \mp 1,x+2a)$. We have $\mathcal{B}_n^{\pm} \onto \mathcal{F}_n^{\pm} \onto BMW_n^{\pm}$,
and $\mathcal{B}_3 \simeq \mathcal{F}_3$.
Recall from proposition \ref{propsplittingsBMW} that the morphism $\mathcal{F}_n^{\pm} \to BMW_n^{\pm}$
admits exactly one splitting, given by $s_i \mapsto s_i - (a/3) C$. 

By explicit computations inside $U^{\pm}/V^{\pm}$, we get the following.
\begin{enumerate}
\item The bimodule action of $H_4$ on $\mathcal{I}_4^{\pm}$
factorizes through $(\Q \mathfrak{S}_4)\otimes (\Q \mathfrak{S}_4)^{\mathrm{op}}$
(that is, the left and right actions of the $s_i^2$ are trivial). The corresponding
representations are $\chi_4 \otimes \chi_4 + \chi_{31} \otimes \chi_{31}$ if $a=1$,
$\chi_{1^4} \otimes \chi_{1^4} + \chi_{211} \otimes \chi_{211}$ if $a = -1$, where
$\chi_{\la}$ denotes the irreducible representation of $\mathfrak{S}_n$
associated to the partition $\la$ of $n$, with the convention that $\chi_{[n]}$
is the trivial representation.
\item The subalgebra of $\mathcal{I}_4^{\pm}$ generated by $C_1^{\pm}$ and $C_2^{\pm}$
is 5-dimensional, and defined by the relations $(C_i^{\pm})^2 = 6a C_i^{\pm}$, 
$C_1^{\pm} C_2^{\pm} C_1^{\pm} - C_2^{\pm} C_1^{\pm} C_2^{\pm} = 4(C_1^{\pm} - C_2^{\pm})$, one possible basis being
$C_1^{\pm},C_2^{\pm},C_1^{\pm}C_2^{\pm}, C_2^{\pm}C_1^{\pm}, C_1^{\pm}C_2^{\pm}C_1^{\pm}$.
\item By direct computation, we check that the ideal $\mathcal{I}_4^{\pm}$ is generated
by $C_1^{\pm}$, $C_2^{\pm}$ and the $e_i$'s. Since we have $\mathcal{B}_3^{\pm} = \mathcal{F}_3^{\pm}$,
by lemma \ref{lemformulesFsee} (iii) we know that the $C_i^{\pm}$'s belong to the
subalgebra generated by the $e_i$'s. Therefore, the extension $\mathcal{B}_4^{\pm}$
of $BMW_4^{\pm}$ is basically determined by the induced extension of the
Temperley-Lieb subalgebra of $BMW_n^{\pm}$, at least when $n=4$. We suspect it is the case in general.
\item As an algebra, using known algorithms used
by GAP4, we check that $\mathcal{I}_4^{\pm}$ can be split into a direct
sum of two unital $\Q$-algebras, one of them being 1-dimensional,
the other one being 9-dimensional. We check that the latter
is central and simple, but not a division ring. Therefore, we have
$\mathcal{I}_4^{\pm} \simeq \Q \oplus Mat_3(\Q)$ as a $\Q$-algebra.
\end{enumerate}

One problem we face to extend these properties further is that we need to know whether
$C_3^{\pm}$ and $C_1^{\pm}$ do commute (or to what extent they do not). This
should be doable by computing inside $H_5$, which is still finite-dimensional. However
its dimension ($155520$) is a lot larger than the dimension of $H_4$, and
there is no software capable of dealing with it yet.

These computations are however sufficient to guess a plausible conjecture. For $n \geq 3$, let $TL_n^{\pm}$
denote the quotient of the group algebra $\Q \mathfrak{S}_n$ of the symmetric
group by the ideal $\mathcal{J}_n^{\pm}$ generated by $\mathcal{T} = s_1 s_2 s_1 + a s_1 s_2 + a s_2 s_1 + s_1 + s_2 + a$
with $a = \pm 1$. It is a specialization of the Temperley-Lieb algebra, and has dimension the Catalan number $Cat_n$.
Now recall from the proof of proposition \ref{propTracesXm2a} that we have a surjective morphism $\mathcal{B}_n^{\pm}
\onto \Q \mathfrak{S}_n$. By direct computation we get that $\mathcal{S}_1$ is mapped onto $\mathcal{T}$. Therefore,
we have a commutative diagram of horizontal short exact sequences
$$
\xymatrix{
0 \ar[r] & \mathcal{I}_n^{\pm} \ar[d]_{\pi_n^{\pm}}\ar[r] & \mathcal{B}_n^{\pm} \ar[r] \ar@{->>}[d] & BMW_n^{\pm} \ar[r] \ar@{->>}[d]& 0 \\
0 \ar[r] & \mathcal{J}_n^{\pm} \ar[r] & \Q \mathfrak{S}_n \ar[r] & TL_n^{\pm}\ar[r] & 0 \\
}
$$
The computations above together with the identification of $\mathcal{B}_3^{\pm}$ with a specialization of $F_3$ shows
that, for $n=3$ and $n=4$, the leftmost vertical map $\pi_n^{\pm}$ is an isomorphism. We conjecture
that it is the case in general.

\begin{conj}
$\pi_n^{\pm}$ is an isomorphism for all $n \geq 3$.
\end{conj}
If this conjecture holds true, then the algebra structure of $\mathcal{B}_n^{\pm}$ would be completely determine by
an explicit linear splitting of $\mathcal{B}_n^{\pm} \to BMW_n^{\pm}$ together with
the corresponding Hochschild 2-cocycle. Indeed, the bimodule action of $BMW_n^{\pm}$ on $\mathcal{I}_n^{\pm}$
would be easily determined, since the action of every braid on $\mathcal{I}_n^{\pm}$ would
be identified under $\pi_n^{\pm}$ with the action of the corresponding permutation on the ideal $\mathcal{J}_n^{\pm}$.
Moreover, we would have for the dimension of $\mathcal{B}_n^{\pm}$ the conjectural formula
$$
\begin{array}{lcl}
\dim \mathcal{B}_n^{\pm} &=& \dim BMW_n^{\pm} + \dim \Q \mathfrak{S}_n - \dim TL_n^{\pm} \\
&=& 1.3. \dots . (2n-1) + n! - Cat_n 
\end{array}
$$
Finally, we suspect that the special retraction pointed out in remark \ref{remarkretractTL} echoes some
special phenomenon in $\mathcal{B}_n^{\pm}$ that needs to be
understood further.

\section{Knot invariants}

\subsection{Number of connected components, special trace and change of variables}

If $\beta$ is a braid on $n$ strands, then its closure $L$ is an oriented link. As before,
we denote
$\ell : B_n \to \Z$
the abelianization morphism which maps $\sigma_i$ to $1$.
We will use the following
classical fact :

\begin{lemma} \label{lemwrithe} If $\beta \in B_n$ and $L$ is the closure of $\beta$, then
$n+\ell(\beta) \equiv (\# L)  \mod 2$, where $\# L$ denote the number of
components of the link $L$.
\end{lemma}
\begin{proof}
Decomposing the projection
$\overline{\beta} \in \mathfrak{S}_n$ 
of $\beta$ under $B_n \to \mathfrak{S}_n$
into a product of disjoint cycles,
we get
$r$ cycles of even lengths $2 a_1,\dots,2 a_r$
and $s$ cycles of odd lengths $b_1,\dots,b_s$.
Then $\# L = r+s$, $n = 2 \sum a_i + \sum b_i \equiv s \mod 2$
and $\ell(\beta) \equiv r \mod 2$. This proves the claim.

\end{proof}

This provides a simple interpretation of the trace $t_n^{\dagger\dagger}$, with values in $S^{\dagger\dagger}$. Recall that $a^2 = 1$
inside $S^{\dagger \dagger}$.

\begin{prop} \label{propinvdagdag} If $\beta \in B_n$ and $L$ is the closure of $\beta$, then $t_n^{\dagger \dagger}(\beta) = a^{\# L}$.
\end{prop}
\begin{proof}
By construction, we have $t_n^{\dagger\dagger}(\beta) = a^n a^{\ell (\beta)} = a^{n + \ell(\beta)}$. Since $a^2 = 1$
the conclusion follows from lemma \ref{lemwrithe}.
\end{proof}

There is a well-known connection between `two versions' of the
Kauffman polynomial, one being connected to the
other through a sign depending on the number of connected
components of the link. In our formalism this is seen as
follows, using
the automorphism $\overline{E}$ 
of $\overline{BMW}_n \otimes_{\overline{R}} \overline{S} = BMW_n^+ \oplus BMW_n^-$. 
First recall that the values of $t_n^+$ and $t_n^-$ on braids belong to
a submodule $S_0 = \Q[a^{\pm 1},x^{\pm 1}]$ of $S_+$ and $S_-$, respectively.
Also recall that $S_0$ can be considered as a submodule of $\Q(\alpha,q)$
under $a \mapsto \alpha^{-2}$, $x \mapsto \alpha^{-1}(q\pm q^{-1})$.
We let $\tau_n^{\pm}$ denote the Markov traces $t_n^{\pm}$ viewed
as functions with values in the same ring $S_0$.

We can now show how this well-known connection between the two versions
fits into our setting. This connection can be stated as follows (see e.g. \cite{LICKO}, p. 177).

\begin{prop} \label{invlcorrespBMWplusmoins} If $\beta$ is a braid on $n$ strands whose closure is the link $L$,
then 
$$
 \left.  \tau_n^-
  (\beta)
   \right|_{\stackrel{a\mapsto -a}{x \mapsto -x}} = (-1)^{\# L -1} \tau_n^+(\beta)$$
where we identified $\beta$ with its image in $BMW_n^+$ on the LHS, and its image in $BMW_n^-$ in the RHS,
and $\# L$ denotes the number of
components of $L$.
\end{prop}
\begin{proof}
The LHS can be viewed inside $S_+$ as $\eta \circ t_n^- \circ \overline{E}( \overline{E}(\beta))$, which is equal to
$(-1)^{n-1} t_n^+(\overline{E}(\beta))$ by corollary \ref{cortnplusettnmoins}. 
Now $(-1)^{n-1} t_n^+(\overline{E}(\beta)) = (-1)^{n-1} (-1)^{\ell(\beta)} t_n^+(\beta)$
where $\ell : B_n \to \Z$ denotes as before the abelianization
morphism.
The conclusion then follows from the identity
$n+\ell(\beta) \equiv (\# L)  \mod 2$, proved in lemma \ref{lemwrithe}.
\end{proof}

\subsection{The special case $a = \pm 1$, $b = \mp j$, $c = \mp j^2$}

Recall that $j$ denotes a primitive $3$-root of $1$. We have $y = 1$, $x = a$. In this case, the link invariant afforded
by the Kauffman trace $t_n^+$ is constant equal to 1 (that is, it takes the value 1 on every link), see \cite{LICKO} p. 186, table 16.3 row A,
and this is reproved by the observation that, in this case, $\delta_K = a$ hence $t_n^K = t_n^{\dagger\dagger}$.
For $y= 1$ and $x = a = -1$, according to \cite{LICKO} p. 186, table 16.3 row D, the link invariant associated to $t_n^H$
maps a link $L$ to $(\ii \sqrt{2})^{d_2(T(L))}$, where $T(L)$ is the 3-fold cyclic cover of $S^3$ branched over $L$,
and $d_2(T(L)) = \dim_{\mathbbm{F}_2} H_1(T(L),\mathbbm{F}_2)$. When $y=1$, $x=a=1$, we consider the automorphism $\varphi$
of the group algebra of the braid group defined by $s_i \mapsto s_i^{-1}$, and show that, as in the proposition \ref{invlcorrespBMWplusmoins}  and
corollary \ref{cortnplusettnmoins}, the formula
$T_n(b) = (-1)^{n-1} t_n(\varphi(b))$ induces a bijection between the traces factorizing through the two Hecke algebras
with quadratic conditions $s^2 + s + 1 = 0$ and $s^2 - s+1 = 0$, hence the two invariants are related by multiplication by $(-1)^{\# L -1}$.
As a consequence we get the general formula $t_n^H(\beta) = a^{\# \hat{\beta} - 1} (\ii \sqrt{2})^{d_2(T(L))}$.

We are looking for a Markov trace $t_n^0$ on $\WBMW_n \otimes_R R/(a=\pm 1, b= \mp j, c = \mp j^2)$
such that $t_3^0(1) = t_3^0(s_1) = 0$ and $t_3^0(s_1s_2) = 1$. This implies $t_3^0(s_1^{\alpha} s_2^{\beta}) = 1$ for all $\alpha,\beta \in \{-1,1 \}$.

We checked by computer that there is a (necessarily unique) extension of $t_3^0$ to $F_4$ satisfying the
Markov property. If there is an obstruction to these traces to genuine Markov traces one
thus needs to look for it on at least 5 strands.

\begin{conj} The trace $t_3^0$ can be extended to a Markov trace on the tower $(F_n)$.
\end{conj}

Assuming that conjecture, we get that
the corresponding link
invariant would take the value $a^n \times(4.2^n - 4 - 2n)$ on the $(n+3)$-components unlink.

\subsection{The special case $a =b=c= \pm 1$}

We have $y = 1$, $x = 2a=\pm2 $. In this case, $\delta_H = a$, hence $t_n^H$ coincides
with $t_n^{\dagger\dagger}$ (see also \cite{LICKO} table 6.3, row A), and, for $a=y=1$,
the Kauffman invariant maps a link $L$ to $(\det L)^2$, where $\det L = \Delta_{L}(-1) = \nabla_L(0)$
(see \cite{LICKO} table 6.3, row B). Since this invariant is afforded by $t_n^+$, and because of proposition \ref{invlcorrespBMWplusmoins}, this implies
that $t_n^K(\beta) = a^{\# \hat{\beta} -1} (\det \hat{\beta} )^2$ for every braid $\beta \in B_n$.

 We are looking for a Markov trace $t_n^0$ on $\WBMW_n \otimes_R R/(a=b=c=\pm 1)$
such that $t_3^0(1) = 1$, $t_3^0(s_1) = t_3^0(s_1s_2) = 0$. This implies $t_3^0(s_1^{\alpha} s_2^{\beta}) = 0$ for all $\alpha,\beta \in \{-1,1 \}$.
By theorem \ref{theoexistencetracedeltanul} we know that such a Markov trace exists, factoring through $\overline{F}_n$.
Moreover, by induction we easily get that $t_{n+3}^0(1) = a^n \times(n+1)$, and therefore the
corresponding link invariant takes the value $a^n \times(n+1)$ on the $(n+3)$-components unlink.
Finally, we can easily identify its restriction to the subalgebra $\widetilde{TL}_n(0)$, as follows.

\begin{prop} \label{proptracet0x2a} There exists a family of traces 
$t_n : \widetilde{TL}_n(0) \to A_0$
satisfying $t_{n}(C) =a^{n}$, $t_1(1) = 0$, $t_n(1) = (n-2)a^{n+1}$ if $n \geq 2$,
that would coincide with the restriction of $t_n^0$ to $\widetilde{TL}_n(0)$
if $(t_n^0)$ is well-defined.
\end{prop}
\begin{proof} Letting $u_n = - a^n$,$v_1 = 0$, $v_n = (n-2) a^{n+1}$
in proposition \ref{proptracesTL0} provides a family of traces $t_n$. We know prove that
$t_n$ necessarily coincides with the restriction of the putative Markov trace $t_n^0$.
We have $t_n(1) = t_n^0(1)$ for $n \leq 3$ and, by proposition \ref{proprelrecutraceFn} (iii), 
$t_{n+3}^0(1) = 2a t_{n+2}^0(1) - t_{n+1}^0(1)$ since, we know that $t_n^0$ factors through $(F_n)$ by theorem \ref{theoisomddaggerFn}.
Because of this, we have $t_n^0(1) = t_n(1)$ for all $n$. By proposition \ref{proprelrecutraceFn} we also get $t_n(C) = t_n^0(C)$
for all $n$. From the relation $e_i = (s_i+s_i^{-1})/2 - a$ and the Markov property we get $t_n^0(e_i) = t_n(e_i)$ for all $i$.
It is well-known that the Temperley-Lieb algebra admits a basis made of words in the $e_i$ which has the property
that the $e_i$ of maximal index appears exactly once. Let us consider such a basis element of $\widetilde{TL}_n$ of the form $Ae_iB$ with $A,B$ words in the $e_j$ for $j < i$.
Expanding $e_i$ as before, we get from the Markov property that $t_{n+1}^0(A e_i B) = t_n^0(A B) - at_{n+1}^0(AB)$.
By induction on the length of the words in the $e_i$'s, we get from this formula that $t_{n+1}^0(Ae_i B) = 0$  as soon as $A$ or $B$
has length $\geq 1$, and therefore $t_n^0$ coincides on $t_n$ on every basis element of $\widetilde{TL}_n(0)$, which proves the claim.

\end{proof}

\subsection{Tables}

We gather here some computations that we made of these two invariants, one of them (for $x =a$) being still conjectural.
$$
\begin{array}{|lcc|lcc|lcc|}
\hline
Knot & x=a & x=2a &Knot & x=a & x=2a &Knot & x=a & x=2a  \\
\hline
0_{1} & &   & 3_{1} & 4 & 0  & 4_{1} &10 & 16  \\ \hline
5_{1} & 1 & 0  & 5_{2} &13 &48   & 6_{1} & 7& 80  \\ \hline
6_{2} &13 &96   & 6_{3} & 25 & 144  & 7_{1} & 1 & 0  \\ \hline
7_{2} & -2 & 112  & 7_{3} & 10 & 160  & 7_{4} & 7 & 224   \\ \hline
7_{5} & 25 & 288  & 7_{6} & 25 & 336  & 7_{7} & 31 & 416  \\ \hline
\end{array}
$$
{}

$$
\begin{array}{|lcc|lcc|lcc|}
\hline
Knot & x=a & x=2a &Knot & x=a & x=2a &Knot & x=a & x=2a  \\
\hline
8_{1} & -2 & 160  & 8_{2} & 1 & 240  & 8_{3} & 13 & 288  \\ \hline
8_{4} & -2 & 352  & 8_{5} & -8 & 384  & 8_{6} & 25 & 528   \\ \hline
8_{7} & 13 & 480  & 8_{8} & 13 & 624  & 8_{9} & 25 & 576  \\ \hline
8_{10} & 16 & 672  & 8_{11} & 28 & 720  & 8_{12} & 37& 816  \\ \hline
8_{13} & 22 &832   & 8_{14} & 37 & 960  & 8_{15} & 40 & 1104   \\ \hline
8_{16} & 25 & 1152  & 8_{17} & 49 & 1296  & 8_{18} & 64 & 1936   \\ \hline
8_{19} & -8 & -48  & 8_{20} & -8  & 96  & 8_{21} & 16 & 240  \\ \hline
\end{array}
$$
{}

$$
\begin{array}{|lcc|lcc|lcc|}
\hline
Knot &x=a & x=2a &Knot & x=a & x=2a&Knot & x=a & x=2a  \\
\hline
9_{1} & 4 & 0  & 9_{2} & 7  & 224  & 9_{3} & 1 & 336  \\ \hline
9_{4} & 7 & 416  & 9_{5} & 1 & 528  & 9_{6} & 4  & 720  \\ \hline
9_{7} & 13 & 816  & 9_{8} & 1 & 960   & 9_{9} & 13 & 960  \\ \hline
9_{10} & 31 & 1088  & 9_{11} & 7 & 1040   & 9_{12} & 22 & 1216  \\ \hline
9_{13} & 22 & 1360   & 9_{14} & 10 & 1360   & 9_{15} & 31  & 1520   \\ \hline
9_{16} & 16 & 1536  & 9_{17} & 7 & 1472  & 9_{18} &  37 & 1680  \\ \hline
9_{19} & 25 & 1680   & 9_{20} & 13 & 1632  & 9_{21} & 34  & 1840  \\ \hline
9_{22} & 10 & 1792   & 9_{23} & 28  & 2016  & 9_{24} & 40 & 1968   \\ \hline
9_{25} & 34 & 2224  & 9_{26} & 37 & 2160   & 9_{27} & 37 & 2352   \\ \hline
9_{28} & 40 & 2544   & 9_{29} & 4 & 2544   & 9_{30} & 46 & 2752   \\ \hline
9_{31} & 49 & 2976  & 9_{32} & 49 & 3408   & 9_{33} & 61 & 3648   \\ \hline
9_{34} & 67 & 4688   & 9_{35} & -14  & 720   & 9_{36} & -2 & 1312   \\ \hline
9_{37} & 34 &2016   & 9_{38} & 52 & 3264   & 9_{39} & 46 & 3040   \\ \hline
9_{40} & 70 & 5536   & 9_{41} & -2 & 2416  & 9_{42} & -14  & 64   \\ \hline
9_{43} & -14 & 112  & 9_{44} & -2 & 304   & 9_{45} & 10 & 544   \\ \hline
9_{46} & -11 & 80   & 9_{47} & 13 & 656   & 9_{48} & 37 & 704   \\ \hline
9_{49} & 34 & 640 & & & & & &   \\ \hline
\end{array}
$$

{}
$$
\begin{array}{|lcc|lcc|lcc|}
\hline
Knot & x=a & x=2a &Knot & x=a & x=2a &Knot & x=a & x=2a  \\
\hline
3_1 \# 3_1 & 16 & 64 & 3_1 \# 3_1 \# 3_1 & 10 & 704 & 3_1 \# 3_1 \# 3_1 \# 3_1 & 40  & 6528 \\
3_1 \# 4_1 & 22 & 208 & 4_1 \# 4_1 & 28 & 608 & & &  \\
\hline
\end{array}
$$
{}
For knots of at most 10 crossings as well as small links we use the notations in Rolfsen's book, see \cite{ROLFSEN}.
For knots of 11 crossings we use the notation of the KnotScape software, and give at the same time a braid description in order not to avoid
possible ambiguities.
It can be checked that the knots $8_2$ and $11_{373}$, the latter one being the braid closure of $\bar{1}2\bar{1}2\bar{3}\bar{3}\bar{4}223\bar{4}\bar{4}$, have the same Alexander polynomial, but both invariants distinguish them (and so do the Homfly polynomial).
The knots $10_{41}$ and $10_{94}$ have the same Jones polynomial, but our
invariant for $x = 2a$ distinguishes them ($4992$ on $10_{41}$, $4896$ on $10_{94}$),
while we get $25$ on both for $x = a$.
We did not manage to find a pair of knots with the same Homfly invariant wich are distinguished by at least one of our
two invariants. On the 11-crossings knots  $11_{280}$ and $11_{439}$ with the same Kauffman polynomials, which are the braid closures of
$112\bar{1}2\bar{3}\bar{2}1\bar{2}\bar{2}3\bar{4}3\bar{4}$ and $\bar{1}\bar{1}22\bar{1}2\bar{1}2\bar{3}2\bar{3}$,
our invariants for $x=a$ are $85$ and $82$, and for $x = 2a$ they are $22176$ and $22048$.
On the other hand, we computed our invariants on Kanenobu's example of 4 distinct knots, presented as the closure of 3-braids,
with the same Homfly and Kauffman polynomial, see \cite{KANENOBU}, example at the end of \S 3; our invariants cannot distinguish them, getting as value
$84$ for $x=a$ and
$734157650613659985$ for $x=2a$. One possibility is therefore that these invariants both depend on the HOMFLY polynomial,
in a way still to be discovered.

{}
$$
\begin{array}{|l|l|c|c|l|}
\hline
Rolfsen   & Braid  & x=a & x=2a & Name \\
\hline
2_1^2 & \bar{1}\bar{1} & 3a & 0 & \mbox{Hopf link} \\
\hline
4_1^2 & 3\bar{2}1\bar{2} \bar{3}\bar{2}\bar{1}\bar{2} & 6a & 16a & \mbox{Solomon's knot} \\
\hline
5_1^2 & \bar{2}1\bar{2}1\bar{2} & 15a & 48a & \mbox{Whitehead link} \\
\hline
6_3^2 & 3\bar{2}1\bar{2}3\bar{2}\bar{1}\bar{2} & 21a & 128a &  \\
\hline
7_1^3 & 3\bar{2}\bar{1}23\bar{2}1\bar{2}3 & 9 & 9 &  \\
\hline
6_2^3 & 2\bar{1}2\bar{1}2\bar{1} & 33 & 225 & \mbox{Borromean link} \\
\hline
6_3^3 & 2\bar{1}21\bar{2}1 & -3 & 21 &  \\
\hline
8_1^4 & 5\bar{4}\bar{3} 21\bar{4} 32 \bar{4} 3 \bar{4} \bar{5} \bar{4} \bar{3}\bar{2}\bar{1}\bar{2}3 \bar{4} & 6 & 61 &  \\
\hline
8_2^4 & 54\bar{3}2143243\bar{4}\bar{5}\bar{4}\bar{3}\bar{2}\bar{1}\bar{2}34 & 0 & -3 & \mbox{Whitehead link} \\
\hline
\end{array}
$$

\end{document}